\documentclass[10pt]{amsart}
\usepackage{amsthm,amsmath,amssymb,amscd,graphics,enumerate, stmaryrd,xspace,verbatim, epic, eepic,url}

\usepackage[usenames,dvipsnames]{color}

\usepackage[colorlinks=true]{hyperref}

\usepackage[active]{srcltx}
\usepackage[all]{xypic}
\SelectTips{cm}{}

{
   \newtheorem{theorem}[subsubsection]{Theorem}
      \newtheorem*{theorem*}{Theorem}
   \newtheorem{proposition}[subsubsection]{Proposition}

   \newtheorem{lemma}[subsubsection]{Lemma}

   \newtheorem{corollary}[subsubsection]{Corollary}
   
   \newtheorem*{conjecture*}{Conjecture}

}
{\theoremstyle{definition}
          \newtheorem*{exercise*}{Exercise}
   
   \newtheorem{example}[subsubsection]{Example}
   \newtheorem*{example*}{Example}
   \newtheorem{definition}[subsubsection]{Definition}
   
   \newtheorem*{definition*}{Definition}
   
   \newtheorem{remark}[subsubsection]{Remark}
   
   \newtheorem{notation}[subsubsection]{Notation}
}
\newcommand{\DD}{{\mathbb{D}}}

\newcommand{\CC}{{\mathbb{C}}}
\newcommand{\QQ}{{\mathbb{Q}}}
\newcommand{\NN}{{\mathbb{N}}}

\newcommand{\PP}{{\mathbb{P}}}
\newcommand{\ZZ}{{\mathbb{Z}}}

\newcommand{\GG}{{\mathbb{G}}}

\renewcommand{\AA}{{\mathbb{A}}}

\def\hatcO{{\widehat\cO}}
\def\hatcJ{{\widehat\cJ}}
\def\ket{_{\rm k\acute et}}

\renewcommand{\cD}{{\mathcal D}}

\newcommand{\cF}{{\mathcal F}}

\renewcommand{\cH}{{\mathcal H}}
\newcommand{\cI}{{\mathcal I}}
\newcommand{\cJ}{{\mathcal J}}

\renewcommand{\cL}{{\mathcal L}}
\newcommand{\cM}{{\mathcal M}}

\newcommand{\cN}{{\mathcal N}}
\newcommand{\cO}{{\mathcal O}}

\newcommand{\cT}{{\mathcal T}}

\newcommand{\cZ}{{\mathcal Z}}

\newcommand{\oNN}{{\overline\NN}}
\newcommand{\oJ}{{\overline{J}}}

\newcommand{\oM}{{\overline{M}}}

\newcommand{\cs}{{\rm cs}}

\def\<{\langle}
\def\>{\rangle}

\newcommand{\Spec}{\operatorname{Spec}}

\newcommand{\Hom}{{\operatorname{Hom}}}

\newcommand{\chara}{\operatorname{char}}

\newcommand{\das}{\dashrightarrow}

\newcommand{\ocM}{\overline{{\mathcal M}}}

\newcommand{\oX}{{\overline{X}}}

\newcommand{\op}{{\overline{p}}}

\newcommand{\inn}{{\operatorname{in}}}

\def\:{{\colon}}
\def\.{{,\dots,}}
\def\dim{{\rm dim}}

\def\inv{{\rm inv}}
\def\Inv{{\rm Inv}}

\newcommand{\double}{\genfrac..{0pt}1
{\raise -1pt\hbox{$\scriptstyle\longrightarrow$}}{\raise 3pt\hbox
{$\scriptstyle\longrightarrow$}}}

\renewcommand{\setminus}{\smallsetminus}

\def\sat{{\rm sat}}
\def\nor{{\rm nor}}
\def\int{{\rm int}}

\def\tototi{\mathbin{\mathop{\otimes}\limits^{\raise-1pt\hbox
{$\scriptscriptstyle {\rm L}$}}}}

\def\indlim{\mathop{\vrule width0pt height7pt depth
4pt\smash{\lim\limits_{\raise 1pt\hbox to 14.5pt
{\rightarrowfill}}}}}
\def\projlim{\mathop{\vrule width0pt height7pt depth
4pt\smash{\lim\limits_{\raise 1pt\hbox to 14.5pt
{\leftarrowfill}}}}}

\newcommand\displaceamount{3pt}

\newcommand{\doubledown}{\ar@<\displaceamount>[d]\ar@<-\displaceamount>[d]}

\newcommand{\doubleup}{\ar@<\displaceamount>[u]\ar@<-\displaceamount>[u]}

\newcommand{\doubleright}{\ar@<\displaceamount>[r]\ar@<-\displaceamount>[r]}

\newcommand{\tor}{{\operatorname{tor}}}
\newcommand{\res}{{\operatorname{res}}}

\newcommand{\rk}{\operatorname{rk}}

\def\into{\hookrightarrow}
\def\onto{\twoheadrightarrow}

\def\gp{\text{gp}}

\newcommand{\logord}{{\operatorname{logord}}}
\newcommand{\Dord}{{\cD\!\operatorname{ord}}}
\newcommand{\Der}{{\cD er}}
\newcommand{\ord}{{\operatorname{ord}}}
\newcommand{\Log}{{\mathbf{Log}}}

\def\supp{{\operatorname{supp}}}

\begin{document}
\title[Principalization]{Principalization of ideals on toroidal orbifolds}

\author[D. Abramovich]{Dan Abramovich}
\address{Department of Mathematics, Box 1917, Brown University,
Providence, RI, 02912, U.S.A}
\email{abrmovic@math.brown.edu}
\author[M. Temkin]{Michael Temkin}
\address{Einstein Institute of Mathematics\\
               The Hebrew University of Jerusalem\\
                Giv'at Ram, Jerusalem, 91904, Israel}
\email{temkin@math.huji.ac.il}

\author[J. W{\l}odarczyk] {Jaros{\l}aw W{\l}odarczyk}
\address{Department of Mathematics, Purdue University\\
150 N. University Street,\\ West Lafayette, IN 47907-2067}
\thanks{This research is supported by  BSF grant 2014365}

\date{\today}


\begin{abstract}
Given an ideal $\cI$ on a variety $X$ with toroidal singularities, we produce a modification $X' \to X$, functorial for toroidal  morphisms, making the ideal monomial
on a toroidal stack $X'$. We do this by adapting the methods of \cite{Wlodarczyk}, discarding steps which become redundant.

We deduce functorial resolution of singularities for varieties with logarithmic structures. This is the first step in our program to apply logarithmic desingularization to a morphism $Z \to B$, aiming to prove functorial semistable reduction theorems.
\end{abstract}
\maketitle
\setcounter{tocdepth}{1}


\section{Introduction}\label{Sec:intro}

\subsection{Declaration of  principles}
This paper opens a large project aiming to develop a new generation of desingularization algorithms, in characteristic 0, that apply to \emph{morphisms} $Z \to B$.  This paper is focused on the case where the base $B$ is a point, namely desingularization of \emph{varieties}.
Much of this introduction serves to introduce the entire project, in particular explaining why we believe a new generation of desingularization algorithms in characteristic 0 is needed.

\subsubsection{Classical principles} Desingularization, or resolution of singularities, of a variety $Z$  is  a proper birational morphism $Z' \to Z$ where $Z'$ is nonsingular. In characteristic 0 its existence   was  proven  by Hironaka \cite{Hironaka}.

As detailed in Section \ref{Sec:summary}, our approach shares, in modified form, a number of principles with the approach of Hironaka's paper and the many that followed it: most classical desingularization algorithms are based on \emph{embedded desingularization}, namely resolving singularities of a subvariety $Z$ embedded in a smooth variety $X$ by successively modifying \emph{admissible centers -} carefully chosen loci in $X$ lying over the singular locus of $Z$.  Embedded desingularization of a subvariety of a smooth variety $X$ is reduced to the problem of \emph{principalization of a coherent ideal sheaf $\cI \subset \cO_X$}, namely producing a modification $X' \to X$ of a  smooth  variety $X$ so that the resulting ideal $\cI \cO_X'$ is the ideal of a normal crossings divisor $D \subset X'$. Principalization of an ideal sheaf is achieved inductively via \emph{order reduction} of the ideal, and
finally one reduces the order of an ideal sheaf using induction on the dimension of the ambient variety $X$ by restricting to a smooth \emph{hypersurface of maximal contact} $H \subset X$.

This classical approach requires working locally and making a number of choices, which need to be reconciled. The paper \cite{Wlodarczyk}, and in a different manner \cite{Bierstone-Milman-funct}, reconcile these choices by producing a desingularization {\em functor}, which is necessarily independent of choices and glues on overlaps of open patches: whenever $X_1 \to X$ is a smooth morphism, the desingularization $X_1' \to X_1$ is the pullback of the desingularization  $X' \to X$. We call this the {\em functoriality principle}.

These classical principles are reviewed and reinterpreted in Section \ref{varsec}.

\subsubsection{Extended functoriality and toroidal orbifolds}\label{Sec:extended-funct}
One fundamental aspect of our approach is a significant strengthening of the functoriality principle. As we explain in Section \ref{morsec}, when keeping desingularization of \emph{morphisms} in mind, it is natural to insist that all operations be compatible with suitable base-change morphisms, namely \emph{logarithmically smooth morphisms} in the sense of \cite{Kato-log}. In classical terms, logarithmically smooth varieties are \emph{toroidal embeddings} \cite{KKMS}, loosely speaking ``varieties which locally look like toric varieties". In characteristic 0, logarithmically smooth morphisms between logarithmically smooth varieties are \emph{toroidal morphisms} \cite{AK}, namely those that locally look like dominant torus-equivariant morphisms of toric varieties.

Accordingly we adopt the \emph{extended functoriality principle: our basic resolution and principalization algorithms must be functorial for logarithmically smooth morphisms.} This includes, in particular, \emph{toroidal blowings up} and \emph{extracting roots of monomials.}

This principle forces us to take a departure from previous algorithms:

\begin{enumerate}
\item We can no longer work with a smooth ambient variety $X$ - \emph{we must allow $X$ to have toroidal singularities.}
\item We cannot use only blowings up of smooth centers as our basic operations. Instead we use a class of modifications, called \emph{Kummer blowings up}, which are stable under logarithmically  smooth base change. These involve taking roots of monomials, in particular:
\item We can no longer work only with varieties $X$ - \emph{we must allow $X$ to be a Deligne--Mumford stack.}
\end{enumerate}

Item (3) is probably the hardest to accept in resolution of singularities. We argue in Section  \ref{examkumsec} that it  arises naturally from the extended functoriality principle. We do restrict the type of stacks we work with to the minimum necessary:

\begin{definition}
Throughout this paper, a \emph{toroidal orbifold} $(X,U)$ is a Deligne--Mumford stack $X$ with finite diagonalizable inertia with a simple toroidal embedding $U\subset X$, where $U$ is a scheme. In other words, it is a simple toroidal orbifold in the sense of \cite{ATW-destackification} with the additional condition that $U$ is a scheme.
\end{definition}

\begin{remark}\label{localy toric}As any orbifold structure can be described in terms of an atlas of compatible charts over the coarse moduli space, one can circumvent the language of Deligne--Mumford stacks.
This will be worked out in detail in \cite{ATW-kummer-varieties}.
\end{remark}

%

\subsection{Statements of main results}

\subsubsection{Logarithmic principalization}
Here is the main result of the paper.

\begin{theorem}[Principalization]\label{Th:principalization}
Let $(X, U)$ be a toroidal orbifold and $\cI$ an ideal sheaf.
There is a sequence  $X_n \to X_{n-1} \to \dots\to X_0 = X$ of  Kummer blowings up on toroidal orbifolds, all supported over the vanishing locus $V(\cI)$, such that $\cI \cO_{X_n}$ is an invertible monomial ideal. The process is functorial for logarithmically smooth  base change morphisms  $Y\to X$, in the sense that the sequence of  Kummer blowings up for $Y$ is the saturated pullback of the sequence for $X$, with trivial blowings up removed.
\end{theorem}

See Section \ref{Sec:order-reduction} for proof. We briefly review the notion of \emph{Kummer blowing up} in Section \ref{Sec:Kummer-review}; details are given in \cite[Section 5.4]{ATW-destackification}. We note that in each stage the toroidal structure is enriched by the exceptional locus, and the modification is not necessarily representable.

\emph{Monomial ideals} on a toroidal orbifold $(X,U)$, namely those which locally correspond to toric ideals, are introduced in Section \ref{Sec:monomial-ideals}. \emph{Invertible} monomial ideals are the correct generalization of ideals of normal crossings divisors.

The \emph{saturated pullback} in the statement is the closure of $U_Y \times_{X} X_n$ in the normalization of the usual pullback $Y\times_X X_n$. It forms the pullback in the category of fine and saturated logarithmic {\em stacks}. Accordingly, in this paper all logarithmic structures are fine and saturated, and all pullbacks are taken in the saturated sense.

The precise meaning of functoriality with  ``trivial blowings up removed" is spelled out in Theorem \ref{Th:order-reduction} below.

\subsubsection{Non-embedded logarithmic desingularization}
Classical ideas reducing resolution to principalization apply in the logarithmic setting. In Section \ref{Sec:nonembedded-proof} we obtain the following:

\begin{theorem}[Non-embedded desingularization]\label{Th:nonembedded}
\begin{enumerate}
\item {\bf Existence:}
Let $Z$ be a logarithmic Deligne--Mumford stack of finite type over $k$ and assume that $Z$ is generically logarithmically smooth and locally equidimensional. Then there is a modification  $\cF(Z)\:Z_{res} \to  Z$  which is an  isomorphism over the logarithmically smooth locus of $Z$ and $Z_{res}$ is logarithmically smooth.
\item {\bf Functoriality:}
The process assigning to $Z$ the modification $\cF(Z)$ is functorial for logarithmically smooth  morphisms:   Given a logarithmically smooth  $Y\to Z$, with $\cF(Y)\:Y_{res} \to  Y$ we have  $Y_{res} = Y \times_Z Z_{res}$, the logarithmic pullback taken in the fs category.
\end{enumerate}
\end{theorem}

See \cite[\S3.1.1]{ATW-destackification} for general references on logarithmic DM stacks, which here are always assumed to carry a fine and saturated logarithmic structure.  A {\em modification} of such stacks is a proper morphism inducing an isomorphism of dense open substacks, see Example \ref{Ex:stack-modification}.

Functoriality implies, in particular, that loci where $Z$ is smooth with normal crossings divisor, or even where $Z$ is toroidal, are untouched. Functoriality for logarithmically smooth  morphisms does not hold in the classical algorithms, and the task of maintaining the normal crossings locus has been considered an important challenge. See \cite[Section~12]{Bierstone-Milman}, \cite[Theorem~1.5]{B-M-except-I}, \cite[Theorem~1.4]{BDSMV} for results on \emph{simple} normal crossings and further discussion.\footnote{We emphasize that pinch points in the divisor do not induce fine and saturated logarithmic structures, hence do not contradict  Theorem \ref{Th:nonembedded}, see also \cite[Appendix A.7]{Temkin-embedded}.}

If one is willing to weaken the extended functoriality requirement one can resolve the toroidal singularities and, using \cite[Theorem 1.2]{Bergh} or \cite[Section 4]{ATW-destackification}, remove the stack structure which arise in Theorem \ref{Th:nonembedded}. To illustrate this we show in Theorem \ref{resolution} how to deduce the simplest form of non-functorial resolution of singularities directly from Theorem \ref{Th:principalization}. In Section \ref{Sec:to-smooth} we apply similar methods with greater care and show how to deduce  resolution of singularities, \emph{functorially for smooth morphisms,} from Theorem \ref{Th:nonembedded}.


\subsection{Technical prerequisites} Our work requires foundations in logarithmic, or toroidal, geometry and in algebraic stacks. This is admittedly a significant requirement. To sweeten the bitter pill,  much of the technique is developed in the companion paper \cite{ATW-destackification} and used here as a ``black box".

Toroidal embeddings are introduced in \cite[Chapter 2]{KKMS} - roughly speaking these are open embeddings $U \subset X$ which \'etale locally look like the embedding of a torus as dense orbit in a toric variety. Toroidal morphisms are introduced in \cite{AK} - these are the morphisms which locally look like dominant torus equivariant morphisms of toric varieties.
We freely use the equivalent language of logarithmically smooth schemes and logarithmically smooth morphisms between them, because it gives us access to a vast arsenal of tools. Basic references for logarithmic geometry are \cite{Kato-log,Kato-toric,Ogus-logbook}.

Toroidal geometry is a natural part of resolution of singularities - it was recognized from the outset that toroidal singularities can be combinatorially resolved. One can ask for a partial resolution procedure reducing an arbitrary singularity to a toroidal one, which our Theorem \ref{Th:nonembedded} provides, though at this point we cannot claim the result is simpler than the classical approach. However there are several points where our procedure enjoys a surprising efficiency, see the discussion in Section \ref{varsec},  the examples in Section \ref{Sec:summary} and Proposition \ref{Absolute}. In particular, the procedure   significantly  simplifies the analysis of  monomial factors of an ideal. Other aspects where logarithmic structures become natural are discussed in Section \ref{varsec}.

Algebraic stacks arise in our work as an outcome of the need to take roots of monomials, see Section \ref{examkumsec}. Basic references for stacks are \cite{Deligne-Mumford,LMB,stacks,Olsson-book}.

As we noted in Remark~\ref{localy toric}(1), one can describe the algorithm on the level of coarse spaces with an additional structure, see \cite{ATW-kummer-varieties}. This allows to use a more elementary language, but requires an extension of the class of toroidal varieties.

The paper \cite{ATW-destackification} also introduces the modifications used in our work, namely Kummer blowings up, reviewed in Section \ref{Sec:Kummer-review}. These are the normalized blowings up of generalized ideals of the form $(x_1,\ldots, x_r, m_1^{1/d},\ldots,m_k^{1/d})$ where $x_i$ are \emph{regular parameters}, $m_j$ are \emph{monomials}, and $d$ is an integer.

\subsection{Functorial desingularization of varieties, reinterpreted}\label{varsec} In this section we review the state of the art of desingularization of varieties in characteristic 0, re-interpreting some of the methods in view of logarithmic geometry, and highlighting similarities and differences with our resolution algorithm.

\subsubsection{Hironaka's principalization}
The general framework and the first construction was suggested by Hironaka in \cite{Hironaka}: one locally embeds a variety $Z$ into a smooth variety $X=X_0$ and deduces desingularization of $Z$ from a principalization of the ideal $I=I_Z\subset\cO_X$ of $Z$. \emph{Principalization} is short for  ``making an ideal sheaf locally principal and monomial". This is a sequence of blowings up $f_i\:X_{i+1}\to X_i$ of subvarieties $V_i \subset X_i$ with $0\le i\le n-1$ such that the following conditions hold:

(1) The centers $V_i$ are smooth and have simple normal crossings with the \emph{boundary divisors} $E_i$, which are iteratively defined by $E_0=\emptyset$  and $E_{i+1}=f_i^{-1}(V_i)\cup f^{-1}(E_i)$. In particular, this implies that each $X_i$ is smooth and each $E_i$ has simple normal crossings.

(2) The pullback $I_n=I\cO_{X_n}$ is an invertible monomial ideal, i.e. it defines a divisor supported on $E_n$.

As indicated above, the smooth blowings up  are replaced in our work by Kummer blowings up.

\subsubsection{The role of the boundary - classical and present viewpoints}
One obvious role of the boundary divisor is that it ``absorbs'' $I$  at the end of principalization. In addition, it introduces canonical local parameters $t_1\.t_s\in\cO_x$ corresponding to the components of $E$ through $x\in X$. These are known as {\em exceptional} parameters, as opposed to the non-exceptional or {\em ordinary} or simply {\em regular} parameters. One profits from this in that possible choices of regular families of parameters at $x$ are reduced to those containing the exceptional ones. The price one has to pay is that the centers must have simple normal crossings with the boundary.

Classically, one also encodes the history of the process in the order of appearance of boundary components, \emph{but we do not}. In classical algorithms the history is  used in the monomial stage of the algorithm, see Section \ref{princsec} below. In our work the monomial stage is immediate and requires no history.

In our work the main role of the boundary divisor $E$ is that it enriches $X$ with the structure of a logarithmically smooth variety $(X,M_X)$, or equivalently a toroidal variety $(X,U)$, where $U=X\setminus E$ and $M_X=\cO_{U}^\times\cap\cO_X$.

The same role appears in a more subtle way when deducing resolution from principalization: at some stage $i\leq n$ the variety $Z_i$ is  resolved, and moreover is transversal to $E_i$, and hence the exceptional divisor $E_i|_{Z_i}$ of $Z_i\to Z$ has simple normal crossings. This shows that, in fact, the algorithm solves a stronger problem: it resolves the logarithmic scheme $Z$ with trivial logarithmic structure into  a toroidal scheme $Z_i$, with structure determined by  the divisor $E_i$.

\begin{remark}\label{excrem}
Divisors also appear in de Jong's method (see Section \ref{morsec}). We claim that, again, their real role is to define  logarithmic structures; this time, these structures make morphisms logarithmically smooth. This point of view is emphasized in \cite[Theorems 2.1, 2.4, 3.5]{Illusie-Temkin}.
\end{remark}

\subsubsection{Algorithms and smooth functoriality}
Hironaka proved the existence of resolution, but later work refined this to canonical  algorithms, depending only on $X$. See, e.g. \cite{Villamayor, Bierstone-Milman, Wlodarczyk, Kollar, Bierstone-Milman-funct}. The descriptions of the algorithms vary, for example, they may or may not use invariants, equivalence of resolution data, tuning of ideals. Nevertheless, a number of them produce the same desingularization, and it seems that these algorithms  differ mainly in issues of eliminating the old boundary divisor.

In fact, not only these algorithms are canonical, but they possess the following stronger \emph{functoriality} property: for any smooth morphism $Y\to Z$ the obtained desingularizations $Y_\res\to Y$ and $Z_\res\to Z$ are compatible, in particular, $Y_\res=Z_\res\times_ZY$. This property was first emphasized in \cite{Wlodarczyk}, both as guiding principle and as proof technique. Since then it plays an important role in desingularization theory and its applications. For example, it allows to extend the desingularization to stacks and formal schemes, see \cite[Section 5]{Temkin-qe}.

\subsubsection{Construction of principalization and associated invariants}\label{princsec}
A primary invariant of the algorithm is the order of an ideal at a point $x\in X$, and the algorithm runs by reducing its maximum $d$ on $X$. Order reduction is done by using only {\em $(I,d)$-admissible} centers $V$, where $\ord(I)=d$ along $V$, and applying the \emph{controlled} or $d$-transform: dividing the pullback of $I$ by the $d$-th power of the exceptional divisor. This principle is in force in this paper, in modified form, see Section \ref{derivsec} below.

In classical algorithms, it is important to separate the order-$d$ locus from the old boundary and for this one uses the secondary invariant $s$ counting the number of old components of $E$ through $x$. Our work  requires neither this separation nor the additional invariant.

The main engine of the algorithm is induction on the dimension $n=\dim(X)$ achieved by restricting onto a {\em hypersurface of maximal contact} $H$ and replacing $I$ by a {\em coefficient ideal} $C_H$. Roughly speaking, if $H=V(t)$ and $I$ is generated by elements $\sum_{i=0}^dc_{ij}t^i$ then $C_H$ is generated by the weighted coefficients $c_{ij}^{d!/(d-i)}$, and order reduction of $(I,d)$ is equivalent to that of $(C_H(I),d!)$. In particular, the main invariant of the classical algorithm is a string $(d_n,s_n,d_{n-1},s_{n-1}\.d_1,s_1)$ ordered lexicographically. Here $d_{n-i}$ is the order of the iterated coefficient ideal after restricting $i$ times to a hypersurface of maximal contact, and $s_i$ is the corresponding invariant counting old boundary components. In our algorithm, this simplifies to a string $(d_n,d_{n-1}\.d_1)$.

We arrive at another point where our algorithm becomes simpler. The above scheme has one complication: the order $d_{n-1}$ of $C_H(I)$ can exceed $d!$. Thus, for induction to work one has to establish order reduction of arbitrary {\em marked ideals} $(I,d)$, whose maximal order may exceed $d$. In this case, the order can increase after blowing up and applying $d$-transform, but this is only caused by the {\em monomial part} $M(I)$ of $I$, which is the maximal invertible monomial ideal dividing $I$. Namely, factoring $I$ as $M(I)N(I)$, one has that the order of $N(I)$ does not increase.

Thus one has to reduce the order of $N(I)$. In classical algorithms this works well as long as $\ord(N(I))\ge d$. For $1\le\ord(N(I))\le d-1$ one has to pass to a {\em companion ideal}, which is a weighted sum of $M(I)$ and $N(I)$, and in the purely monomial case $I=M(I)$ one uses a different purely combinatorial algorithm with an auxiliary invariant of its own.

In this paper one immediately reduces to the case where $N(I)$ is \emph{clean}, not contained in any nontrivial monomial ideal. It then suffices to reduce the order of $N(I)$ below $d$, companion ideals are unnecessary and no special attention is needed for invertible monomial ideals.

\subsubsection{Derivations, logarithmic derivations, and orders}\label{derivsec}
The  derivation ideals $\cD_X^l(I)$ of $I$, where $\cD_X^l$ stands for the sheaf of $l$-th order derivations, play a crucial role in the classical principalization construction as they control almost all critical parts of the algorithm, namely:

(1) They determine the \emph{order} $d=\ord_x(I)$  as the minimal $l>0$ such that $\cD_X^l(I)_x=\cO_x$.

(2) They determine all maximal contact hypersurfaces at $x\in X$ as the zero sets $H=V(t)$, where $t$ is a parameter in $\cD_X^{d-1}(I)_x$.

(3) They determine the coefficient ideal $C(I)=\sum_{l=0}^{d-1}(\cD_X^l(I))^{d!/(d-l)}$ and its restriction $C_H=C(I)_H$.

In many respects, \emph{logarithmic derivations} $\cD^l_{X,E}$ are more convenient to work with. For example, the formulas for their transform under blowings up are simpler, see \cite[Lemma~3.1]{Bierstone-Milman-funct}. Furthermore, in the approach of Biertstone and Milman with equivalence of marked ideals, it is critical to use logarithmic coefficient ideal, see \cite[Remark~3.13]{Bierstone-Milman-funct}.

The only reason why one has to also use ordinary derivations is that they determine the order, while the sheaf of logarithmic derivations $\cD_{X,E}$ does not have this feature. For example, $\cD_{X,E}(I)=I$ for any monomial ideal $I$, in particular the order of a monomial ideal is infinite. This paper views this issue as a virtue, accounting for one of the bigger changes in the algorithm: logarithmic derivatives detect the monomial part of an ideal, which is principalized by a single blowing up. Once this is done, logarithmic derivatives detect the \emph{logarithmic order} of an ideal, and provide the analogue of properties (1), (2) and (3) above.
\begin{remark}
The fact that monomials must have infinite order in our algorithm is dictated by  Extended Functoriality with respect to Kummer coverings. Indeed, if $m$ is a monomial, then the covering $X[m^{1/d}]\to X$ must respect invariants of the algorithm, and hence $\logord(m)=\logord(m^{1/d})$ for any $d$.
\end{remark}

\subsection{Resolution of morphisms}\label{morsec}
The subject of resolution of singularities of morphisms has not been explored as much as resolution of varieties. The classical cases are semistable reduction of schemes  over a one-dimensional base in characteristic 0, see \cite{KKMS}, and semistable reduction of relative curves, see  \cite{dejong-curves}. The general case is  studied in \cite{AK, Karu-semistable, ADK}, \cite{Cutkosky-local, Cutkosky-m32, C-K-mn2, Cutkosky-m, Cutkosky-toroidalization, Ahmadian}, \cite[Section~3]{Illusie-Temkin}, and \cite{Temkin-p-alteration}.

\subsubsection{The problem of resolution of morphisms}
In vague terms, resolution of a morphism $f\:Z\to B$ aims to find an \emph{alteration} $B'\to B$ and a \emph{modification} $Z'\to Z^\mathrm{pr}$ of the proper transform of $Z$ such that the morphism $f'\:Z'\to B'$ is ``as nice as possible'', see \cite[Section 0.1]{AK}. It is a fact of life that one cannot expect $f'$ to be smooth in general, so a suitable weakening of smoothness is required.

Classically, one wanted $f'$ to be ``semistable" in a strong sense, with all fibers locally isomorphic to normal crossings divisors. An example of Karu \cite[Example 2.1.12]{Karu-thesis} shows that this is impossible in general.

\subsubsection{Brief summary of past results}
\paragraph{\it Curves} Semistable reduction when $\dim~B = 1, \dim~Z=2$ in mixed characteristics is classical and can be found in \cite{Deligne-Mumford}. De Jong  \cite{dejong-curves} deduced the case $\dim~Z = \dim~B+1$ with $f$ proper using properness of moduli spaces. The properness assumption was removed in \cite{Temkin-stable}.

\paragraph{\it One-dimensional base} The case where $\dim~B = 1$ in characteristic 0 is treated in \cite{KKMS}. It relies on resolution of singularities: one can assume that the base is a trait $B=\Spec(R)$ with $R$ a discrete valuation ring. In characteristic 0 it  follows from Hironaka's theorem applied to the closed fiber that $f$ can be replaced by a logarithmically smooth morphism.  A subsequent combinatorial simplification leads to semistable reduction.

This method does not readily extend to higher dimensional base, or higher rank valuation rings, as Karu's example demonstrates. In addition it is not compatible with ramified base change, because Hironaka's resolution leads to smooth, rather than logarithmically smooth, schemes.

\paragraph{\it Fibrations by curves} De Jong's method using fibrations by curves directly gives semistable reduction by \emph{alterations} of both base and total space, see \cite{dejong-curves}. In characteristic 0 one can carefully take quotients along the way. This was carried out in \cite{Abramovich-deJong} for varieties. Resolution of morphisms in characteristic 0 is proven in \cite{AK, Karu-semistable, ADK, Illusie-Temkin}. One obtains  modifications $B' \to B$ and $Z' \to Z$ such that $Z' \to B'$ is logarithmically smooth. Base change of combinatorial nature allows one to further improve it.

The main weakness of the fibration method is that it is severely non-canonical.  In particular, the general fiber is modified even if it is smooth.

\paragraph{\it Cutkosky's work}  The work \cite{Cutkosky-local, Cutkosky-m32, C-K-mn2, Cutkosky-m, Cutkosky-toroidalization, Ahmadian} led by Cutkosky aims at gaining further control on the modifications $B' \to B$ and $Z' \to Z$ above. Most of the results obtained in this line of work are about \emph{monomialization}, namely achieving the result \emph{locally at a valuation} of $Z$ and $B$ using only blowings up of nonsingular centers on both $Z$ and $B$. The stronger \emph{toroidalization} result is obtained when $Z$ is a threefold in \cite{Cutkosky-toroidalization, Ahmadian}. Here the modifications   $B' \to B$ and $Z' \to Z$ are sequences of blowings up of nonsingular centers.

\subsubsection{Resolution of morphisms and logarithmic smoothness}

Every single one of the references in Section \ref{morsec} relies on  toroidal, or equivalently logarithmically smooth, formalism. We believe that there is decisive accumulated evidence  for using \emph{logarithmical smoothness} of $f'$, or \emph{toroidality} with respect to an appropriate choice of divisors, as the right replacement of smoothness for desingularization of morphisms.

Here are some  points in favor of logarithmic smoothness:
\begin{enumerate}
\item Logarithmically smooth morphisms have a base-change property analogous to smooth morphisms.
If $Z \to B$ is logarithmically smooth, $B'\to B$  a logarithmic morphism with $B'$ logarithmically smooth, and $Z'\to B'$ the saturated pullback, then the morphism $Z'\to B'$, and hence $Z'$, is also logarithmically smooth.

\item Logarithmic smoothness has further  strong functorial properties similar to smoothness: it is closed under products and compositions.

\item Singularities in logarithmically smooth morphisms are of purely combinatorial nature, so further improvement of such morphisms (when possible) is a combinatorial problem.

\item In characteristic zero, one can transform a morphism $Z\to B$ into a logarithmically smooth morphisms using only modifications $B'\to B$ and $Z'\to Z$, not requiring an alteration.
\end{enumerate}
	
\subsubsection{Extended Smoothness and extended functoriality}\label{smoothprinc}
The following principle is foundational for this paper: once a good extension of smoothness, such as logarithmic smoothness, is found, it is worth working out the methods entirely in the generalized context. If needed and if possible, one can attempt to reduce to the usual notion of smoothness at the very end of a desingularization algorithm. We find this more compelling than to attempt to go back and forth between the two notions of smoothness during the process.

As a corollary of Extended Smoothness, the Functoriality Principle is naturally replaced by the Extended Functoriality Principle, introduced in Section \ref{Sec:extended-funct}.

\tableofcontents

\section{Principles of the principalization functor}
\addtocontents{toc}
{\noindent We give a broad outline of the principalization process, with examples. We prove Theorem \ref{Th:principalization}
using the techniques developed in Sections \ref{Sec:summary}-\ref{Sec:coefficient}.}
\label{Sec:summary}

{Before discussing examples we must introduce \emph{monomial blowings up} and the \emph{cleaning up} process.

\subsection{Monomial ideals}\label{Sec:monomial-ideals}
Let $(X,\cM)$ be a toroidal orbifold, with $\alpha: \cM\to \cO_X$ the logarithmic structure. An ideal sheaf $\cI$ is said to be \emph{monomial} if there is an ideal $J\subseteq \cM$  such that $\cI = \alpha_X(J) \cO_X$. By convention, $0$ is the monomial ideal corresponding to $J=\emptyset$. Note that $J = \alpha^{-1} (\cI)$. The monoid ideal $J$ is uniquely determined by its image $\oJ \subseteq \ocM$ in the characteristic sheaf of monoids.  The notion of monomial ideal coincides with the notion of \emph{canonical ideal sheaf} of \cite[p. 82-83]{KKMS}.

\begin{proposition} \label{Prop:monomial-blowup}

(1) The normalized blowing up $\pi: X' \to X$ of a monomial ideal $\cI$ is a logarithmically smooth morphism.

(2) The stack $X'$ is a toroidal orbifold.

(3) $\cI \cO_{X'}$ is invertible.

(4) If $Y \to X$ is logarithmically smooth and $Y'\to Y$ is the normalized blowing up of $\cI\cO_Y$, then $Y' = Y\times_X X'$, the logarithmic pullback taken in the fs category.
\end{proposition}
See Section \ref{Sec:monomial-blowup}.
We note that even if $X$ is regular, $X'$ is in general singular. With our formalism we do not care, since it is a toroidal orbifold. Hence the price we pay and the advantage we attain.
}

\subsection{Reduction to clean ideals}
This section generalizes \cite{Kollar-toroidal}, allowing for singular toroidal stacks and spelling out functoriality.

An ideal sheaf $\cI$ on a toroidal variety is said to be \emph{clean} if its zero scheme $V(\cI)$ contains no logarithmic strata (generalizing \emph{toroidally resolved} ideals \cite[Definition 3]{Kollar-toroidal}.) This property is compatible with \'etale covers and hence extends to toroidal orbifolds. We show in Theorem \ref{Th:monomial-part} that there is a functorial minimal monomial ideal $\cM(\cI)$ containing $\cI$. (Compare \cite[Theorem 17]{Kollar-toroidal}.)

\begin{proposition}[Cleaning up - see Proposition \ref{Absolute}]\label{Prop:cleaning} Given an ideal $\cI$ on a toroidal orbifold $X$, the normalized blowing up $\pi: X' \to X$ of the monomial ideal $\cM(\cI)$ is a   toroidal orbifold, the morphism $\pi$ is \emph{logarithmically smooth}, and $\cM(\cI\cO_{X'}) = \cM(\cI)\cO_{X'}$. In particular we have a unique functorial factorization $\cI \cO_{X'}= \cI' \cL$, where $\cI'$ is {\em clean} and $\cL = \cM(\cI)\cO_{X'}$ is an invertible monomial ideal sheaf. If $Y \to X$ is logarithmically smooth and $Y'\to Y$ the normalized blowing up of $\cM(\cI \cO_Y)$,  then $Y' = Y \times_X X'$, the logarithmic pullback taken in the fs category.
\end{proposition}
Compare {\cite[Lemma 19]{Kollar-toroidal}}. Given Proposition \ref{Prop:monomial-blowup} and Theorem \ref{Th:monomial-part}, the only part requiring proof is the fact that $\cI'$ is clean.  Note also that $\pi$ blows up all components of $X$ on which $\cI$ vanishes. In particular, $V(\cI')$ is nowhere dense and further principalization is done by modifications of $X'$.

In practice the inductive process will dictate blowing up the Kummer center $\cM(\cI)^{1/a}$ for an appropriate integer $a$. We treat  this  where the terms are defined, see Proposition \ref{Absolute} below.

\begin{example}\label{Ex:monomial}
Here and later we {use} the letters $u,v,w$ for generating monomials and $x,y,z$ for \emph{ordinary parameters}, namely elements of $\cO_{X}$ which reduce to regular parameters on logarithmic strata. In later examples ordinary parameters may turn into exceptional variables, which are monomials.

Consider the ideal $\cI_0= (u-v)$ on $\Spec \CC[u,v]$. Its monomial saturation is $\cM(\cI_0) = (u,v)$. The blowing up of $\cM(\cI_0)$ is the blown up plane. One chart is $X'_u = \Spec \CC[u,v/u]$ where $\cI\cO_{X'_u} = (u(1-v/u))$ with monomial part $\cL = \cM(\cI_0)\cO_{X'_u} = (u)$ and clean part $(1-v/u)$. The other chart is similar.
\end{example}

As noted in \cite{Kollar-toroidal}, this is as far as one can get by blowing up monomial ideals. The remaining task is to principalize clean ideals. For technical reasons it will be convenient to work with the class of slightly more general {\em balanced ideals} of the form $\cM\cdot\cI$, where $\cI$ is clean and $\cM$ is monomial and invertible.

\subsection{Easy example: blowing up a logarithmically admissible center} \label{easyexample}
We start with an easy case: Consider the ideal $\cJ = (u^2,x)$ on  $X = \Spec(\NN \to \CC[\NN][x]) = \Spec \CC[u,x]$, where the logarithmic structure is given by $1\in \NN$ mapping to $u\in \cO_X$.

As we will see in this example, an ordinary parameter may become the equation of an exceptional divisor, and thus monomial, after blowing up.

The ideal $\cJ$ is clean. Blowing up the ideal $\cJ$ itself introduces a modification $X' \to X$ with an exceptional divisor $E$. We use $E$ to enrich the logarithmic structure. There are two affine charts:
\begin{itemize}
\item The $x$-chart $X'_x:=\Spec \CC[x,u,v] / (vx = u^2)$, where $v = u^2  /x$. Here $x=0$ is the equation of the exceptional divisor, which is now a monomial.

The ideal $(x,u^2)$ is transformed into the ideal $(x)$, which is a principal monomial ideal. Note that {$X'_x$ is a \emph{singular} toric variety, in particular} toroidal orbifold.
\item The $u^2$-chart  $X'_{u^2} := \Spec \CC[u,y]$, where $y = x/u^2$.  Here $u=0$ is the equation of the exceptional divisor, and the ideal $(x,u^2)$ is transformed into the monomial ideal $(u^2)$.

Note that  $X'_{u^2}$ is again a {toric variety,} in this case nonsingular.
\end{itemize}
We thus obtained a toroidal orbifold $X'$ on which the ideal $\cJ\cO_{X'}$ is monomial.

Once again we see a glimpse of both the price and the advantage: $X'$ is singular, but it is a toroidal orbifold, so we do not need to desingularize it. The classical principalization procedure would require two blowings up. Were we to start with the ideal $(x,u^{100})$, our procedure would still principalize it in one step, whereas the classical principalization procedure would require 100 blowings up.

\subsection{Analysis of example: logarithmic order and maximal contact}
The same blowing up works for the clean  ideal $\cI_1:=\cJ^2$, as well as its unsaturated variant $\cI_2 := (u^4,x^2)$. But how do we know in all these cases to blow up $\cJ$?

The answer is that, when restricting the ideal $\cJ$ to the hypersurface $\{x=0\}$ we obtain the monomial ideal $(u^2)$, which hints that we should lift this ideal from $\{x=0\}$ to $X$, giving $(u^2,x)$. What distinguishes $x$ here is that it is an element of $\cJ$ which is an \emph{ordinary} parameter. Other ordinary parameters in $\cJ$, such as $x + u^2$, would do just as well.

This works for an ideal  of \emph{logarithmic order 1}, which contains an ordinary parameter, namely when restricting $\cJ$ to the stratum $u=0$ the resulting ideal $(x)$ defines a logarithmically smooth substack, see  \cite[\S5.1.1]{ATW-destackification}. The logarithmic order is defined in Section \ref{Sec:logord} below. Note that $\cI$ is clean if and only if its maximal logarithmic order is finite, see Corollary~\ref{11cor}.

For the clean ideals $\cI_1 = \cJ^2$ or $\cI_2$, which have logarithmic order $2$ on the stratum $u=0$, we would need to mimic the procedures of usual resolution of singularities, and pick an ordinary parameter inside {$\cD^{(\leq 2-1)}(\cI_j) = \cD^{(\leq 1)}(\cI_j)$. Here $\cD^{(\leq i)}$} denotes the sheaf of \emph{logarithmic} differential operators on $X$ of order $\leq i$.

In both cases it is crucial that the ideal $\cJ$ defines an \emph{admissible center} for $\cI_j$, in the sense that $\cI_j \subseteq \cJ^2 = \cJ^{\logord(\cI_j)}$, see Definition~\ref{Def:admissible}. In particular $\cI_j\cO_{X'}$ automatically factors out the invertible monomial ideal $(\cJ\cO_{X'})^2$, see \cite[Theorem~5.4.14]{ATW-destackification}. In the example we consider, $\cI_j\cO_{X'} = (\cJ\cO_{X'})^2$ is already an invertible monomial ideal, whereas in general  \emph{the other factor is automatically clean}, see Proposition \ref{Prop:admissible-blowup}, hence we can continue the process of principalization of a clean ideal.

\subsection{Example: blowing up a Kummer admissible center}\label{examkumsec}
A new phenomenon arises with the clean ideal $\cI_3 = (x^2, u)$, which again has logarithmic order 2.

Since $H = \{x=0\}$ is a hypersurface of maximal contact, and since $\cI_3\cO_H = (u)$ one might be tempted to  blow up $(x,u)$, but \emph{this ideal is not admissible}, as $\cI_3 \not\subseteq  (x,u)^2$; At least in this paper the whole procedure would collapse were we to allow  blowing  up $(x,u)$.

What could an admissible center look like? The fact that $\cI_{3}$ is invariant under the torus action $(x,u) \mapsto (t_1x,t_2u)$ implies that a functorial admissible center must be monomial in $x$ and $u$. Since $x$ is the only ordinary parameter of this form, the center is of the form $(x,u^\beta)$ or $(u^\beta)$. Admissibility implies that $u\in (x^2,xu^\beta,u^{2\beta})$ or $u\in (u^{2\beta})$, yielding $\beta \leq 1/2$. The power of $u$ must be fractional!

We are inevitably led to define the \emph{Kummer ideal sheaf} $(x,u^{1/2})$, which is evidently admissible: ${\cI_3 = (x^2,u)}  \subseteq (x,u^{1/2})^2 = (x^2,xu^{1/2},u)$, as well as its associated modification $X' \to X$. This is carried out in \cite[Sections~5.3--5.5]{ATW-destackification}.

In short, a Kummer ideal sheaf is a coherent ideal sheaf in the fine and saturated Kummer-\'etale topology of $X$, which is determined by an ideal sheaf on some Kummer covering of $X$, in this case the covering $w^2 = u$. To define  Kummer blowings up, we restrict to Kummer ideals of the form $(x_1,\ldots,x_k, m_1,\ldots,m_l)$ generated by ordinary parameters $x_i$ and Kummer monomials $m_j$.  The Kummer blowing up $X' \to X$ is a toroidal orbifold which is necessarily an algebraic stack, and on which the Kummer ideal sheaf is principalized.

Once again there are two affine charts:
\begin{itemize}
\item The $x$-chart $X'_{x}:=\Spec \CC[x,u,v] / (vx^2 = u)$, where $v = u  /x^2$.
Here $x=0$ is the equation of the exceptional divisor, which is monomial.

The ideal $(x^2,u)$ is transformed into the ideal $(x^2)$, which is a principal monomial ideal. Also the Kummer ideal $(x,u^{1/2})$ is transformed into the principal monomial ideal $(x)$. Note that {$X'_x$ is a toroidal orbifold, in fact a smooth} scheme.
\item The $u^{1/2}$-chart is the stack quotient $X'_{u^{1/2}} :=\left[\Spec \CC[w,y] \big/ \mu_2\right]$, where $y = x/w$ and the diagonalizable group $\mu_2 = \{\pm 1\}$ acts via $(w,y) \mapsto (-w,-y)$.

Here $w=0$ is the equation of the exceptional divisor, and the ideal $(x^2,u)$ is transformed into the principal monomial ideal $(u) = (w^2)$. The center is transformed into the principal monomial ideal $(w)$.

Outside $\{y=0\}$ this becomes the schematic quotient $$\Spec \CC[w,y,y^{-1}] / \mu_2 = \Spec \CC[y^2,y^{-2},wy] = \Spec \CC[v^{-1},v,x],$$ which is naturally an open subscheme in $X'_{x}$, allowing gluing of the two charts.

Note that  $X'_{u^{1/2}}$ is again a toroidal orbifold \emph{with respect to the logarithmic structure determined by $E$}, but that the stabilizer of $y=w=0$ \emph{does not act as a subgroup of the torus}. This means that the coarse moduli space is not toroidal in any natural manner, and in order to maintain the toroidal structure the stack structure must remain.
\end{itemize}

Once again, the classical principalization process would require two blowings up, and were we to start with the ideal $(x^{100},u)$, it would require 100 blowings up, compared to one  Kummer blowing up we need. Analogously, the Kummer blowing up of $\cI_3=(x^2,u)$ along  $(x,u^{1/3})$ is admissible, but not optimal, since it does not principalize $\cI_3=(x^2,u)$ in one step.

This example is quite general:

\begin{proposition}[See Proposition \ref{5}]\label{Prop:admissible-blowup} Let $X$ be a toroidal orbifold, $\cI$ a clean ideal with maximal logarithmic order $a$. Let $\cJ$ be a Kummer ideal which is $(\cI,a)$-admissible: $\cI \subseteq (\cJ^a)^\nor$. Let $X'\to X$ be the Kummer blowing up of $\cJ$, with  $\cJ' := \cJ\cO_{X'}$. Then
$\cI \cO_{X'} = \cI' {\cJ'}^a$, where  $\cI'$ is a \emph{clean} ideal with maximal logarithmic order $\leq a$, and ${\cJ'}$ an invertible monomial ideal.
\end{proposition}
%
\subsection{Order reduction} \label{Sec:order-reduction} Proposition \ref{Prop:admissible-blowup}, and the fact that the maximal logarithmic order is an integer $\geq 0$, suggests that we should principalize a clean ideal by finding a sequence of admissible Kummer blowings up which reduces the order. Our inductive process requires restricting to a \emph{hypersurface of maximal contact}, with an ideal which is not necessarily clean. As in the classical algorithm we address this by proving order reduction for an arbitrary \emph{marked ideals} $(\cI,a)$. 

\begin{definition}[Admissible Kummer sequence]\label{Def:admissible-sequence}
Let $X$ be a toroidal orbifold and $(\cI,a)$ a \emph{marked ideal}, i.e. $\cI$ is an ideal and $a$ is a positive integer. An \emph{$(\cI,a)$-admissible Kummer sequence}  $(X_i, \cI_i,\cJ_i)$ consists of
\begin{itemize}
\item a sequence of toroidal orbifolds
\begin{equation}\label{Eq:admissible-Kummer-Sequence} X_n \to X_{n-1} \to \dots \to X_1 \to X_0 = X,\end{equation}
\item ideal sheaves $\cI_i$ for $i\leq n$,
\item  Kummer ideals $\cJ_i$ on $X_i$ for $i<n$,
\end{itemize}
such that for all $i<n$
\begin{itemize}
\item  $\cJ_i$ is $(\cI_i,a)$-admissible, see Definition \ref{Def:admissible}
\item $X_{i+1}$ is the $\cJ_i$-Kummer blowing up of $X_i$, with exceptional divisor $E_{i+1}$,
\item  $\cI_{i}\cO_{X_{i+1}} = \cI_{i+1} \cI_{E_{i+1}}^a$.
\end{itemize}
\end{definition}

\begin{theorem}[Order reduction]\label{Th:order-reduction}

\begin{enumerate} \item {\bf Existence:} Let $X$ be a toroidal orbifold with a marked ideal $(\cI,a)$. Then there is an $(\cI,a)$-admissible Kummer sequence  $(X_i, \cI_i,\cJ_i)$ such that $\cI_n$ has maximal logarithmic order $<a$.

\item{\bf Functoriality:}
The procedure assigning to $(X, \cI,a)$ the sequence $(X_i, \cI_i,\cJ_i)$ is  functorial for logarithmically smooth morphisms: if $X' \to X$ is a logarithmically smooth morphism, with associated sequence $(X'_j, \cI'_j,\cJ'_j), j=1,\ldots,n'$, then there is a strictly increasing function $i(j)$ such that
\begin{itemize}
\item  $X'_j = X'\times_X X_{i(j)}$,
\item $\cI'_j = \cI_{i(j)} \cO_{X'_j}$, and
\item $\cJ'_j = \cJ_{i(j)} \cO_{X'_j}$,
\end{itemize}
while for the remaining $i$, those not in the image of $i(j)$, we have that $\cJ_i \cO_{X_{n'}}$ is the unit ideal.
\end{enumerate}

\end{theorem}

See Section \ref{algsec}. Assuming Order Reduction  we immediately have:
\begin{proof}[Proof of the Principalization Theorem \ref{Th:principalization}]
As in the classical case, we apply the order reduction theorem to the marked ideal $(\cI,1)$, obtaining a sequence $X'\to X$, such that $\cI\cO_{X'}=\cM\cdot\cI'$. Here $\cI'$ is of maximal logarithmic order $<1$, hence  $\cI'=(1)$. The ideal $\cI\cO_{X'}=\cM$ is accumulated from the ideals $\cI_{E}$, split off the pullbacks of $\cI$ during the order reduction process. Hence it is an invertible monomial ideal, as required.
\end{proof}

\emph{It remains to prove Order Reduction.}
\subsection{Homogenization}
The key step is  to construct Order Reduction when $(\cI,a)$ is of maximal order.

A hypersurface of maximal contact $H$ exists locally at a point $p\in X$ when $X$ is a toroidal variety and $\cI$ has logarithmic order precisely $a$. Indeed, $\cD_X^{(\leq a)}(\cI) = (1)$ by Lemma \ref{11new}, hence there exists a local section $x$ of $\cD_X^{(\leq a-1)}(\cI)$ which is an ordinary  parameter at $p$, and $H=\{x=0\}$ is a maximal contact near $p$.

By design, hypersurfaces of maximal contact exist only \'etale-locally and are not unique, so any procedure depending on the choice of hypersurface of maximal contact requires proving that it is independent of choices and can be glued across local patches. To overcome this we follow in Section \ref{Sec:homogenization} the principle introduced in \cite{Wlodarczyk} and define the \emph{homogenization} $\cH(\cI,a)$ of a clean ideal $\cI$ of logarithmic order $a$. This is a clean ideal containing $\cI$ and of logarithmic order $a$, having the following key properties:

\begin{proposition}\label{homprop} Let $(\cI,a)$ be a marked ideal of maximal order on a toroidal orbifold $X$.
\begin{itemize}
\item  {{\rm (See Lemma \ref{Lem:homogenization-equivalent}.)}}
 Any   $(\cI,a)$-admissible Kummer sequence is also an $(\cH(\cI,a),a)$-admissible Kummer sequence, and vice versa.
 \item {{\rm (See Lemma \ref{le: homo}.)}} If $X$ is a toroidal variety, then for any two hypersurfaces of maximal contact $H_1,H_2$ for $\cH(\cI,a)$ at a point $p\in X$ there is an \'etale-local automorphism of $X$ preserving $\cH(\cI,a)$ and taking $H_1$ to $H_2$.
\item  {{\rm (See Lemma \ref{Lem:homogenization-functorial}.)}}
 {If $Y \to X$ is logarithmically smooth, then
 $$\cH(\cI,a) \cO_Y = \cH(\cI \cO_Y,a).$$}\end{itemize}
\end{proposition}
This guarantees that a \emph{functorial} procedure for $\cH(\cI,a)$ is independent of $H_i$, glues across patches, and applies to $\cI$. In the language of \cite[3.53]{Kollar}, the ideal $\cH(\cI,a)$ is MC-invariant. We may now replace $\cI$ by $\cH(\cI,a)$ and assume given a global hypersurface of maximal contact $H$.

\emph{It remains to prove order reduction for $(\cI,a)$ with given hypersurface of maximal contact $H$,  \emph{functorially} with respect to the data $(\cI,H)$.}

\subsection{Coefficient ideals}
Homogenization is not sufficient: by induction on dimension one can reduce the order of $\cI|_H:=\cI\cO_H$, but in general this does not imply that the order of $\cI$ itself is reduced, even in a neighborhood of $H$ (see Example~\ref{notenough}).

To correct this, we follow the principles of \cite{Villamayor,Villamayor-patching,Bierstone-Milman,Wlodarczyk} and introduce the \emph{coefficient ideal $C(\cI,a)$}, a clean ideal of order $a!$. The treatment here follows  \cite[Section 3.6]{Wlodarczyk}, and differs from the earlier cited treatments in that $C(\cI,a)$ is an ideal on $X$ and not on a hypersurface of maximal contact. In Lemma \ref{Lem:coefficient-functorial} we prove that the coefficient ideals are compatible with logarithmically smooth morphisms $Y\to X$, namely $C(\cI,a) \cO_Y = C(\cI \cO_Y,a)$, and hence can be used in functorial constructions.

Next we use the calculus of \emph{marked ideals} (Section \ref{Sec:marked-ideals}) to show that a sequence $X'\to X$ is $(\cI,a)$-admissible if and only if it is $(C(\cI,a),a!)$-admissible (Lemma \ref{Lem:C-equiv}). Finally, we show in Theorem~\ref{Hth} that the same equivalence persists when one restricts $(C(\cI,a),a!)$ onto $H$. This is the main property of the coefficient ideal, which implies the following:

\begin{theorem}\label{Th:reductionthm}  Let $(\cI,a)$ be a marked ideal of maximal order on a toroidal orbifold $X$ and let $i\:H\into X$ be a maximal contact hypersurface. Assume that $\tau\:H'\to H$ is an order reduction of $(C(\cI,a)|_H, a!)$, and let $i_*(\tau)\:X'\to X$ be the sequence obtained by blowing up the same centers (See Section \ref{Sec:pushforwards}). Then $i_*(\tau)$ is an order reduction of $(\cI,a)$.
\end{theorem}

It is this theorem that enables induction on the dimension of $X$ in the construction of order reduction.

\subsection{The algorithm}\label{algsec}
\begin{proof}[Proof of Order Reduction (Theorem \ref{Th:order-reduction})]
We apply induction on $n=\dim(X)$. As dimension is not invariant under logarithmically smooth morphisms we check as we go that the resulting procedure is a functor.

\paragraph{\bf Base case.} If $n=0$ then working locally we can assume that $X$ is a point. If $\cI=(1)$ then the order is already reduced and the algorithm does nothing. If $\cI=0$ then the algorithm blows up $\cM(\cI)=0$ and finishes with $X_1=\emptyset$. Assume now that $n>0$.

\paragraph{\bf Case A: \it The maximal order case.} Assume that $(\cI,a)$ is of maximal order. First, we replace $\cI$ by its homogenization $\cH(\cI,a)$. By Proposition~\ref{homprop} this is a functorial operation which preserves order reduction sequences.

Next, we claim that it suffices to work \'etale-locally on $X$. Indeed, assume $X_0\to X$ is an \'etale covering and $X_1=X_0\times_XX_0$. If we construct a \emph{functorial} order reduction $X'_i\to X_i$ for $\cI_i=\cI\cO_{X_i}$ with $i=0,1$, then they are automatically compatible with respect to both projections $X_1\to X_0$ and hence descend to an order reduction $X'\to X$ of $\cI$. Functoriality can be checked \'etale-locally and hence is preserved under this descent.

Working \'etale-locally we can guarantee existence of a global hypersurface $i\:H\into X$ of maximal contact. Let $C_H=C(\cI,a)|_H$ be the restriction on $H$ of the coefficient ideal. By the induction assumption, $(C_H,a!)$ possesses a functorial order reduction $\tau\:H'\to H$. Applying Theorem \ref{Th:reductionthm}, we define the order reduction $X'\to X$ of $(\cI,a)$ to be the pushforward $i_*(\tau)$ of $\tau$.

Thanks to the homogenization, this construction is independent of the choice of $H$ (Proposition~\ref{homprop}). The construction is functorial, because maximal contacts and coefficient ideals are compatible with logarithmically smooth morphisms.

\paragraph{\bf Case B: \it The general case.} The algorithm runs by reducing the maximal order of the clean part $\cI^{cln}$ of $\cI$, but we must first make sure the clean part is well-defined.

\subparagraph{\sc Step 0:} {\it Initial Cleaning.}
If $\cM(\cI)\neq (1)$ we construct a single, functorial  $(\cI,a)$-admissible \emph{Kummer} blowing up $\sigma_1:X_1\to X$ with the effect that $\cI\cO_{X_1} = \cI_1\cI_E^a$ with $\cI_1$  \emph{clean}.

A naive attempt would be to blow up the monomial ideal $\cM=\cM(\cI)$ from Proposition~\ref{Prop:cleaning}. However, $\cM$ is not $(\cI,a)$-admissible when $a>1$.

Instead we  define $\sigma_1:X_1\to X$ to be the \emph{Kummer} blowing up  along $\cM^{1/a}$, which is always $(\cI,a)$-admissible with  $\cI_1$  clean, see Proposition~\ref{mixed}(1). We note  that this blowing up may introduce a non-trivial orbifold structure even when $X$ is a variety.

The operation is functorial for arbitrary logarithmically smooth morphisms by Proposition~\ref{mixed}(2).



\subparagraph{\sc Step 1:} {\it Reducing the order of the clean part of $\cI$.} While the marked ideal $\cI_1$ is clean, this step produces a sequence $X_r \to X_1$ with ideals $\cI_j$ which are only \emph{balanced}, namely $\cM_j := \cM(\cI_j)$ is invertivle, so that $\cI_j = \cM_j \cdot\cI_j^{cln}$ with $\cI_j^{cln}$ a clean ideal, and finally the maximal logarithmic order of the clean ideal $\cI_r^{cln}$ is $<a$.

We apply induction on the maximal logarithmic order of our ideal $\cI_1=\cI_1^{cln}$. We emphasize that the \emph{maximal} logarithmic order of an ideal is not invariant under arbitrary logarithmically smooth morphisms, as one sees by considering the open embedding of $X \setminus V(\cI)$. As a result, if $Y \to X$ is logarithmically smooth, the procedure on $Y$ is the pullback of the procedure on $X$ \emph{with trivial steps removed}, as required in the theorem.

Let $b$ be the maximal logarithmic order of $\cI_1$. For the base case $b<a$, the order of the ideal is reduced and there is nothing to prove.

Consider the case $b\geq a$. By the maximal order case, Case A above, there is a functorial order reduction for $(\cI_1,b)$. This is a sequence $X_{r_1} \to X_1$ of $(\cI_1,b)$-admissible Kummer blowngs up such that $\cI_1 \cO_{X_j} = \cI_{E_{X_j}}^b \cI^{cln}_j$, with $E_{X_j}$ the exceptional divisor of $X_j \to X_1$, all the ideals $\cI_j^{cln}$ are clean. Finally the maximal logarithmic order of $\cI^{cln}_{r_1}$ is $<b$, hence we may apply induction.

Writing $\cI_j := \cI_{E_{X_j}}^{b-a} \cI^{cln}_j$, we have that $\cI_j$ is \emph{balanced} with clean part $\cI_j^{cln}$. As $b\geq a$ the sequence $X_{r_1} \to X_1$ is automatically $(\cI_1,a)$-admissible, see Corollary~\ref{admissibility reduced}. We have $\cI_1 \cO_{X_j} = \cI_{E_{X_j}}^a \cI_j$, and the maximal logarithmic order of \emph{the clean part} $\cI_j^{cln}$ of $\cI_j$ is $<b$.

By the inductive assumption there is an $(\cI^{cln}_{r_1},a)$-admissible Kummer sequence $X_r \to X_{r_1}$ so that  the maximal logarithmic order of the clean ideal $\cI_r^{cln}$ is $<a$. Thus the combined sequence $X_r \to X$ is $(\cI_{r_1},a)$-admissible and the maximal logarithmic order of the clean part  of $\cI_1 \cO_{X_r}$, that is, the ideal $\cI_r^{cln}$, is $<a$, as this step required.




Functoriality of this step follows from functoriality of the monomial part $\cM(\cI)$ of an ideal, and thus of the clean part $\cI^{cln} = \cI \cdot \cM(\cI)^{-1}$ when the monomial part is invertible.


\paragraph{\sc Step 2:} {\it Final Cleaning.} From the previous step we obtain a sequence $X_r\to X$ such that $\cI_r=\cM_r\cdot\cI_r^{cln}$ and the maximal order of $\cI_r^{cln}$ is at most $a-1$. As in Step 0, we get rid of $\cM_r$ by  defining $X_{r+1}$ to be the Kummer blowing up along $\cM_r^{1/a}$. Then the transform clears off the monomial part: $\cI_{r+1}=\cI_r^{cln}\cO_{X_{r+1}}$. The maximal order of $\cI_{r+1}$ and $\cI^{cln}_r$ are equal by Lemma~\ref{functorlem}(1) since the morphism $X_{r+1}\to X_r$ is logarithmically \'etale. Thus, $X_{r+1}\to X$ provides an order reduction of $(\cI,a)$.

Functoriality is, again, due to Proposition~\ref{mixed}(2).
\end{proof}

\begin{remark}\label{Rem:algorithm}
(1) We preferred to start the induction at $\dim(X)=0$ for clarity, but one could even use the induction base $X = \emptyset$: if $\dim(X)=0$ and $\cI$ is not of reduced order then the Initial Cleaning step resolves it.

(2) Step 0 is non-trivial if and only if $\cM(\cI)\neq (1)$, that is, $\cI$ is not clean. Thus the step is run if and only if the maximal logarithmic order is infinite, and it outputs an ideal of a finite logarithmic order.

(3) After Step 0 a canonical splitting $\cI=\cM\cdot\cI^{cln}$ with an invertible $\cM$ is defined, and it is maintained until the end of the algorithm. Before the Initial Cleaning, such a splitting does not make any sense, unless $\cI$ is clean. Indeed, even when $\cI=\cM$ is monomial it can happen that $\cM$ is not invertible on $X$ but is invertible on a non-empty open $X'\subset X$.

(4) Instead of performing one Final Cleaning in Step 2, one can blow up $\cM(\cI_j)^{1/a}$ after each Kummer blowing of of Step 1.
\end{remark}

\begin{remark} \label{Rem:speeding}
One might hope to make Step 1 more efficient by introducing \emph{weighted} Kummer blowings up, where different ordinary parameters might come with different weights. We postpone pursuing this.
\end{remark}

For the proof of Theorem \ref{Th:nonembedded} we note the following phenomenon, a property of our algorithm, which is manifest in the classical algorithms as well:

\begin{proposition}[Re-embedding principle]\label{Prop:product}
Let $\cI$ be an ideal on a toroidal orbifold $X$ with its principalization sequence $X' = X_n \to \cdots \to X_0 = X$ with centers $\cJ_i$. Let $Y = X \times \Spec k[x_1,\ldots,x_n]$, with $Y\to X$ strict, and with embedding $X \subset Y$ defined by $\cI_X = (x_1,\ldots,x_n)$.  Define inductively
\begin{itemize}
\item $\cJ_i^Y = \cI_{V(\cJ_i)\subset X_i}  \subset \cO_{Y_i}$,
\item $Y_{i+1}=$ the Kummer blowing up of $\cJ_i^Y$, and
\item $X_{i+1} \hookrightarrow Y_{i+1}$ the strict transform of $X_i$.
\end{itemize}
Then the principalization of the ideal $\cI^Y :=\cI\cO_Y+\cI_X$ consists of the resulting  Kummer sequence $Y' = Y_n \to \cdots \to Y_0 = Y$.
\end{proposition}

Indeed, we may choose the first $n$ hypersurfaces of maximal contact for $(\cI^Y,1)$ on $Y$ to be $V(x_i)$, restricting to the ideal $\cI$ on $X$.

\subsubsection{An invariant}
Following tradition, one  can define an invariant of marked ideals that controls the order reduction algorithm. We saw that the algorithm runs by implementing a series of embedded loops that reduce orders $a_i$ of iterated coefficient ideals on iterated maximal contacts $H_i$ of codimension $i$. Loosely speaking, the invariant simply reads off the status of the string $(a_0,a_{1},\dots)$ with the lexicographic order.
\begin{notation}
Let $\Inv$ denote the set of finite strings $(d_0\dots d_n)$, $d_i\in\QQ_{\ge 1}$ for $i<n$ and $d_n\in \QQ_{\ge 0}\cup\{\infty\}$. The empty string ``$()$" is allowed.
We provide $\Inv$ with the lexicographical order, where $(d_0\dots d_n)<(d_0\dots d_n d_{n+1})$
by convention.
\end{notation}

The invariant of $(\cI,a)$ is a function $\inv_{(\cI,a),k_0}\:|X|\to\Inv$. 
The input $k_0=k_0 ({(\cI,a)},p)$ is an auxiliary  integer attached to $p\in |X|$ indicating the minimal codimension where an iterated coefficient ideal might need Initial Cleaning.
One sets the initial values and then proceeds inductively. The initial settings are:
\begin{itemize}
\item $\inv_{(\cI,a),k_0}(p) = ()$ is the empty sequence if $(\cI,a)$ is of reduced order at $p$, i.e. $\logord_p(\cI)<a$.
\item $\inv_{(\cI,a),k_0}(p) =(\infty)$ if $p$ should be blown up at the Initial Cleaning stage, i.e. $k_0 = 0$ and $p$ is a point of the original orbifold $X$ such that $\logord_p(\cI)=\infty$.
\item $\inv_{(\cI,a),k_0}(p) =(\logord_p(\cI^{cln})/a)$ if $p$ should be blown up at the \emph{Final} Cleaning stage, i.e. $\logord_p(\cI)=\infty$ but $k_0>0$ and $\logord_p(\cI^{cln})<a.$
\end{itemize}
In the remaining cases, either $\cI$ is clean or $k_0>0$; however note that whenever $\cI$ is clean at $p$ we necessarily have $\cI^{cln} = \cI$ well-defined and thus $k_0>0$ anyway. Moreover $a \leq \logord_p(\cI^{cln}) <\infty$ since other cases were treated as initial cases above. Let $\logord_p(\cI^{cln}) =b$, with $H$ a maximal contact hypersurface for $(\cI^{cln},b)$ at $p$ and $C_H=C(\cI,b)|_H $.

In  the inductive procedure we set:
\begin{itemize}
\item  $\inv_{(\cI,a),k_0}(p) =(b/a , \inv_{(C_H,b!),k_0-1}(p))$.
\end{itemize}
the concatenated string.
%

\begin{remark}
(1) One can show that the invariant is compatible with logarithmically smooth functions $f\:Y\to X$, i.e.
  $\inv_{(\cI,a),k_0}\circ f=\inv_{(\cI\cO_Y,a),k_0}$. 

(2) If $\inv_p(\cI,a)=(d_0\dots d_n)$ then $n$ is the number of passages to maximal contacts one has to do in order to reduce the order at $p$. Each $d_j$ is the weighted order of the $j$-th restriction $\cI_{j}\subseteq\cO_{H_{j}}$. The order of $\cI_n$ can be infinite, resulting in either $d_n=\infty$ (for Initial Cleaning) or $d_n <1$ (for Final Cleaning).

(3) The Final Cleaning step is the analogue of the monomial step in the classical algorithm. Recall that in the classical algorithm the monomial step is not controlled by the order and forces one to introduce an additional auxiliary invariant involving an ordering of the boundary components. In our case, the step simplifies to a single blowing up, but we still have to use the auxiliary marker $k_0$ not related to the order.

(4)
Note that the equivalent ideals $(\cI,a)$ and $(\cI^n,na)$ have the same order reduction and hence the same invariant. This holds because we  normalized the orders.

(5) in terms of our invariants, a step in the algorithm results in a modification $\sigma$, a transformed ideal, as well as a new integer function $k_0$ which one uses as input for the next step.
\end{remark}

\subsection{From princialization to resolution}\label{Sec:resolution}
Let us first deduce the simplest form of non-functorial resolution of singularities:

\begin{theorem}[{\cite[Main Theorem I]{Hironaka}}]\label{resolution} Let $Z$ be a projective integral variety  over $k$. There is a resolution of singularities $Z' \to Z$ with normal crossings exceptional locus.
\end{theorem}
\begin{proof}
Choosing a very ample line bundle on $Z$ we may assume given an embedding $Z \subseteq X := \PP^N$ with trivial logarithmic structure. Let $\cI  =\cI_Z$ be its ideal. Since the logarithmic structure is trivial it is clean.

By Theorem \ref{Th:principalization} there is a Kummer sequence  $X_n \to X_{n-1} \to \dots \to X_1 \to X_0 = X$ principalizing $\cI$. There is a unique $r$ where $X_{r+1} \to X_r$ blows up the proper transform $Z_r\subseteq X_r$ of $Z$, at which point $Z_r$ is necessarily a toroidal orbifold with respect to the logarithmic structure given by the exceptional locus $E_r$.

By \cite[\S3.4.1 and Theorem~4.1.3(i)]{ATW-destackification} the torification $Z_r^{tor} \to Z_r$ has the property that the coarse moduli space $Z_r^{tor} \to (Z_r^\tor)_{cs}$ is toroidal. Since $Z$ is a scheme the morphism  $Z_r^{tor} \to Z$ factors uniquely through a modification $(Z_r^\tor)_{cs} \to Z$. As $(Z_r^\tor)_{cs}$ is toroidal it has a toroidal desingularization $Z' \to (Z_r^\tor)_{cs}$, and the resulting morphism $Z'\to Z$ is a resolution of singularities. The support of the logarithmic structure is the exceptional locus $E$, which is necessarily normal crossings, since $(Z',E)$ is logarithmically smooth.
\end{proof}

To prove Theorem \ref{Th:nonembedded} more care is needed, {see Section \ref{Sec:nonembedded}}. First, $Z$ comes with a logarithmic structure. In general one cannot globally embed $Z$ in a {\em toroidal stack}. To address this we prove in Lemma \ref{emblem} that $Z$ can be \'etale-locally embedded in a toroidal variety, and in Lemma \ref{emblem2} we show that this is compatible with logarithmically smooth morphisms $Z' \to Z$. Theorem \ref{embth} shows that any two local embeddings in toroidal varieties \emph{of the same dimension} are \'etale-equivalent. Choosing such embeddings, and running our functorial principalization theorem on the ideal $\cI_Z$, one arrives at local modifications $Z'_\alpha \to Z$, where $Z'_\alpha$ are toroidal orbifolds, which can be patched together, giving in particular the existence of a toroidal desingularization $Z' \to Z$.

It remains to show that the desingularization is functorial, for which the only obstacle is the choice of embedding dimension: while one can choose the minimal embedding dimension, this is a global invariant which is not even compatible with restriction to open subsets! This is addressed, similarly to the classical situation, using the Re-embedding Principle, Proposition \ref{Prop:product}. For details see Section \ref{Sec:nonembedded-proof}.

\section{Basic notions}\label{basic}
\addtocontents{toc}{\noindent We recall logarithmic derivations, introduce logarithmic orders, and adapt marked ideals to our setup.}

\subsection{Logarithmic parameters}\label{logparamsec}

\subsubsection{Logarithmic stratification}
A toroidal variety $X$ possesses a natural {\em logarithmic stratification} by locally closed smooth subvarieties $X(i)$ such that $\rk(\oM_p^\gp)=i$ for any $p\in X(i)$. For example, if $X$ is a toric variety then this is the stratification by orbits of codimension $i$. Logarithmic stratifications are preserved by \'etale, in fact, even Kummer logarithmically \'etale, morphisms, hence this construction extends to arbitrary toroidal orbifolds.

\subsubsection{Parameters}
Assume that $X$ is a toroidal variety and $p\in X$ is a point, and let $S=s_p$ be the logarithmic stratum through $p$. By {\em logarithmic parameters} or {\em coordinates} at a point $p$ of $X$ we mean a family $x_1\.x_n\in\cO_{X,p}$ that reduces to a regular family of parameters of $\cO_{S,p}$ and a monoidal chart $u\:\oM_p\into\cO_{X,p}$. We will say that the parameters $x_i$ are {\em ordinary}, while the elements of $u(\oM_p\setminus\{0\})$ will be called {\em monomial parameters} and denoted $u^\alpha=u(\alpha)$.

Similarly, an element $x\in\cO_{X,p}$ is an {\em ordinary parameter} at $p$ if it reduces to a parameter {in $\cO_{S,p}$.} This happens if and only if $V(x)$ is logarithmically smooth at $p$.

\begin{remark}\label{formalrem}
(i) Once parameters are fixed we obtain a formal-local description of $X$ at $p$ via $$\hatcO_{X,p}=k(x)\llbracket\oM_p,x_1\.x_n\rrbracket.$$

(ii) Furthermore, thanks to the characteristic zero assumption, if $f\:Y\to X$ is a morphism of toroidal varieties and $q$ is closed in the fiber $f^{-1}(p)$ then $f$ is logarithmically smooth at $q$ if and only if $\oM_p\into\oM_q$ and the $\hatcO_{X,p}$-algebra $\hatcO_{Y,q}$ is of the form $$k(x)\llbracket\oM_q,x_1\.x_n,y_1\.y_m\rrbracket.$$ In fact, one can take any $y_1\.y_m\in\cO_{Y,q}$ whose images form a basis of $\Omega^1_{s_q/s_p,q}$, where $s_p$ and $s_q$ are the logarithmic strata through $p$ and $q$, respectively. Furthermore, after replacing $Y$ by a strictly \'etale neighborhood of $q$ the composed homomorphism $\oM_p\stackrel{u}\to\cO_{X,p}\to\cO_{Y,q}$ can be extended to a monoidal chart $u'\:\oM_q\to\cO_{Y,q}$, thus extending the coordinates $(x,u)$ at $p$ to coordinates $(x,y,u')$ at $q$.
\end{remark}

\subsection{Blowing up monomial ideals}\label{Sec:monomial-blowup}

\begin{proof}[{Proof of Proposition \ref{Prop:monomial-blowup}}] (1) The property of the morphism $\pi: X' \to X$ being logarithmically smooth is \'etale-local in $X$, hence we may replace $X$ by a representable \'etale cover. Since the blowing up $\pi: X' \to X$ is representable, we may apply  \cite[Theorem 10*, p. 93]{KKMS}, yielding the assertion.

(2) It follows from (1) that $X'$ is a toroidal stack. Since  $\pi: X' \to X$ is representable, the stabilizers of $X'$ embed in the stabilizers of $X$, therefore they are diagonalizable.

(3) The sheaf  $\cI \cO_{X'}$ is invertible by the universal property of blowings up.

(4) By the  universal property of blowings up the morphism $Y'\to X$ factors through $Bl_\cI X$. Since $Y'$ is normal this factors through $X'$, giving a representable morphism $Y' \to  Y\times_X X'$. On the other hand $X'\to X$ is logarithmically smooth and representable hence the fs product $Y\times_X X' \to Y$ is logarithmically smooth and representable. In particular $Y\times_X X'$ is normal. The ideal $\cI\cO_{Y\times_X X'}$ is invertible since it is pulled back from $X'$, giving a representable morphism $Y\times_X X' \to Y'$. Since the morphisms restrict to the identity outside the cosupport of $\cI$, their composition in either order is the identity, giving the claim.
\end{proof}

In fact, we will only blow up saturated monomial ideals. This does not really restrict the generality because of the following observation.

\begin{remark}\label{saturationrem}
(1) Recall that for a monomial ideal $\cI$ corresponding to an ideal $J\subseteq\ocM$, the {\em saturation} $\cI^\sat$ corresponds to the saturated ideal $J^\sat$ consisting of all elements $x\in\ocM$ such that $lx\in l J$ for some $l>0$. Inspecting monoidal charts it is easy to see that the \emph{normalized} blowing up $X'\to X$ of a monomial ideal $\cI$ on $X$ coincides with the \emph{normalized} blowing up of its saturation $\cI^\sat$, and, in addition, $\cI\cO_{X'}=\cI^\sat\cO_{X'}$.

(2) We will not need this fact, but note for the sake of completeness that by \cite[Lemma 5.1.3]{AT1} the normalized blowing up of a monomial ideal $\cI$ is in fact the usual blowing up of an ideal $(\cI^n)^\sat$ for a large enough $n$.
\end{remark}

\subsection{Logarithmic differential operators}\label{3.1}
\subsubsection{{Logarithmic derivations}}
Let $X=(X,D)$ be a toroidal $k$-variety. In the sheaf $\Der_k(X, \cO_X)$ of all $k$-derivations $\cO_X\to\cO_X$ consider the subsheaf $$\Der_k((X,D), \cO_X) = \Der_{\Log}(X, \cO_X)$$ of {\em logarithmic derivations}, i.e. derivations that take the ideal defining the toroidal divisor $D$ to itself. In other words, these are derivations preserving $D$, or tangent vectors on $X$ tangent to $D$. We call $\Der_k((X,D), \cO_X)$ the {\em logarithmic tangent sheaf} of $(X,D)$.

Furthermore, if $f\:Y\to X$ is an \'etale strict morphism then the isomorphism $\Der_k(Y,\cO_Y)=f^*\Der_k(X,\cO_X)$ induces an isomorphism $$\Der_k((Y,f^{-1}(D)),\cO_Y) = f^*\Der_k((X,D),\cO_X).$$ Therefore, the formation of logarithmic tangent sheaves extends to toroidal orbifolds. Lemma \ref{Lem:toroidal-logetale} below shows the stronger statement that logarithmic derivations commute with \emph{logarithmically} \'etale morphisms, a fact used in \cite{ATW-kummer-varieties}.

The following remark will not be used, so the reader can skip it.

\begin{remark}\label{Logrem}
The above ad hoc definition agrees with the more general theory of logarithmic schemes. First, the sheaf of logarithmic derivations $\Der_k((X,\cM), \cO_X)$ with values in $\cO_X$ can be naturally defined for any logarithmic scheme $(X,\cM)$, see \cite[IV Definition 1.1.1]{Ogus-logbook}. Second, $\Der_k((X,\cM), \cO_X) = \Der_{\Log}(X, \cO_X)$ is the relative tangent sheaf of $X$ over Olsson's stack $\Log$, giving an interpretation of logarithmic derivations in terms of ordinary ones. Furthermore, the latter definition makes sense even when $X$ is a logarithmic stack.
\end{remark}

For brevity, we denote the logarithmic tangent sheaf by $\cD^1_X$.
\subsubsection{Local description {of logarithmic derivations}}

\begin{lemma}\label{logderiv}
Assume that $X$ is a toroidal variety, $p\in X$ is a closed point, and $x_1\.x_n\in\cO_{X,p}$ and $u\:\oM_p\into\cO_{X,p}$ are logarithmic parameters at $p$. Then,

(1) For each $1\le i\le n$ there exists a unique element $\frac{\partial}{\partial x_i}\in\cD^1_{X,p}$ vanishing on $u(\oM_p)$ and satisfying $\frac{\partial}{\partial x_i}x_j=\delta_{ij}$.

(2) For each element $L$ of $N_p:=\Hom_\ZZ(\oM^\gp_p,\cO_{X,p})$ there exists a unique derivation ${\DD}_L\in\cD^1_{X,p}$ such that ${\DD}_L(u^\alpha)=L(\alpha)u^\alpha$ for any monomial coordinate and ${\DD}_L(x_i)=0$ for $1\le i\le n$.

(3) The construction of (2) provides an embedding ${\DD}\:N_p\into\cD^1_{X,p}$ and then
$$\cD^1_{X,p}={\DD}(N_p)\oplus\left(\oplus_{i=1}^n\cO_{X,p}\frac{\partial}{\partial x_i}\right).$$

(4) For any $\partial\in\cD^1_{X,_p}$ and a monomial $m\in M_p$ we have that $\partial m\in m\cO_{X,_p}$. In particular, logarithmic derivations preserve monomial ideals, including the ideals of logarithmic strata.
\end{lemma}
\begin{proof}
Locally at $p$ the coordinates induce a logarithmically smooth morphism to $X_0=\Spec(k[x_1\.x_n][\oM_p])$. This reduces the claim to the case of $X_0$, which can be done by a direct inspection. In particular, in this case $\cD^1_{X,p}$ is freely generated by $\frac{\partial}{\partial x_1},\ldots,\frac{\partial}{\partial x_n}$ and $u_1\frac{\partial}{\partial u_1},\ldots,u_l\frac{\partial}{\partial u_l}$ where $u_i$ form a basis for $\oM_p^\gp$. (In fact, since ${\rm char}(k)=0$ it suffices to take $u_i$ that form a basis of $\oM_p^\gp\otimes\QQ$.)
\end{proof}

\subsubsection{Basic properties} It follows from Lemma \ref{logderiv}(4) that without choices one can say about $\cD^1_{X,p}$ the following:
\label{logderlem}
\begin{enumerate}
\item A logarithmic $k$-derivation of $\cO_{X,p}$ preserves the ideal of the logarithmic stratum $S=s_p$ through $p$ and hence restricts to a $k$-derivation of $\cO_{S,p}$. This provides a surjective $\cO_{S,p}$-homomorphism $\cD^1_{X,p}\otimes_{\cO_{X,p}}\cO_{S,p}\to\cD^1_{S,p}$, which does not possess, however, a natural lift to a homomorphism of $\cO_{X,p}$-modules $\cD^1_{X,p}\to \oplus_i\cO_{X,p}\frac{\partial}{\partial x_i}$.

\item For any $\partial\in\cD^1_{X,p}$ and {any} monomial $m\in M_p$ the element $\frac{\partial m}m\in\cO_{X,p}$ is uniquely defined. By Leibnitz rule, sending $m$ to $\frac{\partial m}m$ one obtains a homomorphism of monoids $M_p\to(\cO_{X,p},+)$, whose extension to $M_p^\gp$ will be denoted $\phi_\partial$. In particular, a homomorphism $\phi\:\cD^1_{X,p}\to\Hom_\ZZ(M^\gp,\cO_{X,p})$ arises. So, any monoidal chart $u\:\oM_p\into M_p$ induces by composing $\phi$ with the dual of $u$ an epimorphism $\cD^1_{X,p}\to N_p$. To split the latter one also has to fix regular parameters at $p$.
\end{enumerate}

\begin{lemma}\label{tangentsheaf}
Assume that $X$ is a toroidal orbifold. Then the logarithmic tangent sheaf $\cD_X^1$ is locally free of rank $\dim(X)$.
\end{lemma}
\begin{proof}
The assertion is compatible with strict \'etale base change, reducing to the case of scheme, which follows from Lemma~\ref{logderiv}(3).
\end{proof}

\begin{remark}
The lemma also follows from the fact that  $\cD^1_X$ is dual to the logarithmic cotangent sheaf $\Omega^{\rm log}_{(X,D)/k}$, which is locally free of rank $\dim(X)$ by results of Kato.
\end{remark}

\subsubsection{{Higher order differential operators}}
Since $\chara k =0$, a nontrivial monomial ideal sheaf $\cI$ is stable under $\cD^1_X$, namely $\cD^1_X(\cI) = \cI$. This does not hold for the unit ideal sheaf (e.g. take $X=\Spec k[M]$), and to by-pass this inconvenience we will work with the sheaf $\cD_X^{(\leq 1)}$, which does stabilize every monomial ideal sheaf. It can be defined as the subsheaf of the total sheaf of differential operators generated as an $\cO_X$-module by $\cO_X $ and the logarithmic tangent sheaf $\cD_X^1$. We will also make use of the subsheaf $\cD_X^{(\leq n)}$ generated by the images of $(\cD_X^1)^{\otimes i}$ for $0\le i\le n$. In particular, we will consider the quasi-coherent algebra of logarithmic differential operators $\cD_X^{\infty}=\cup_i\cD_X^{(\leq i)}$.

\begin{remark}
As in Remark \ref{Logrem}, one can interpret $\cD_X^{(\leq i)}$ as follows: it is the sheaf of relative differential operators of $X/\Log$ of order at most $i$. Note that we use here that the characteristic is zero and hence the algebra of differential operators is generated by operators of order 1.
\end{remark}

\subsubsection{Functoriality}

\begin{lemma}\label{Lem:toroidal-logetale}
If $f:Y \to X$ is a logarithmically smooth morphism of toroidal orbifolds, then the natural homomorphisms $\cD_Y^{(\leq i)}\to f^*(\cD_X^{(\leq i)})$ and  $\cD_Y^\infty \to f^*(\cD_X^\infty)$ are surjections and the kernels act {on $f^{-1}(\cO_X)$ by 0.} Furthermore, if $f$ is logarithmically \'etale, for example, a toroidal modification or  an \'etale morphism, then all these maps are isomorphisms.
\end{lemma}
\begin{proof}
We can work \'etale-locally on $X$ and $Y$, so by Remark~\ref{formalrem} we can choose compatible coordinates $x_1\.x_n$ and $u\:\oM_p\to\cO_{X,p}$ at $p\in X$ and $x_1\.x_n,y_1\.y_m$ and $u'\:\oM_q\to\cO_{Y,q}$ at $q\in f^{-1}(p)$. Fix a basis $u_1\.u_r$ of $\oM_p^\gp$ and let $u_1\.u_r,u'_1\.u'_s$ be its extension to a basis of $\oM^\gp_q\otimes\QQ$. By Lemma~\ref{logderiv}, $\frac{\partial}{\partial x_1},\ldots,\frac{\partial}{\partial x_n},u_1\frac{\partial}{\partial u_1},\ldots,u_r\frac{\partial}{\partial u_r}$ form a basis of $\cD_{X,p}^1$ and $$\frac{\partial}{\partial x_1},\ldots,\frac{\partial}{\partial x_n},\frac{\partial}{\partial y_1},\ldots,\frac{\partial}{\partial y_m},u_1\frac{\partial}{\partial u_1},\ldots,u_r\frac{\partial}{\partial u_r},u'_1\frac{\partial}{\partial u'_1},\ldots,u'_s\frac{\partial}{\partial u'_s}$$ form a basis of $\cD_{Y,q}^1$. In addition, it is clear that $\frac{\partial}{\partial y_i}$ and $u'_j\frac{\partial}{\partial u'_j}$ vanish on $f^*\cO_{X,p}$, hence $f^*(\frac{\partial}{\partial x_i})=\frac{\partial}{\partial x_i}$ and $f^*(u_j\frac{\partial}{\partial u_j})=u_j\frac{\partial}{\partial u_j}$. Finally, note that if the morphism is logarithmically \'etale then $m=s=0$. The assertion for $i=1$ follows.

The case of $i=1$ implies that the maps {$\cD_Y^{(\leq i)} \to f^*(\cD_X^{(\leq i)})$} are surjective, and, since the kernel is the two-sided ideal generated by $\frac{\partial}{\partial y_i}$ and $u'_j\frac{\partial}{\partial u'_j}$, it vanishes on $f^*(\cO_{X,p})$. This proves the assertion for any finite $i$, and hence also for $\infty$.
\end{proof}

\begin{corollary}\label{Cor:logsmooth}
If $f\:Y\to X$ is a logarithmically smooth morphism of toroidal orbifolds and $\cI$ is an ideal on $X$ then $(\cD_X^{(\leq i)}\cI)\cO_Y=\cD_Y^{(\leq i)}(\cI\cO_Y)$ and $(\cD_X^{\infty}\cI)\cO_Y=\cD_Y^{\infty}(\cI\cO_Y)$.
\end{corollary}
\begin{proof}
The ideals $(\cD_X^{(\leq i)}\cI)\cO_Y$ and $\cD_Y^{(\leq i)}(\cI\cO_Y)$ have the same sets of generators.
\end{proof}

Deriving ideals on toroidal orbifolds is compatible with restriction onto the logarithmic strata. Namely, {\S \ref{logderlem}(1)} immediately implies the following result for varieties, and the general case follows:

\begin{lemma}\label{diffonstrata}
Assume that $X$ is a toroidal orbifold, $S\into X$ is a logarithmic stratum, and $\cI$ is an ideal on $X$. Then for any $i\ge 0$ we have that $\cD_X^{(\leq i)}(\cI)|_{S}=\cD_S^{(\le i)}(\cI|_{S})$.
\end{lemma}

\subsection{The monomial saturation of an ideal}\label{monomsatursec}
This section generalizes the discussion of \cite[Proposition 15 and Theorem 17]{Kollar-toroidal}.

\begin{definition}
Let $\cI$ be an ideal sheaf on a toroidal orbifold. Define the \emph{monomial saturation of $\cI$} to be
$$\cM(\cI) := \bigcap_{\substack{\cN\supseteq \cI\\ \cN\text{ monomial}}}\cN.$$
\end{definition}

Clearly $\cM(\cI)$ is a monomial ideal containing $\cI$, and $\cI$ is monomial if and only if $\cM(\cI) = \cI$.

\begin{theorem}\label{Th:monomial-part} Let $\cI$ be an ideal sheaf on a toroidal orbifold $X$.
	\begin{enumerate}
		\item $\cI$ is monomial if and only if $\cD_X^{(\leq 1)} \cI = \cI$.
		\item $\cD_X^\infty \cI = \cM(\cI)$
		\item If $Y \to X$ is logarithmically smooth then $\cM( \cI \cO_Y) = \cM(\cI) \cO_Y$.
	\end{enumerate}
\end{theorem}
\begin{proof}
(1) We assume $\cD_X^{(\leq 1)} \cI = \cI$ and need to show that $\cI$ is monomial, the other direction being trivial.

Since the notion of monomial ideal is stable under \'etale base change it suffices to verify the equality on the strict henselization of $X$ at every geometric point. Let $p$ be a geometric point of $X$ and $\widehat X$ the completion at $p$. Since the completion of a local ring is faithfully flat it suffices to check the statement on  $\widehat X$; indeed if $J \subseteq \cM$ generates $\widehat\cI$ then it generates $\cI$. We may write $\widehat X = \Spec \widehat R$ where  $R = \bar k[x_1,\ldots,x_n][M]$, and $M$ is a sharp saturated monoid.

Write $M^+$ for the maximal ideal of $M$ and $P$ for the resulting ideal in $\widehat R$. For every $N>0$ we have $\cD_X^{(\leq 1)} P^N = P^N$, so we consider the induced action of the operators $\cD_X^{(\leq 1)}$ on  $\cI_N:= {(\widehat\cI + P^N) / P^N} \subseteq \widehat R/P^N$.

The operators $$1, u_1\frac{\partial}{\partial u_1},\ldots,u_l\frac{\partial}{\partial u_l}$$ commute and have distinct systems of eigenvalues on the subgroups $u \, \bar k[[x_1,\ldots,x_n]]$, for distinct  $u \in M\smallsetminus N(M^+)$. Hence $$\cI_N = \bigoplus_{u\in M\smallsetminus N(M^+)} \cI_N \cap u \, \bar k[[x_1,\ldots,x_n]] $$ splits as a direct sum.

Suppose $0 \neq u\,f(x_1,\ldots,x_n) \in \cI_N \cap u \, \bar k[[x_1,\ldots,x_n]] $  and $x_1^{r_1}\cdots x_n^{r_n} = \inn(f)$ is the initial term of $f$ with respect to some monomial order. As  $\chara k=0$, we have
$$\frac{\partial^{r_1}}{\partial x_1^{r_1}}\cdots\frac{\partial^{r_n}}{\partial x_n^{r_n}}(uf) = u \mu \in \cI_N$$ for some unit $\mu \in \bar k[[x_1,\ldots,x_n]]$,  hence $u \in \cI_N.$ It follows that $\cI_N$ is monomial, so $\widehat \cI$ is monomial, as required.

For the proof of part (3) below we note that this implies that $\widehat\cM(\cI)$ is described as follows: write the generators of $\widehat\cI$ as $f_j = \sum c_{j,\alpha}(x) u^\alpha$. Then $\widehat\cM(\cI) = (u^\alpha)_{c_{j,\alpha}(x) \neq 0}$.

(2) Since $\cD_X^{(\leq 1)} \cD_X^\infty \cI = \cD_X^\infty \cI$, the ideal $\cD_X^\infty \cI$  is a monomial ideal by (1). Since $\cI\subseteq\cD_X^\infty \cI$ we have $\cM(\cI) \subseteq \cM(\cD_X^\infty \cI) = \cD_X^\infty \cI.$ On the other hand $\cD_X^\infty \cI \subseteq\cD_X^\infty \cM(\cI) = \cM(\cI)$, giving the equality.

(3) We may pass to completions,  where we may write $\widehat X = \Spec \widehat R$ with  $R = \bar k[x_1,\ldots,x_n][M]$ and  $\widehat Y = \Spec \widehat S$ with  $S = \bar k[x_1,\ldots,x_n,x_{n+1},\ldots,x_m][P]$ and with $M^{gp} \to P^{gp}$ injective. Writing $f_j = \sum c_{j,\alpha}(x) u^\alpha$ for the generators of $\widehat \cI$ we have as in (1) above, {we obtain that the ideals} $\widehat{\cM(\cI) \cO_Y}= (u^\alpha)_{c_{j,\alpha}(x) \neq 0}\cO_Y$ and $\widehat{\cM(\cI \cO_Y)}= (u^\alpha)_{c_{j,\alpha}(x) \neq 0}$ have the same generators, as required.
\end{proof}

\subsection{Balanced ideals and  clean ideals} \label{Sec:logord}

\begin{definition}\label{Def:classify-ideals}
A nowhere zero ideal $\cI$ on a toroidal orbifold is
\begin{itemize}
\item {\it balanced} if the monomial ideal $\cM(\cI)$ is invertible.

\item {\it clean} if $\cM(\cI)=1$.
\end{itemize}
Given a balanced ideal $\cI$ we define its {\em clean part} $$\cI^{cln}:=(\cM(\cI))^{-1}\cI.$$
\end{definition}
In particular, $\cI$ factors as $\cI^{cln}\cdot \cM(\cI)$ and this is compatible with differentiation by the following lemma:

\begin{lemma} \label{commute}
Let $X$ be a toroidal orbifold, $\cN$ a monomial ideal and $\cI$ an arbitrary ideal. Then
$$\cD_X^{(\leq i)}(\cN\cdot\cI)=\cN\cdot\cD_X^{(\leq i)}(\cI).$$
\end{lemma}
\begin{proof}
Working \'etale-locally we can assume that $\cD_X^{(\leq 1)}(\cN \cI)$ is generated by elements $uf$ and $\nabla(uf)$, and $\cN\cD_X^{(\leq 1)}(\cI)$ is generated by elements $uf$ and {$u\nabla(f)$, where $\nabla$ is a logarithmic derivation, $u$ is a monomial in $\cN$, and $f$ is an element of $\cI$. Since $\nabla(u)\in(u)$ and $\nabla(uf)=u\nabla(f)+\nabla(u)f$, we obtain that $\nabla(uf)-u\nabla(f)\in(uf)$} and hence $\cD_X^{(\leq 1)}(\cN \cI)=\cN\cD_X^{(\leq 1)}(\cI)$. This implies that $$\cD_X^{(\leq i)}(\cN\cdot\cI)= \cD_X^{(\leq i-1)}\left(\cD_X^{(\leq 1)}(\cN\cdot\cI)\right)=\cD_X^{(\leq i-1)}\left(\cN\cdot\cD_X^{(\leq 1)} \cI\right),$$ and induction on $i$ gives the result.
\end{proof}

\subsection{The logarithmic order of an ideal} \label{Sec:logord}

\subsubsection{The order}
Note that the usual order of an ideal $\cI$ on a smooth variety $X$ (e.g., \cite[\S2.1]{Wlodarczyk}) is compatible with smooth morphisms. Therefore, the order descends to arbitrary DM (even Artin) stacks $X$: an ideal $\cI\subseteq\cO_X$ induces a map $\ord_\cI$ from the underlying topological space $|X|$ to $\oNN:=\NN\cup\{\infty\}$. The value at $p\in|X|$ will be denoted $\ord_p(\cI)$.

Let $\cI$ be an ideal on a toroidal orbifold $X$. By the {\it  logarithmic order} of $\cI$ at a point $p\in|X|$ we mean $$\logord_p(\cI)=\ord_p(\cI|_{s_p})$$ where $s_p$ is the logarithmic stratum through $p$. This defines a function $\logord_\cI\:|X|\to\oNN$, which coincides with $\ord_\cI$ when the logarithmic structure is trivial. Note that $\logord_p$ is a monomial order where we put infinite weights on monomial parameters and weight 1 on ordinary parameters. An element $x$ is an ordinary parameter at $p$ if and only if $\logord_p(x)=1$.

\subsubsection{Relation to differentials}
The following result indicates that $\logord_p$ is a logarithmic analogue of the classical order on smooth varieties.

\begin{lemma}\label{11new}
Let $X$ be a toroidal variety, $\cI$ an ideal, and $p\in X$ a point. Then $$\logord_p(\cI)=\min\{a\in\NN\mid {\cD_X^{(\leq a)}(\cI)}_p=\cO_{X,p}\},$$ where $\min(\emptyset)=\infty$ by convention.
\end{lemma}
\begin{proof}
Thanks to Lemma \ref{diffonstrata}, the statements reduces to the differential characterization of the classical order on the logarithmic strata.
\end{proof}

\subsubsection{The infinite locus}
The classical order is infinite at $p$ if and only if $\cI_p=0$. The logarithmic analogue is more interesting:

\begin{lemma}\label{11inf}
Let $X$ be a toroidal orbifold, $\cI$ an ideal, and $p\in|X|$. Then $\logord_p(\cI)=\infty$ if and only if $p\in V(\cM(\cI))$.
\end{lemma}
\begin{proof}
Working \'etale-locally we can assume that $X$ is a variety. Then the lemma follows from the observation that $\cI|_{s(p)} = 0$ if and only if $\cM(\cI)_p\neq\cO_{X,p}$.
\end{proof}

\begin{corollary}\label{11cor}
An ideal $\cI$ is clean if and only if its maximal logarithmic order is finite.
\end{corollary}

\subsubsection{Functoriality}
Finally, we should check that the logarithmic order fits the Extended Functoriality Principle.

\begin{lemma}\label{ordfunctor}
Assume that $f\:X'\to X$ is a logarithmically smooth morphism between toroidal varieties, $\cI$ is an ideal on $X$, and $\cI'=\cI\cO_{X'}$. Then $$\logord_{\cI'}=\logord_\cI\circ f.$$
\end{lemma}
\begin{proof}
Note that $f$ induces smooth morphisms between logarithmic strata. Therefore, the claim follows from compatibility of the classical order with smooth morphisms.
\end{proof}

\subsubsection{Differential logarithmic order}
This subsection will not be used in the sequel, so the reader can skip it. On can also define the {\it differential logarithmic order} at a point $p\in X$ by $$\Dord_p(\cI)=\min\{a\in\NN\mid {\cD_X^{(\leq a)}(\cI)}_p={\cD_X^\infty(\cI)_p}\}.$$
By Lemma~\ref{11new} it coincides with $\logord_p$ when the latter is finite, and it contains a finer information when $\logord_p$ is infinite. For example, by Lemma~\ref{commute}, $\Dord_p(\cM\cI)=\Dord_p(\cI)$ for any monomial ideal $\cM$. In addition, under monomial blowings up $\Dord$ can only drop, giving an additional control on the Initial Cleaning step. Probably, $\Dord_p$ can be useful for studying the complexity of our algorithm.

As another side of the same coin, $\Dord$ is not compatible with logarithmically \'etale morphisms, see an example below. So, it is not consistent with the Extended Functoriality and we prefer not to use it in the paper.

\begin{example}
Consider $\cI=(u+v)^3 \subset k[u,v]$ and the transformation $ k[u,v] \subset  k[u,w]$ given by $v=uw$, {with exceptional $(u)\subset k[u,w]$  and} $\cI' = u^3(1+w)^3$. Taking $p=(u,v)$ and $p' = (u,w)$ we have $\Dord_p(\cI)> 0$ whereas $\Dord_{p'}(\cI') = 0$ since $\cI'_{p'} = (u^3)$ is monomial.
\end{example}

\subsection{Marked ideals and their maximal contact}\label{Sec:marked-ideals}
Inspired by Hironaka \cite{Hironaka-idealistic}, and following Villamayor \cite{Villamayor} and Bierstone--Milman \cite{Bierstone-Milman-simple}, it is convenient to consider {\emph{marked ideals $(\cI,a)$,}} that is coherent ideals $\cI$ on toroidal varieties $X$  with associated integer weights $a$ which keep track of the singular locus of an ideal. In Section \ref{Sec:controlled-transform} below the marking $a$ dictates the transformation of a marked ideal by an admissible blowing up.

\begin{definition}
Given a marked ideal $(\cI,a)$ on a toroidal orbifold $X$ we define a closed subset $\supp(X,a)\subset |X|$ called the {\it logarithmic cosupport} or simply {\it cosupport} of $(\cI,a)$ by $$\supp(\cI,a)=\{p\in X\mid \logord_p(\cI)\geq a\}.$$
\end{definition}

\begin{definition}
A marked ideal $(\cI,a)$ on a toroidal orbifold $X$ is {\it of maximal order} if $\logord_p(\cI)\le a$ for any $p\in |X|$. This happens if and only if $\logord_p(\cI)=a$ for any $p\in\supp(\cI,a)$, and we allow the case when $\supp(\cI,a)=\emptyset$.
\end{definition}

As an example, if $\cI$ is an invertible ideal then the marked ideal $(\cI,1)$ is of maximal order if and only if $V(\cI)$ is a toroidal suborbifold.

\begin{lemma}\label{Lem:supports2} Let $\cJ\subseteq \cI$ be two ideals on a toroidal orbifold $X$.

(1) $(\cJ,a)$ is of maximal order if and only if $\cD_X^{(\leq a)}(\cJ)=1$.

(2) If $(\cJ,a)$ is of maximal order then $(\cI,a)$ is of maximal order.
\end{lemma}
\begin{proof}
The first claim holds by Lemma \ref{11new}, and the second claim follows.
\end{proof}

The following lemma is the analogue of \cite[Lemma 3.2.2]{Wlodarczyk} and \cite[Lemma 3.3.3]{Wlodarczyk}, which in turn refer to \cite{Villamayor}:

\begin{lemma}\label{Lem:supports}
Let $(\cI,a)$ be a marked ideal on a toroidal orbifold $X$. Then

\begin{enumerate}
\item $\supp(\cD_X^{(\leq i)}(\cI),a-i)=\supp(\cI,a)$ for any $i<a$.
\item {$\supp(\cI,a)=V(\cD_X^{(\leq a-1)}(\cI))$} is closed.
\item If  $(\cI,a)$ is of maximal order then for $i<a$ the marked ideal $(\cD_X^{(\leq i)}(\cI),a-i)$  is also of maximal order.
\end{enumerate}
\end{lemma}
\begin{proof}
It suffices to consider the case of varieties. Thanks to Lemma \ref{diffonstrata}, the statements can be checked in the structure sheaves of strata, where the result follows from the non-logarithmic situation.
\end{proof}

To simplify notation we will use in the sequel the following

\begin{definition}\label{Def:derivative-marked}
Given a marked ideal $(\cI,a)$ and $i\ge 0$ we set $$\cD^{(\leq i)}(\cI,a) = (\cD^{(\leq i)}(\cI), a-i).$$
\end{definition}

\begin{definition}
Let $(\cI,a)$ be of maximal order on a toroidal orbifold $X$. A \emph{hypersurface of maximal contact} to $(\cI,a)$ is a closed toroidal suborbifold $H\into X$ of pure codimension 1 such that its ideal $\cI_H$ is contained in $\cT(\cI,a):=  \cD_X^{(\leq a-1)}(\cI)$.
\end{definition}

Hypersurfaces of maximal contact exist \'etale-locally. Loosely speaking, they are given by ordinary parameters contained in $\cT(\cI,a)$. For example, if $X$ is a toroidal \emph{variety} and $p\in X$ is a point then $\cT(\cI,a)_p$ contains elements $x$ such that $\cD_X^1(x)_p$ contains a unit. Therefore $x$ is an ordinary parameter at $p$ and $H=V(x)$ is a maximal contact locally at $p$.

Note also that Lemma \ref{Lem:supports}(2) implies that $\supp(\cI,a)\subseteq H$. This fact can be strengthened in terms of admissibility, see Corollary~\ref{Hcor} below.

Finally, Lemma~\ref{ordfunctor} implies that various notions we have introduced in this section are compatible with logarithmically smooth morphisms:

\begin{lemma}\label{functorlem}
Assume that $f\:X'\to X$ is a logarithmically smooth morphism between toroidal orbifolds, $\cI$ is an ideal on $X$ with $\cI'=\cI\cO_{X'}$, and $a>0$. Then

(1) $\supp(\cI',a)=f^{-1}(\supp(\cI,a))$.

(2) If $f$ is surjective then $(\cI',a)$ is of maximal order if and only if $(\cI,a)$ is of maximal order.

(2) If $H$ is a hypersurface of maximal contact to $(\cI,a)$ then $H\times_XX'$ is a hypersurface of maximal contact to $(\cI',a)$.
\end{lemma}

\section{Admissibility of Kummer centers}
\addtocontents{toc}{\noindent Blowing up the monomial saturation makes an ideal balanced. We study the effect of an admissible modification on a clean ideal.}
We study the effect of blowing up Kummer centers (reviewed in Section \ref{Sec:Kummer-review}) admissible for a marked ideal (Section \ref{Sec:admissible-marked}). The special case of a monomial Kummer center  is discussed in Section \ref{Sec:monomial-saturation}.

\subsection{{Kummer monomials, Kummer ideals and Kummer blowings up}}\label{Sec:Kummer-review}
We briefly recall results and constructions of \cite[Sections 5.3--5.5]{ATW-destackification}, designed specifically to principalize clean ideals. The goal is to find a rigorous way to use roots of monomials in definitions of ideals and blowings up.

\subsubsection{The Kummer topology} The basic idea is very simple: refine the Zariski topology of a toroidal variety $X$ to the coarsest topology where elements of the form $m^{1/d}$ are defined locally. This naturally leads to the fs-Kummer-\'etale topology $X\ket$ on a toroidal orbifold $X$, and for brevity we will call it the {\em Kummer topology} of $X$. Objects of this topology are Kummer logarithmically \'etale morphisms, and coverings are surjective such morphisms, see \cite[\S5.3.5]{ATW-destackification}. Recall that this topology is generated by \'etale covers and Kummer \'etale covers of the form $X[m_1^{1/d}\.m_r^{1/d}]$, see \cite[Lemma~5.3.6]{ATW-destackification}. Working with elements of the form $m^{1/d}$ is very convenient on $X\ket$: each time such an element shows up we actually work on an $X[m^{1/d}]$-scheme $X'$ of $X\ket$. This $X'$ can be thought of as a local chart on which $m^{1/d}$ is defined.

\subsubsection{Kummer ideals} The presheaf $\cO_{X\ket}$ is in fact a sheaf of rings and we call its finitely generated ideals {\em Kummer ideals} of $X$. Kummer-locally these are usual coherent ideals. For example, if $y_1\.y_l$ are global functions on $X$ and $m_1\.m_r$ are global monomials on $X$ then $(y_1\.y_l,m_1^{1/d}\.m_r^{1/d})$ is a Kummer ideal on $X$ which becomes an ordinary globally generated ideal once restricted to the chart $X[m_1^{1/d}\.m_r^{1/d}]$.

\subsubsection{Kummer sheaves of logarithmic differential operators} Associating to objects $X'$ of $X\ket$ the sheaf $\cD^{(\le i)}_{X'}$ one obtains presheaves $\cD^{(\le i)}_{X\ket}$. It follows from Lemmas~\ref{tangentsheaf} and \ref{Lem:toroidal-logetale} that {each} $\cD^{(\le i)}_{X\ket}$ is a locally free sheaf. Moreover, if $X$ is a toroidal variety then $\cD^{(\le i)}_{X\ket}$ is locally free already for the Zariski topology of $X$. For any Kummer ideal $\cI$, applying $\cD^{(\le i)}_{X\ket}$ as a presheaf and sheafifying one obtains a Kummer ideal $\cD^{(\le i)}_{X\ket}(\cI)$.

\subsubsection{Ordinary ideals} Any ideal $\cI\subseteq\cO_X$ defines a Kummer ideal $\cI\ket$ on $X$ by assigning $\cI_{X'}:=\cI\cO_{X'}$ to objects $X'$ of $X\ket$. If a Kummer ideal is of the form $\cI\ket$ we will call it \emph{an ordinary ideal.} Moreover, for brevity we will write $\cI$ and $\cD^{(\le i)}_X$ instead of $\cI\ket$ and $\cD^{(\le i)}_{X\ket}$. This is safe since all operations agree: $(\cI+\cJ)\ket=\cI\ket+\cJ\ket$, $(\cI\cJ)\ket=\cI\ket\cJ\ket$, and $\cD^{(\le i)}_X(\cI)\ket=\cD^{(\le i)}_{X\ket}(\cI\ket)$.

\subsubsection{Monomial Kummer ideals}
A Kummer ideal $\cJ$ on a toroidal orbifold $(X,\cM_X)$ is called {\em monomial} if Kummer-locally it is of the form $(m^{1/d}_1\.m_r^{1/d})$. Such ideals provide a natural extension of monomial ideals. In particular, they correspond to ideals $J\subseteq\frac{1}{d}\cM_X$ for some $d$. It follows that each monomial Kummer ideal is of the form $\cN^{1/d}$, where $\cN$ is an ordinary monomial ideal.

\begin{definition}\label{Def:Kummer}
(1) A {\it Kummer center} on a toroidal orbifold $X$ is a Kummer ideal of the form $\cJ=\cI_H+\cN^{1/d}$, where $\cI_H$ is the ideal of a toroidal suborbifold $H$ and $\cN$ is a monomial ideal.

(2) The \emph{vanishing locus} of $\cJ$ is the locus where it is non-trivial, i.e. $V(\cJ)=H\cap V(\cN)$.
\end{definition}

The monomial Kummer ideal $\cN^{1/d}$ is determined by $\cJ$ as $$\cN^{1/d}=\alpha_{X\ket}(\alpha_{X\ket}^{-1}(\cJ))\cO_{X\ket}.$$ There might be many choices of $H$, for example, $(x,m^{1/2})=(x+m,m^{1/2})$. Note also that Kummer centers are the Kummer ideals that \'etale-locally are of the form $$(x_1,\ldots,x_k,m_1,\ldots,m_r),$$ where $m_i$ are Kummer monomials and $x_1\.x_k$ is a part of a regular family of parameters.


\subsubsection{Kummer blowings up}
Recall that the blowing up $X'\to X$ along Kummer center $\cJ$ is defined in \cite[Section 5.4]{ATW-destackification} as follows: locally on $X'$ there exists a Kummer $G$-covering $Y\to X$ such that $\cI=\cJ_Y$ is an ordinary ideal. {The logarithmic structure on $Y$ is associated to the open set $U_Y$, the preimage of $U_X$. Consider $Y'=Bl_{\cI}$, with its exceptional locus $E_Y$. Writing $U^0_{Y'}$ for the preimage of $U_Y$, we endow $Y'$ with the logarithmic structure associated to $U_{Y'}:=U^0_{Y'} \setminus E_Y$.} Then $Y'=Bl_{\cI}\to Y$ is a $G$-equivariant morphism of toroidal varieties.

The stack quotient $[Y'/G]$ is a toroidal orbifold and $E_Y$ descends to a Cartier divisor on $[Y'/G]$, thereby giving rise to a morphism $\phi\:[Y'/G]\to B\GG_m$. We define $X'=[Bl_\cJ(X)]$ to be the relative coarse space $[Y'/G]_{\cs/B\GG_m}$ of $[Y'/G]$. Note that $E_Y$ descends to a Cartier divisor $E$ on $X'$ because $\phi$ factors through $X'$. Therefore, $\cJ\cO_{X'\ket}$ is an invertible (ordinary) ideal with vanishing locus $E$. Naturally, we call $E$ the {\em exceptional divisor} of the Kummer blowing up $X'\to X$ and denote its ideal $\cI_E=f^{-1}(\cJ)$, where $f\:X'\to X$ is the Kummer blowing up.

It turns out that though $X'\to X$ does not have to be representable, it satisfies the usual universal property, and, in addition, $X'$ acquires a natural structure of a toroidal orbifold. Here is the summary of main properties of Kummer blowings up, see \cite[Theorems~5.4.5, 5.4.14 and Lemma~5.4.18]{ATW-destackification} for details:

\begin{theorem}\label{Kummertheorem}
Let $X=(X,U)$ be a toroidal orbifold with a permissible Kummer center $\cJ$. Consider the Kummer blowing up $f\:X'=[Bl_\cJ(X)]\to X$ and let $E$ be the exceptional divisor. Then

(1) $f$ is a modification inducing an isomorphism $X'\setminus E=X\setminus V(\cJ)$.

(2) $(X',f^{-1}(U)\setminus E)$ is a toroidal orbifold.

(3) $f$ is the universal morphism of toroidal orbifolds $h\:Z\to X$ such that $h^{-1}(\cJ)$ is an invertible monomial ideal.

(4) $f$ is compatible with logarithmically smooth morphisms $h\:Y\to X$. Namely, $[Bl_{h^{-1}\cJ}(Y)]=X'\times_XY$.
\end{theorem}

\subsubsection{Kummer sequences, strict transforms and compatibility}
A \emph{Kummer sequence} is a sequence $\sigma\:X_n\to X_0=X$ of Kummer blowings up.

Assume that $X$ is a toroidal orbifold, $i\:H\into X$ is a closed toroidal suborbifold, and $\sigma_0\:X'\to X$ is a Kummer blowing up at center $\cJ$. The {\em strict transform} of $H$ is the closure $H'$ of $H\setminus V(\cJ)$ in $X'$. For a Kummer sequence $\sigma\:X_n\to X_0=X$ we obtain a sequence $\sigma|_H\:H_n\to H_0=H$ of (iterated) strict transforms, where $H_{i+1}$ is the strict transform of $H_i\into X_i$.

Kummer blowings up were defined in \cite{ATW-destackification} only for Kummer centers, so the following restriction naturally arises. We say that $\cJ$ and $\sigma_0$ are {\em $H$-compatible} if the restriction $\cJ|_H:=\cJ\cO_{H\ket}$ is a Kummer center on $H$. We say that a Kummer sequence $\sigma$ is {\em $H$-compatible} if each $X_{i+1}\to X_i$ is $H_i$-compatible.

It is proved in \cite[Lemma~5.4.16]{ATW-destackification} that if $\sigma_0$ is $H$-compatible then, similarly to classical blowings up, $H'$ is the Kummer blowing up of $H$ at $\cJ|_H$. This immediately implies a similar result for Kummer sequences: if $\sigma$ is $H$-compatible then $\sigma|_H\:H_n\to\dots\to H_0=H$ is a  Kummer sequence whose centers $\cJ_i|_{H_i}$ are restrictions of the centers $\cJ_i$ of $\sigma$.

\subsubsection{Pushforwards} See \cite[3.30.3]{Kollar} \label{Sec:pushforwards}.
Conversely, for any Kummer center $\cI$ on $H$ its preimage under $\cO_{X\ket}\onto \cO_{H\ket}$ is a Kummer center on $X$ that will be denoted $i_*(\cI)$. Indeed, this can be checked \'etale-locally on $X$, hence one can choose coordinates $(x_1\.x_n,M)$ so that $H=(x_1\.x_l)$ and $\cI=(x_{l+1}\.x_{l+r},m_1^{1/d}\.m_s^{1/d})$, and then one has that $i_*(\cI)=(x_1\.x_{l+r},m_1^{1/d}\.m_s^{1/d})$.

If $\tau\:H'\to H$ is the Kummer blowing up at $\cI$ then its {\em pushforward} $i_*(\tau)\:X'\to X$ to $X$ is defined to be the Kummer blowing up along $i_*(\cI)$. Since $H'$ is the strict transform of $H$, we have that $H'\into X'$ and hence Kummer blowings up of $H'$ can be pushed forward to $X'$, etc. By induction on length, for any Kummer sequence $\tau\:H_n\to\dots \to H_0=H$ we obtain an $H$-compatible pushforward sequence $i_*(\tau)\:X_n\to\dots\to X_0=X$, such that $i_*(\tau)|_H=\tau$ and the centers of $i_*(\tau)$ are pushforwards of the centers of $\tau$.

\subsubsection{$H$-admissibility}
We say that a Kummer sequence $\sigma\:X'\to X$ is {\em $H$-admissible} if all centers are supported on the strict transforms of $H$. This happens if and only if $\sigma=i_*(\sigma|_H)$ and is more restrictive than $H$-compatibility.

\subsection{Logarithmic derivations on a Kummer blowing up}
The differential $d\sigma: \cD_{X'}^i \to \sigma^* (\cD_X^i)$ is a homomorphism of locally free sheaves isomorphic over an open set, so we suppress its notation and view $\cD_{X'}^i$ as a subsheaf of $\sigma^* (\cD_X^i)$. We need to consider  the annihilator of the quotient.

\begin{lemma}\label{33} Let $X$ be a toroidal orbifold and let $\sigma: X'\to X$ be the Kummer blowing up at center $\cJ$. Then
sections of  $\cI_E\cdot \sigma^*(\cD_X^1)$ form logarithmic derivations on $X'$, that is
$\cI_E\cdot\sigma^*(\cD_X^1)\subseteq \cD_{X'}^1$. In general $\cI^i_E\cdot\sigma^*(\cD^{(\leq i)}_X)\subseteq \cD^{(\leq i)}_{X'}$.
\end{lemma}
\begin{proof}
We claim that the assertion is local with respect to the Kummer topology. Indeed, if $h\:Y\to X$ is a Kummer logarithmically \'etale covering then $\tau\:Y'=Y\times_XX'\to Y$ is the Kummer blowing up with center $\cJ_Y:=\cJ\cO_Y$, and $h'\:Y'\to X'$ is logarithmically \'etale. In particular, $h^*(\cD^{(\leq i)}_X)=\cD^{(\leq i)}_Y$ and $h'^*(\cD^{(\leq i)}_{X'})=\cD^{(\leq i)}_{Y'}$ by Corollary~\ref{Cor:logsmooth}. Hence it suffices to check the inclusion after pulling back to $Y'$, where it becomes $(\cI^i_E\cO_{Y'})\cdot\tau^*(\cD^{(\leq i)}_Y)\subseteq \cD^{(\leq i)}_{Y'}$. In addition, $\cI^i_E\cO_{Y'}=\cJ^i\cO_{Y'}=\cJ_Y^i\cO_{Y'}=\cI^i_{D}$, where $D$ is the exceptional divisor of $\tau$. Thus, it suffices to prove the lemma for $\tau$. In particular, taking $h$ fine enough we can assume that $\cJ$ is an ordinary ideal.

As in the classical case, the lemma can be reduced to a direct inspection based on the chain rule. We provide an alternative argument in the spirit of \cite{Schwartz}. Since we are working in characteristic 0 it suffices to consider the case $i=1$. Fixing $p\in X$ we may replace $X$ by its localization $X_p$.

Choose ordinary parameters $x_1,\ldots,x_n$ and a monoidal chart $u\:M=\oM_p\into\cO_{X}$ at $p\in X$ so that $\cJ = (x_1,\ldots,x_r,m_1,\ldots,m_s)$. If $r=0$ then $\sigma$ is toroidal, hence $\sigma^*(\cD_X^1)= \cD_{X'}^1$ and the result holds. Consider the case $r>0$. Write $\eta_j \in E$ for the generic points of the exceptional divisor $E$ of $\sigma$.

We denote $D_i = V(x_i)$ and $D_0=\sum_{i=1}^r D_i = V(x_1 \cdots x_r).$ Let $X_0$ be the  logarithmic scheme with underlying scheme $X$ defined by the smaller open subset $U_X \setminus D_0$, and let $X'_0$ be the resulting logarithmic structure on $X'$. By \cite[Lemma~5.4.20]{ATW-destackification}, the morphism $X'_0 \to X_0$ is toroidal, in particular logarithmically \'etale. Thus $\cD^1_{X'_0} = \sigma^*\cD^1_{X_0}$.

We  note that $\eta_j$ does not lie in the proper transform of any $D_i$, so $(\cD^1_{X'_0})_{\eta_j} = (\cD^1_{X'})_{\eta_j}$. Also it lies inside the $x_i$-chart, where  $\cI_E = (\sigma^*x_i)$. Hence for every $i,j$ we have $(\cI_E)_{\eta_j} = (\sigma^*x_i)_{\eta_j}$.

We have injective homomorphisms of locally free sheaves
$$\cD^1_{X'_0} \hookrightarrow \cD^1_{X'} \hookrightarrow\sigma^* \cD^1_X,$$
and our task is to show that $\cI_E$ annihilates $(\sigma^* \cD^1_X)/\cD^1_{X'}$. Since this is the torsion quotient of locally free sheaves of the same rank, its annihilation can be tested at the generic points   $\eta_j\in E$ of its support $E$, where $(\sigma^*\cD^1_{X_0})_{\eta_j} = (\cD^1_{X'_0})_{\eta_j} = (\cD^1_{X'})_{\eta_j}$.

At $p$ the quotient $\cD^1_X/\cD^1_{X_0}$ is generated by the elements $\frac{\partial}{\partial x_i} + \cD^1_{X_0}$, each annihilated by the corresponding $x_i$ since $x_i \frac{\partial}{\partial x_i} \in \cD^1_{X_0}$. It follows that $\left((\sigma^* \cD^1_X)/\cD^1_{X'}\right)_{\eta_j}=\left((\sigma^* \cD^1_X)/\sigma^* (\cD^1_{X_0})\right)_{\eta_j}$ is generated by the pullback elements $\nabla_i$, each  annihilated by $(\sigma^* x_i)_{\eta_j} = (\cI_E)_{\eta_j}$, as needed.
\end{proof}
\begin{corollary}\label{4}
Keeping assumptions of Lemma~\ref{33} and with $\cI$  an ideal sheaf,
$$\cI^i_E\cdot(\cD^{(\leq i)}_X(\cI)\cO_{X'})\subseteq \cD^{(\leq i)}_{X'}(\cI\cO_{X'}).$$
\end{corollary}
\begin{proof}
We have that $$\cI^i_E\cdot\cD^{(\leq i)}_X(\cI)\cO_{X'}=\cI^i_E\cdot\sigma^*(\cD^{(\leq i)}_X)(\cI\cO_{X'})\subseteq \cD^{(\leq i)}_{X'}(\cI\cO_{X'})$$ with the equality being obvious and the inclusion due to Lemma~\ref{33}.
\end{proof}

Lemma \ref{33} suggests the following notion of \emph{controlled transforms} of  sheaves of derivations:

\begin{notation}\label{Not:controlled}
For any $\cO_X$-submodule $\cF \subseteq  \cD^1_{X}$ we set
$$\sigma^c(\cF):=\cI_E(\sigma^*(\cF)) \subseteq \cD^1_{X'}.$$
Furthermore, for any $i>0$ we denote by $\cF^{(\le i)}$ the $\cO_X$-submodule of $\cD_X^{(\le i)}$ generated by the images of $\cF^{\otimes j}$ with $0\le j\le i$.
The most important particular cases are when $\cF=\cD_X^1$ or $\cF$ is generated by a single derivation.
\end{notation}

\subsection{Admissible modifications for marked ideals}\label{Sec:admissible-marked}

\subsubsection{Integral closure of ideals}
Recall that the {\em integral closure $I^\nor$} of an ideal $I$ in a ring $A$ consists of all elements $x\in A$ satisfying a weighted integral equation
\begin{equation}\label{inteq}
x^n+a_1x^{n-1}+\dots+a_n=0,\ \ \ a_i\in I^i.
\end{equation}
For example, $(I^n+J^n)^\nor$ contains the ideals $I^iJ^{n-i}$ for $0\le i\le n$. We use the notation $I^\nor$, indicating ``normalization", rather than $I^{\operatorname{int}}$, to avoid confusion with ``integralization" of logarithmic geometry.

For a Kummer ideal $\cJ$ on a toroidal orbifold $X$ we define the integral closure $\cJ^\nor$ by taking integral closure of ideals $\cJ(U)$ in rings $\cO_X(U)$ and sheafifying. (In fact, the sheafification is not needed but we do not explore this.) The ideal $\cJ^\nor$ on $X\ket$ need not be finitely generated, but this will not be an issue for us.

\begin{lemma}\label{Jnorlem}
Assume that $X$ is a toroidal orbifold, $\cJ$ is a Kummer center on $X$, and $a\ge 0$.

(1) Fix a presentation $\cJ=\cI+\cN$, where $\cN$ is a monomial Kummer ideal and $\cI$ is the ideal of a toroidal suborbifold. Then $$(\cJ^a)^\nor=\sum_{j=0}^a(\cN^j)^\sat\cdot\cI^{a-j}.$$

(2) For any $i$ with $0\le i\le a$
$$\cD_X^{(\le i)}((\cJ^a)^\nor)\subseteq(\cJ^{a-i})^\nor.$$
\end{lemma}
\begin{proof}
The claims are Kummer-local, hence we can assume that $\cN$ is a usual monomial ideal and $\cI=(x_1\.x_n)$ for a family of ordinary parameters $x_1\.x_l$ at $p\in X$. Notice that (2) follows from (1) and the observation that $\cD_X^{(\le i)}$ preserves the monomial ideals $(\cN^j)^\sat$ and takes $\cI^b$ to $\cI^{b-i}$. To prove (1) we increase the toroidal structure by $V(x_1\dots x_n)$, obtaining a new characteristic monoid $\oM^{new}_p=\oM_p\oplus\NN^n$ for which $\cI$ and $\cJ$ are both monomial. A simple computation
 with monomials shows that $\cJ^\sat$ is the sum on the right hand side of the asserted equality, and it remains to recall that $\cJ^\sat=\cJ^\nor$ by \cite[Corollary~5.3.6]{AT1}.
\end{proof}

\subsubsection{Admissibility}
\begin{definition}\label{Def:admissible}
Let $(\cI,a)$ be a marked ideal on a toroidal orbifold $X$ and $\sigma\:X'\to X$ a Kummer blowing up with a center $\cJ$. If $\cI\subseteq(\cJ^a)^\nor$ we say that $\cJ$ is {\it admissible} for the marked ideal $(\cI,a)$ and $\sigma$ is an {\it $(\cI,a)$-admissible Kummer blowing up}.
\end{definition}

Observe that this notion of admissibility is equivalent to Hironaka's original definition \cite[I.2, Definition 6]{Hironaka} in the case of smooth varieties with the trivial logarithmic structure: the admissibility means that $V(\cJ)$ is regular and the order of $\cI$ along $V(\cJ)$ is at least $a$. Note that in the classical case, $\cJ^a$ is integrally closed. However in general, even when $\cJ$ is a saturated monomial ideal, $\cJ^a$ does not have to be saturated. The key property of admissible blowings up is part (3) of the following lemma.

\begin{lemma}\label{3}
Assume $(\cI,a)$ is a marked ideal on a toroidal orbifold $X$, $\sigma: X'\to X$ is the $(\cI,a)$-admissible Kummer blowing up at $\cJ$, and $\cI_E=\sigma^{-1}(\cJ)$. Then
\begin{enumerate}
\item $V(\cJ)\subseteq \supp(\cI,a)$.
\item $\sigma^{-1}((\cJ^a)^\nor)=\cI^a_E$.
\item $\cI\cO_{X'}\subseteq \cI^a_E$.
\end{enumerate}
\end{lemma}
\begin{proof}
(1) By Lemma \ref{Jnorlem}
$$\cD_X^{(\leq a-1)}(\cI)\subseteq \cD_X^{(\leq a-1)}((\cJ^a)^\nor)\subseteq \cJ^\nor.$$ Therefore on the level of supports we get: $$V(\cJ)=V(\cJ^\nor)\subseteq V(\cD_X^{(\le a-1)}(\cI))=\supp(\cI,a).$$

(2) We compare $\cI^a_E=\sigma^{-1}(\cJ^a)\subseteq\sigma^{-1}((\cJ^a)^\nor)\subseteq(\cI^a_E)^\nor=\cI^a_E.$

(3) By (2) we have $\cI\cO_{X'}\subseteq \sigma^{-1}((\cJ^a)^\nor)=\cI^a_E.$
\end{proof}

\subsubsection{The controlled transform of a marked ideal}\label{Sec:controlled-transform}
Lemma \ref{3}(3) enables us to make the following definition.

\begin{definition}\label{Def:controlled-transform}
By the {\it controlled transform} of the  marked ideal $(\cI,a)$ under the admissible Kummer blowing up $\sigma$ we mean the  ideal
$$\sigma^c(\cI,a):=\cI_E^{-a}(\cI\cO_{X'}).$$ The \emph{marked controlled transform} is $(\sigma^c(\cI),a) :=(\sigma^c(\cI,a),a)$. The logarithmic structure on $X'$ is enhanced by the exceptional  divisor  $E$, see \cite[\S4.1.1]{ATW-destackification}.

We extend this definition to an admissible Kummer sequence
\begin{equation} \label{Eq:admissible-sequence}\xymatrix{X'=: X_n \ar[r]^{\sigma_n} &X_{n-1} \ar[r]^{\sigma_{n-1}} & \dots \ar[r]^{\sigma_2} &X_{1} \ar[r]^(.4){\sigma_{1}} & X_0:=  X.}
\end{equation}
with centers $\cJ_i$ on $X_i$ admissible for the controlled transform $\cI_i = \sigma_i^c(\cI_{i-1},a),$ see Definition \ref{Def:admissible-sequence}. The logarithmic structure on $X_i$ is enhanced by the exceptional  divisor  $E_i$ of $\sigma_i$.

Writing $\sigma: X' \to X$ for the composite map, we denote
$\sigma^c(\cI,a) := \cI_n$. {
This can be unwound as follows. Write $\cI_E = \prod \cI_{E_i} \cO_{X'}$, and then in general $$\sigma^c(\cI,a) = \cI_E^{-a} (\cI \cO_{X'}).$$}
\end{definition}

\begin{example}
Assume that $H\into X$ is a closed toroidal suborbifold and $\cI_H$ is its ideal. Then a sequence is $H$-admissible if and only if it is $(\cI_H,1)$-admissible. This notion will be essential later in the case when $H$ is a maximal contact hypersurface.
\end{example}

\subsubsection{Derivatives and transformed ideals}
Note that marked ideals $(\cI,a)$ with  weight $a$ are transformed by the rule $\sigma^c(\cI,a)=\cI_E^{-a}(\cI\cO_{X'})$. From this perspective the derivatives of order $i$ transform as ``marked objects" of weight $-i$, see Notation \ref{Not:controlled}. This is not so surprising since the dual objects, the cotangent sheaves, are transformed as marked ideals with weight 1.

\begin{lemma}\label{44}
Keep assumptions of Lemma~\ref{3}, and let $i$ be an integer with $0\le i\le a$. Then $$\sigma^c\left(\cD^{(\leq i)}_{X}(\cI),a-i\right)\ \subseteq\ \cD^{(\leq i)}_{X'}(\sigma^c(\cI,a)).$$
\end{lemma}
\begin{proof}
Recall that $\cI_E$ is monomial. Using {Corollary \ref{4}} and Lemma \ref{commute} we have
\begin{align*}\cI_{E}^a\sigma^c\left(\cD_X^{(\leq i)}(\cI),a-i\right) &= \cI_E^{i}\left(\cD_X^{(\leq i)}(\cI)\cO_{X'}\right)\\
&\subseteq \cD_{X'}^{(\leq i)}(\cI\cO_{X'})= \cD_{X'}^{(\leq i)}(\cI_E^{a}\sigma^c(\cI,a))\\ &= \cI_E^{a} \cD_{X'}^{(\leq i)}(\sigma^c(\cI,a))
\end{align*}
So  $\sigma^c(\cD_X^{(\leq i)}(\cI),a-i)\subseteq \cD_{X'}^{(\leq i)}(\sigma^c(\cI,a))$, as needed
\end{proof}

In the opposite direction, one can at least embed $\sigma^c(\cD_X)^{(\leq i)}(\sigma^c(\cI,a))$ into a \emph{sum} of transforms of derivations of $\cI$.

\begin{lemma}\label{Flem}
Keep assumptions of Lemma~\ref{3}, and let $\cF\subseteq\cD^1_X$ be an $\cO_X$-submodule and $i$ an integer with $0\le i\le a-1$. Then $$\sigma^c(\cF)^{(\le i)}(\sigma^c(\cI,a))\ \ \subseteq\ \ \sum_{j=0}^i\cI_E^{j-a}\left(\cF^{(\le j)}(\cI)\cO_{X'}\right) =: \sum_{j=0}^i \sigma^c\left(\cF^{(\le j)}(\cI),a-j\right).$$
\end{lemma}
Notice that the right hand side sums usual ideals without markings (not to be confused with the weighted sum in Section \ref{Sec:addmult} below). Note also that despite the fact that $\cF^{(\le j)}(\cI)\subseteq\cF^{(\le i)}(\cI)$, we have to take into account all terms because $\cF^{(\le j)}(\cI)$ is transformed with larger weight.
\begin{proof}
Set $\cF_\sigma=\sigma^c(\cF)$ for shortness. We can work \'etale-locally on $X$, so assume that $X=\Spec(A)$ and $\cF$ is generated by global derivations. Furthermore, it suffices to check the inclusion \'etale-locally on $X'$, so we can assume that $\cJ\cO_{X'}=(y)$ is principal and then $\cF_\sigma(\sigma^c(\cI,a))$ is generated by elements of the form $y\partial(h/y^a)$ with $h\in\cI$ and $\partial\in\Gamma(\cF)$. Since
$$y\partial(h/y^a)\ \ =\ \ \partial(h)/y^{a-1}-a\partial(y)\cdot h/y^a\ \ \in\ \  \sigma^c\left(\cF(\cI),a-1\right)\ +\ \sigma^c(\cI,a),$$
we obtain the claim for $i=1$. The general case follows by induction on $i$ since
\begin{align*}
\cF_\sigma^{(\le i+1)}(\sigma^c(\cI,a))\ \  & = \ \    \cF_\sigma^{(\le 1)}\left(\cF_\sigma^{(\le i)}(\sigma^c(\cI,a))\right)
\\ &  \subseteq\ \
\cF_\sigma^{(\le 1)} \left(\sum_{j=0}^i \sigma^c\left(\cF^{(\le j)}(\cI),a-j\right)\right)
\\ &  \subseteq \ \ \sum_{j=0}^i \sigma^c\left(\cF^{(\le j)}(\cI),a-j\right)\ \ + \ \   \sum_{j=0}^i\sigma^c\left(\cF^{(\le j+1)},a-j-1\right)
\\ & =\ \ \sum_{j=0}^{i+1} \sigma^c\left(\cF^{(\le j)}(\cI),a-j\right).
\end{align*}
\end{proof}

\subsubsection{The order does not increase!}
\begin{proposition} \label{5}
Keep assumptions of Lemma~\ref{3} and assume that $(\cI,a)$ is of maximal order. Then
\begin{enumerate}
\item $\cD_{X'}^{(\leq a)}(\sigma^c(\cI,a))=(1)$.

\item $(\sigma^c(\cI),a)$ is a  marked ideal of maximal order.

\item For any $p'\in|X'|$, and $p=\sigma(p')\in|X|$, we have $\logord_{p'}(\sigma^c(\cI,a))\leq \logord_p(\cI)$.
\end{enumerate}
\end{proposition}
\begin{proof} {By assumption
$(\cD_X^{(\leq a)}(\cI))\cO_{X'}=(1).$ By Lemma \ref{44} we have $$(1) =  (\cD_X^{(\leq a)}(\cI))\cO_{X'} = \sigma^c(\cD_X^{(\leq a)}(\cI),0) \subseteq \cD^{(\leq a)}(\sigma^c(\cI,a)).$$} This proves (1), which implies (2). For (3) we may assume that $p\in C\subseteq \supp(\cI,a)$.
By (1) and Lemma~\ref{11new}, $$\logord_{p'}(\sigma^c(\cI,a))\leq a=
\logord_p(\cI).$$
\end{proof}


Part (2) of Proposition \ref{5}  allows us to define:

\begin{definition}\label{Def:order-reduction}
By an {\it order reduction} of a marked ideal $(\cI,a)$ of maximal order we mean an  $(\cI,a)$-admissible Kummer sequence \eqref{Eq:admissible-sequence}  such that $$\supp(\sigma^c(\cI),a)=\emptyset .$$
\end{definition}

Note that $\supp(\sigma^c(\cI),a)$ is empty if and only if there are no non-trivial $(\sigma^c(\cI),a)$-admissible Kummer blowings up. Thus an order reduction of $(\cI,a)$ is precisely a {\em maximal} $(\cI,a)$-admissible Kummer sequence.

If one uses the same centers but eliminates a smaller power of the exceptional divisor then the obtained ideal will still be balanced:

\begin{corollary}\label{admissibility reduced}
Let $X$ be a toroidal orbifold and let $(\cI,a)$ be a balanced ideal, not necessarily of maximal order. Let $b$ be the order of the clean factor $\cI^{cln}$ of $\cI$, and let $(\cI^{cln},b)$ be the associated ideal of maximal order. Assume $b\geq a$, and let $\sigma: X'\to X$ be the $(\cI^{cln},b)$-admissible Kummer blowing up at center $\cJ$. Then
\begin{enumerate}
\item $\cJ$ is admissible for $(\cI,a)$
\item $\sigma^c(\cI,a)$ is balanced.
\end{enumerate}
\end{corollary}
\begin{proof}
(1) $(\cJ^a)^\nor\supseteq(\cJ^b)^\nor\supseteq \cI^{cln}\supseteq \cI$.

(2) We have that
\begin{align*} \sigma^c(\cI,a) &=\sigma^c(\cI^{cln}\cdot\cM(\cI),a)=\cI_E^{-a}(\cI^{cln}\cO_{X'})\cdot (\cM(\cI)\cO_{X'})\\&=\left(\cI_E^{b-a} (\cM(\cI)\cO_{X'})\right)\  \sigma^c(\cI^{cln},b).\end{align*}
The ideal $\sigma^c(\cI^{cln},b)$ is clean by Proposition~\ref{5}(2), and $\cM(\cI)\cO_{X'}$ is invertible and monomial by the assumption, hence the ideal $\sigma^c(\cI,a)$ is balanced.
\end{proof}

\begin{corollary}\label{derivative}
Let $(\cI,a)$ be a marked ideal of maximal order on a toroidal orbifold $X$, let $\sigma: X'\to X$ be the $(\cI,a)$-admissible Kummer blowing up at center $\cJ$, and let $0<i<a$. Then
\begin{enumerate}
\item $\cJ$ is admissible for {$(\cD_X^{(\leq i)}(\cI),a-i)$ and}
\item {$\sigma^c(\cD_X^{(\leq i)}(\cI),a-i)$ is a marked ideal of maximal order.}
\end{enumerate}
\end{corollary}
\begin{proof}
(1) 
Since $\cI\subseteq(\cJ^a)^\nor$, we obtain by Lemma~\ref{Jnorlem}(2) that $$\cD_X^{(\leq i)}(\cI)\subseteq \cD_X^{(\leq i)}((\cJ^a)^\nor)\subseteq(\cJ^{a-i})^\nor.$$

(2) By Lemma \ref{Lem:supports}(3), $(\cD_X^{(\leq i)}(\cI),a-i)$  is of maximal order, so $\sigma^c(\cD_X^{(\leq i)}(\cI),a-i)$ is of maximal order by Proposition~\ref{5}(2).
\end{proof}

We have seen that
$\sigma^c(\cD_X^{(\leq i)}(\cI),a-i) \subseteq \cD_{X'}^{(\leq i)}(\sigma^c(\cI,a)),$ so property (2) above is perhaps expected.

\subsection{Monomial Kummer blowings up and associated blowings up}\label{Sec:monomial-saturation}
%
%
%

\begin{definition}
Let $(\cI,a)$ be a marked ideal on a toroidal orbifold $X$. The {\it Kummer monomial blowing up associated to $(\cI,a)$} is the Kummer blowing up $\sigma_a:X'_a \to X$ with center $\cM(\cI)^{1/a}$.
\end{definition}
Note that when $a=1$ the modification $X'_1\to X$ is simply the \emph{normalized} blowing up of $\cM(\cI)$. In particular $\sigma_1:X'_1 \to X$ is logarithmically \'etale by Proposition \ref{Prop:monomial-blowup}. There is also an evident morphism $X'_a \to X'_1$ which is easily seen to be also logarithmically \'etale - it is a normalized root stack of a toroidal divisor. Therefore for any $a$ we have that $\sigma_a:X'_a \to X$ is logarithmically \'etale.

\begin{proposition}[{Compare \cite[Proposition 20]{Kollar-toroidal}}]\label{Absolute}\label{mixed}
Let $\sigma_a: X'_a\to X$ be the Kummer monomial blowing up associated to  $(\cI,a)$ on a toroidal orbifold $X$. Then

(1) The Kummer blowing up $\sigma_a$  is $(\cI,a)$-admissible, and the ideal $\sigma_a^c(\cI,a)$ is clean.

(2) If $Y\to X$ is logarithmically smooth and $Y'_a\to Y$ is the Kummer blowing up associated to $(\cI\cO_Y,a)$, then $Y'_a=Y\times_XX'_a$, the logarithmic pullback taken in the fs category.
\end{proposition}
\begin{proof} Write $\cJ_a := \cM(\cI)^{1/a}$ and $\cI_{E_a} = \cJ_a\cO_{X'_a}$.

(1)  The admissibility follows since $\cI\subseteq\cM(\cI)\subseteq\cJ = \cJ_a^a$.
But $\sigma_a$ is logarithmically \'etale, so $\cM(\cI\cO_{X'_a})=\cM(\cI) \cO_{X'_a}= \cI_E^a$ by Theorem \ref{Th:monomial-part}(3). Thus $\cM(\sigma_a^c(\cI,a)) = (1)$ as needed.


(2) Again by Theorem \ref{Th:monomial-part}(3) we have $\cM(\cI\cO_Y) = \cM(\cI)\cO_Y$. The result now follows from Theorem  \ref{Kummertheorem}(4).
\end{proof}

%
%
%
%

\section{Relations and operations with marked ideals}
\addtocontents{toc}{\noindent We adapt to our setup the notions of equivalence of marked ideals, sums and products of marked ideals, and homogenization of marked ideals.}

\subsection{Equivalence and domination of marked ideals}
In analogy with \cite{Wlodarczyk}, let us introduce the following domination and equivalence relations for marked ideals:

\begin{definition}
Let $(\cI,a)$ and $(\cJ,b)$ be marked ideals on a toroidal orbifold $X$

(1) We say that $(\cI,a)$ is {\em dominated} by $(\cJ,b)$ and write $$(\cI,a)\preccurlyeq (\cJ,b)$$ if any  $(\cJ,b)$-admissible Kummer sequence is also an $(\cI,a)$-admissible Kummer sequence.

(2) If both {$(\cI,a)\preccurlyeq (\cJ,b)$ and $(\cJ,b)\preccurlyeq (\cI,a)$}  then we say that the marked ideals are {\em equivalent} and write $(\cI,a)\approx(\cJ,b)$
\end{definition}

\begin{remark}
We only consider relations with respect to what is often called test sequences (of blowings up). {Hironaka was} first to introduce this sort of relations, though he considered a wider class of sequences, see \cite[Definition~1.3]{infinitely_near}. Bierstone and Milman also consider an analogous equivalence relation, though they extend the class of test morphisms even further, see \cite[\S1.2]{Bierstone-Milman-funct}.
\end{remark}

\begin{remark}
(1) The simplest examples of domination are as follows: (a) if $a\ge b$ then $(\cI,a)\preccurlyeq(\cI,b)$, (b) if $\cI\subseteq \cJ$ then $({\cI},a)\preccurlyeq ({\cJ},a)$.

(2) If $({\cI},a)\preccurlyeq ({\cJ},b)$ then automatically $\supp(\cI,a)\supseteq\supp(\cJ,b)$. More generally, for any $(\cJ,b)$-admissible Kummer sequence, we have that $\supp(\cI_i,a)\supseteq \supp(\cJ_i,b)$ for the sequence $(\cI_i, a), (\cJ_i, b)$ of controlled transforms (see Definition~\ref{Def:controlled-transform}).
\end{remark}

Here is a slightly less obvious example.

\begin{lemma}\label{Lem:scaling}
If $(\cI,a)$ is a marked ideal on a toroidal orbifold $X$ and $k>0$ is an integer then

(1) $({\cI},a)\approx ({\cI}^k,ka)$.

{(2) $\sigma^c({\cI}^k,ka)=\sigma^c({\cI},a)^k$ for any $(\cI,a)$-admissible Kummer sequence $\sigma\:X'\to X$.}
\end{lemma}
\begin{proof}
{To prove (1) it suffices to show that a Kummer sequence $\sigma\:X'\to X$ is $(\cI,a)$-admissible if and only if it is $(\cI^k,ka)$-admissible. Using induction on the length $n$ of $\sigma$ we can assume that $\sigma=\tau\circ\sigma_n$, where $\sigma_n$ is a Kummer blowing up with center $\cJ$ and the claims hold for $\tau$.}

Assume that $\sigma$ is admissible with respect to either $(\cI,a)$ or $(\cI^k,ka)$. By induction, $\tau$ is admissible with respect to both marked ideals, $\tau^c({\cI}^k,ka)=\tau^c({\cI},a)^k$, and $\tau^c({\cI},a)\subseteq(\cJ^a)^\nor$ if and only if $\tau^c({\cI}^k,ka)\subseteq(\cJ^{ka})^\nor$. Therefore, the whole $\sigma$ is admissible with respect to both $(\cI,a)$ and $(\cI^k,ka)$, and
\begin{align*}
\sigma^c({\cI}^k,ka)=\cI_E^{-ka}\sigma_n^*(\tau^c({\cI}^k,ka))=\cI_E^{-ka}\sigma_n^*(\tau^c({\cI},a)^k)=\\ \cI_E^{-ka}\sigma_n^*(\tau^c({\cI},a))^k=\sigma^c({\cI},a)^k.
\end{align*}
\end{proof}

The following example of domination will be very useful.

\begin{lemma}\label{diff}
If $(\cI,a)$ is a marked ideal of maximal order on a toroidal orbifold $X$ and $i\le a$, then

(1) {$\cD_X^{(\leq i)}(\cI,a)\preccurlyeq  (\cI,a)$.}

(2) $\sigma^c(\cD_X^{(\leq i)}(\cI,a))\subseteq \cD_{X'}^{(\leq i)}(\sigma^c(\cI,a))$ for any $(\cI,a)$-admissible Kummer sequence $\sigma\:X'\to X$.
\end{lemma}
\begin{proof}
To prove (1) it suffices to show that $\sigma$ is {$\cD_X^{(\leq i)}(\cI,a)$}-admissible. We will prove this and (2) by induction on the length $n$ of $\sigma$. So, we can assume that $\sigma=\tau\circ\sigma_n$, where $\sigma_n\:X'\to Y$ is a Kummer blowing up with center $\cJ$ and the claims hold for $\tau\:Y\to X$. Thus, $\tau$ is {$\cD_X^{(\leq i)}(\cI,a)$}-admissible and $\tau^c(\cD_X^{(\leq i)}(\cI,a))\subseteq \cD_{Y}^{(\leq i)}(\tau^c(\cI,a))$. Since $\tau^c(\cI,a)\subseteq(\cJ^a)^\nor$ we obtain by Lemma~\ref{Jnorlem}(2) that $$\tau^c(\cD_X^{(\leq i)}(\cI,a))\subseteq \cD_{Y}^{(\leq i)}(\tau^c(\cI,a))\subseteq(\cJ^{a-i})^\nor,$$ proving that the whole $\sigma$ is {$\cD_X^{(\leq i)}(\cI,a)$}-admissible. Finally,
$$\sigma^c(\cD_X^{(\leq i)}(\cI,a))\subseteq \sigma_n^c(\cD_{Y}^{(\leq i)}(\tau^c(\cI,a)))\subseteq\cD_{X'}^{(\leq i)}(\sigma_n^c\tau^c(\cI,a))=\cD_{X'}^{(\leq i)}(\sigma^c(\cI,a)),$$ where the second inclusion is due to {Lemma \ref{44}.}
\end{proof}

\begin{corollary}\label{Hcor}
Assume that $X$ is a toroidal orbifold, $(\cI,a)$ is a marked ideal of maximal order, and $H$ is a hypersurface of maximal contact. Then $(\cI_H,1)\preccurlyeq  (\cI,a)$, that is, any $(\cI,a)$-admissible Kummer sequence is $H$-admissible.
\end{corollary}
\begin{proof}
Since $\cI_H\subseteq\cT(\cI,a)$, Lemma \ref{diff} implies that $$(\cI_H,1)\preccurlyeq\cT(\cI,a)=\cD_X^{(\leq a-1)}(\cI,a)\preccurlyeq (\cI,a).$$
\end{proof}

\begin{lemma}\label{norlem}
If $X$ is a toroidal orbifold with a marked ideal $(\cI,a)$, then

(1) $(\cI,a)\approx(\cI^\nor,a)$. In particular, $\supp(\cI,a)=\supp(\cI^\nor,a)$.

(2) $\sigma^c(\cI^\nor,a)\subseteq\sigma^c(\cI,a)^\nor$ for any $(\cI,a)$-admissible Kummer sequence $\sigma\:X'\to X$.
\end{lemma}
\begin{proof}
First assume that $\sigma$ is the blowing up along a Kummer center $\cJ$. Then (1) is obvious since $\cI\subseteq(\cJ^a)^\nor$ if and only if $\cI^\nor\subseteq(\cJ^a)^\nor$

To check the inclusion in (2) we can work \'etale-locally on $X'$. Since Kummer blowings up are compatible with \'etale covers, we can can replace $X$ by an affine \'etale covering $\Spec(A)$. Set $I=\Gamma(\cI)\subseteq A$. Fix an \'etale covering $\Spec(B)\to X'$ such that the invertible ideal $\cJ B$ is principal with a generator $y$. Then $\sigma^c(\cI^\nor,a)B$ is generated by the elements $y^{-a}x$ with $x\in I^\nor$. Taking an integral equation on $x$ as in (\ref{inteq}), we obtain an integral equation on $y^{-a}x$ with coefficients $y^{-ai}a_i$. Now $y^{-ai}a_i\in\sigma^c(\cI^i,ai)B=\sigma^c(\cI,a)^iB$, and hence $y^{-a}x\in(\sigma^c(\cI,a)B)^\nor$.

Apply induction on the length of $\sigma = \sigma_1 \circ \tau$ with $\sigma_1$ an $(\cI,a)$-admissible Kummer sequence and $\tau$  the blowing up of center $\cJ_\tau$. The induction assumption (2) gives  $\sigma_1^c(\cI^\nor,a)\subseteq\sigma_1^c(\cI,a)^\nor \subseteq (\cJ_\tau^a)^\nor$, which implies (1) in general. Part (2) follows from
$$\sigma^c(\cI^\nor, a) = \tau^c(\sigma_1^c(\cI^\nor, a),a) \subseteq \tau^c(\sigma_1^c(\cI, a)^\nor,a) \subseteq \tau^c(\sigma_1^c(\cI, a),a)^\nor = \sigma^c(\cI,a)^\nor. $$
\end{proof}

\subsection{Addition and multiplication of marked ideals}\label{Sec:addmult}
Following \cite{Wlodarczyk, Bierstone-Milman-funct} define the following operations of addition and multiplication of marked ideals on a toroidal orbifold $X$:
\begin{enumerate}
\item $({\cI}_1,a_1)+\ldots+({\cI}_m,a_m)$
\begin{align*}
\quad &:= ({\cI}_1^{a_2\cdot\ldots\cdot a_m}+
{\cI}_2^{a_1a_3\cdot\ldots\cdot a_m}+\ldots+{\cI}_m^{a_1\ldots a_{k-1}} ,a_1a_2\ldots a_m).\end{align*}
\item $({\cI}_1,a_1)\cdot\ldots\cdot({\cI}_m,a_m):=({\cI}_1\cdot\ldots\cdot{\cI}_m,a_1+\ldots+a_m)$
\end{enumerate}

The following is immediate:

\begin{lemma}\label{Lem:operation-pullback}
Keep the above notation. Given a morphism $X' \to X$ we have $$(({\cI}_1,a_1)+\ldots+({\cI}_m,a_m))\cO_{X'}\ \  = \ \ ({\cI}_1,a_1)\cO_{X'} +\ldots+({\cI}_m,a_m)\cO_{X'}$$ and  $$(({\cI}_1,a_1)\cdot\ldots\cdot({\cI}_m,a_m))\cO_{X'} \ \ = \ \ ({\cI}_1,a_1)\cO_{X'}\cdot\ldots\cdot({\cI}_m,a_m)\cO_{X'}.$$
\end{lemma}

The following facts will not be used, but they clarify the definition.

\begin{remark}
(1) The sum is compatible with equivalence. So, although the definition is not associative, it is associative up to equivalence.

(2) The product is, obviously, associative, but it is not compatible with the equivalence: consider $\cI=(z,1)$ and $\cJ=(z^2,1)$. Then  $\cI^n\cJ=(z^{n+2},n+1)$ is not equivalent to $\cI\cJ=(z^{3},2)$.
\end{remark}

\begin{lemma}\label{Lem:sum-max-order}
Keep the above notation. If one of the marked ideals $(\cI_1,a_1),\ldots,(\cI_m,a_m)$ is of maximal order then $(\cI_1,a_1)+\ldots+({\cI}_m,a_m)$ is of maximal order.
\end{lemma}
\begin{proof}
{By Lemma \ref{Lem:scaling},} one of the ideals ${\cI}_1^{a_2\cdot\ldots\cdot a_m}, \ldots,
{\cI}_m^{a_1\ldots a_{k-1}}$ is clean of maximal order $a_1a_2\ldots a_m$. By Lemma \ref{Lem:supports2}  $(\cI_1,a_1)+\ldots+({\cI}_m,a_m)$ is of maximal order, as needed.
\end{proof}

\begin{lemma} \label{le: operations}
Keep the above notation and suppose that $(\cI_1,a_1),\ldots,(\cI_m,a_m)$ are of maximal order. Then
 \begin{enumerate}
\item {If  $({\cI}_i,a_i)\preccurlyeq ({\cI}_1,a_1)$} for $2\leq i\leq m$ then
$(\cI_1,a_1)+\ldots+({\cI}_m,a_m)\approx ({\cI}_1,a_1)$

\item { If  $({\cI}_i,a_i)\preccurlyeq ({\cI}_1,a_1)$ for $2\leq i\leq m$ then
$({\cI}_2,a_1)\cdot\ldots\cdot({\cI}_m,a_m)\preccurlyeq ({\cI}_1,a_1).$
}
\end{enumerate}
\end{lemma}
\begin{proof}

(1) Set $(\cI,a=a_1\dots a_m)=\sum_{i=1}^m(\cI_i,a_i)$. We should prove that if a Kummer sequence $\sigma\:X'\to X$  is either $(\cI_1,a_1)$-admissible or $(\cI,a)$-admissible then it is admissible with respect to both. We will prove this by induction together with the claim that in this case $\sigma^c(\cI,a)=\sum_{i=1}^m\sigma^c(\cI_i,a_i)$. So, assume that $\sigma=\tau\circ\sigma_n$, where $\sigma_n$ is the Kummer blowing up along $\cJ$ and the claim holds for $\tau$. Set $\cI'_i=\tau^c(\cI_i,a_i)$ and $\cI'=\tau^c(\cI,a)$, then $(\cI',a)=\sum_{i=1}^m(\cI'_i,a_i)$ by the assumption. Hence $\sigma_n$ is $(\cI',a)$-admissible if and only if $(\cI'_i)^{a/a_i}\subseteq(\cJ^a)^\nor$ for $1\le i\le m$ if and only if $\cI'_i\subseteq(\cJ^{a_i})^\nor$ for $1\le i\le m$ if and only if $\cI'_1\subseteq(\cJ^{a_1})^\nor$, where the latter equivalence holds because $({\cI}_i,a_i)\preccurlyeq ({\cI}_1,a_1)$ and hence $(\cI'_i,a_i)\preccurlyeq (\cI'_1,a_1)$ for $1\le i\le m$. By our assumption $\sigma_n$ is either $(\cI',a)$-admissible or $(\cI'_1,a_1)$-admissible, hence it is admissible with respect to both. Combining this with the fact that $\tau$ is admissible with respect to both $(\cI_1,a_1)$ and $(\cI,a)$, we obtain that $\sigma$ is admissible with respect to both marked ideals. It remains to use the simple fact that $\sigma_n^c(\sum_{i=1}^m(\cI'_i,a_i))=\sum_{i=1}^m\sigma_n^c(\cI'_i,a_i)$.

(2) This is proved in the same way, so we omit the details.
\end{proof}

\subsection{Homogenization}\label{Sec:homogenization}
Hypersurfaces of maximal contact exist only locally and are not unique. To obtain a global algorithm we adapt from \cite{Wlodarczyk} the concept of homogenization and a corresponding gluing lemma. This does not involve any essential modification. Let $({\cI},a)$ be a  marked ideal of maximal order, with $$\cT({\cI}):={\cD}^{(\le a-1)}_X(\cI).$$ By the corresponding {\it homogenized ideal} we mean the marked ideal
\begin{align*}
{\cH}&({\cI},a):=({\cH}({\cI}),a)\\&=({\cI}+{\cD_X^{(\le 1)}}(\cI)\cdot \cT(\cI)+\ldots+{\cD}^{(\le i)}_X(\cI)\cdot \cT({\cI})^i+ \ldots+{\cD}^{(\le a-1)}_X(\cI)\cdot \cT({\cI})^{a-1},a).
\end{align*}

\begin{lemma}\label{Lem:homogenization-equivalent}
Let $({\cI},a)$ be of maximal order. Then $({\cI},a)\approx ({\cH}({\cI}),a)$.
\end{lemma}
\begin{proof}
Follows from  Lemmas \ref{le: operations} and \ref{diff}.
\end{proof}

\begin{lemma}\label{Lem:homogenization-functorial} Let $({\cI},a)$ be of maximal order and $f:Y \to X$ logarithmically smooth. Then
 ${\cH}({\cI}\cO_Y) = \cH(\cI)\cO_Y$.
\end{lemma}
\begin{proof}
This follows from {Lemmas \ref{Lem:operation-pullback} and \ref{le: operations}} since $\cD_X^{(\leq i)}(\cI)\cO_Y=\cD_Y^{(\leq i)}(\cI\cO_Y)$, by Corollary~\ref{Cor:logsmooth}.
\end{proof}

The main point of the following lemma is that restrictions of the homogenized ideal onto hypersurfaces of maximal contact are isomorphic \'etale-locally.

\begin{lemma} \label{le: homo} ({Gluing Lemma}) Let $({\cI},a)$ be a marked  ideal of maximal order on a {toroidal variety} $X$, and let  $x,y\in\cT({\cI},a)$ be maximal contact elements at  $p\in\supp({\cI},a)$. Then there exist \'etale neighborhoods $\phi_{x},\phi_y: \overline{X}\to X$ of $p=\phi_x(\overline{p})=\phi_y(\overline{p}) \in X$,  where $\overline{p}\in \overline{X}$, and a marked ideal $(\overline{\cI},a)$ on $\overline{X}$, such that

\begin{enumerate}
\item  $\phi_{x}^{*}(\cH(\cI))\cO_{\overline{X}}=\phi_{y}^{*}(\cH(\cI))\cO_{\overline{X}}=\overline{\cI}$.
\item  $\phi_{x}^*(x)=\phi_{y}^*(y)\in\cT(\overline{\cI},a)$.
\item  For any $\overline{q}\in \supp(\overline{\cI},a)$, $\phi_x(\overline{q})=\phi_y(\overline{q})$.
\item  For any $({\cI},a)$-admissible Kummer sequence $X' \to X$  the  induced   modifications $\phi_x^*(X')$ and $\phi_y^*(X')$ of $\overline{X}$ coincide and define an $(\overline{\cI},a)$-admissible Kummer sequence  $\overline{X}$.
\end{enumerate}
\end{lemma}
\begin{proof} The proof is identical to that of \cite[Lemma 3.5.5]{Wlodarczyk}, taking logarithmic structures into account. We focus on parts (1)-(3), part (4) being longer and not requiring changes. The key point is that $\cH(\cI)$ is tuned to the Taylor expansion in terms of $h=(x-y)$, which is insensitive to the logarithmic structure.

Let $U\subseteq X$ be an open subset for which there exist  parameters $x_2,\ldots, x_n$ for the logarithmic stratum $s_p$ through $p$ which are transversal to $x$ and $y$ on $U$, as well as a sharp toric chart $U \to \AA_M:=\AA_{\ocM_p}$  for the logarithmic structure of $X$ at $p$.  In particular $x,x_2,\ldots, x_n$  and $y,x_2,\ldots, x_n$ form  two systems  of ordinary parameters on $U$. Let ${\bf A}^n$ be the affine space with coordinates $z_1,\ldots,z_n$.

We have \'etale morphisms $\phi_1,\phi_2: U\to {\bf A}^n\times \AA_M$ with
\begin{equation} {\phi}^*_{1}(z_1)=x, \quad  {\phi}^*_{1}(z_i)=x_i \quad \textrm{for } \ i>1 \quad\quad \mbox{and} \quad
{\phi}^*_{2}(z_1)=y, \quad {\phi}^*_{2}(x_i)=x_i\quad\mbox{for } \ i>1.\nonumber
\end{equation}
and sending $u\in M$ to $\cO_X$ via the given toric chart.

Consider the fiber product {$U\times_{{\bf A}^n\times \AA_M}U$} associated to  the morphisms $\phi_1$ and $\phi_2$. Let $\phi_x$, $\phi_y$ be the natural projections $\phi_x, \phi_y: U\times_{{\bf A}^n\times \AA_M}U\to U$, so that $\phi_1\phi_x=\phi_2\phi_y$. Define $\overline{X}$  to be the irreducible component of $U\times_{{\bf A}^n\times \AA_M}U$ whose images $\phi_x(U)$ and $\phi_y(U)$ contain $p$.
\begin{align*}
w_1&:=\phi_x^*(x)=(\phi_1\phi_x)^*(z_1)=(\phi_2\phi_y)^*(z_1)=\phi_y^*(y),\\
w_i&:=\phi_x^*(x_i)=\phi_y^*(x_i)\qquad  \textrm{for $i\geq 2$}.
\end{align*}

Then $w_1,\ldots,w_n$ form a  system of ordinary parameters on $\overline{X}$ for the relevant stratum $s'_{p'}$. The monoid structure $\overline{M}$ on $\overline{X}$  can be identified with $M$, and the {\'etale} morphisms $\phi_x,\phi_y: \overline{X}\to X$ send $x,x_2,\ldots, x_n$  and $y,x_2,\ldots, x_n$, respectively, to $w_1,\ldots,w_n$.

{(1)  Let $h:=y-x$. By the above the morphisms $\phi_1$ and $\phi_2$ coincide on $V(h)$, and thus $\phi_x$ and $\phi_y$ coincide on $\phi_x^{-1}(V(h))=\phi_y^{-1}(V(h))$. Since $h\in\cT(\cI)$, we get that $V(h)\supseteq V(\cT(\cI))=\supp(\cI,a)$ and the maps $\phi_x$ and $\phi_y$ coincide on $$\phi_x^{-1}(\supp(\cI,a))=\phi_y^{-1}(\supp(\cI,a)).$$

Furthermore, $\cT(\cI)^a\subseteq\cH(\cI)$, and hence the ideal $\cH(\cI)$ is trivial on the complement of $V(\cT(\cI))=\supp(\cI,a)$ in $X$. In particular, both $\phi_x^*(\cH(\cI))$ and $\phi_y^*(\cH(\cI))$ are equal to the structure sheaf on $\oX\setminus \phi_x^{-1}(\supp(\cI,a))$.

Thus, to prove (1) it suffices to establish equality of the stalks of $\phi_x^*(\cH(\cI))$ and $\phi_y^*(\cH(\cI))$ at a point $\overline{q}\in \overline{X}$ such that $q=\phi_x(\overline{q})\in\supp({\cI}, a)$. By the above, $\phi_y(\overline{q})=q$ too. Denote by $\widehat{\phi}_{y}$ and $\widehat{\phi}_{x}$ the induced morphisms of the completions $\widehat{\overline{X}}_{\overline{q}}\to {\widehat{X}}_{{q}}$, and set  $\widehat{\phi}_{xy}=\widehat{\phi}_{y}\widehat{\phi}^{-1}_{x}$.
Then
\begin{equation}
\widehat{\phi}_{xy}^*(x)=y, \quad \widehat{\phi}_{xy}^*(x_i)=x_i \quad\mbox{for} \ i>1, \quad \widehat{\phi}_{xy}^*(m) = m\ \mbox{for} \quad m\in M.\nonumber \end{equation}
Therefore
\begin{equation}
\widehat{\phi}_{xy}^*(f)=f(x+h,x_2,\ldots,x_n)= \left(f+\frac{\partial{f}}{\partial{x}}\cdot h+\ldots +\frac{1}{i!}
\frac{\partial^i{f}}{\partial{x^i}}\cdot h^i +\ldots\nonumber\right)(x,x_2\ldots,x_n)
\end{equation}
for any  $f\in \widehat{\cI}$, and the right hand side belongs to
\begin{equation}\widehat{\cI}+{\cD^{(\le 1)}_X}(\widehat{\cI})\cdot \widehat{\cT({\cI})}+\ldots +{\cD}^{(\leq i)}_X(\widehat{\cI})\cdot \widehat{\cT({\cI})}^i+ \ldots +{\cD}^{(\leq a-1)}_X(\widehat{\cI})\cdot \widehat{\cT({\cI})}^{(\leq a-1)}={\cH}\widehat{\cI}.\nonumber
\end{equation}
Hence $\widehat{\phi}_{xy}^*(\widehat{\cI})\subseteq {\cH}\widehat{\cI}$. Analogously, we have that
$$\widehat{\phi}_{xy}^*({\cD}^{(\leq i)}_X(\widehat{\cI}))\subseteq
{\cD}^{(\leq i)}_X(\widehat{\cI})+{\cD}^{{(\leq i+1)}}_X(\widehat{\cI})\cdot \cT({\cI})+ \ldots +
{\cD}^{(\leq a-1)}_X(\widehat{\cI})\cdot \widehat{\cT({\cI})}^{a-i-1}={\cH}\widehat{{\cD}^{(\leq i)}_X(I)},$$
In particular, $\widehat{\phi}_{xy}^*(\widehat{\cT({\cI})})\subseteq \cH(\widehat{\cT({\cI})})=\widehat{\cT({\cI})} $, which yields
$$\widehat{\phi}_{xy}^*({\cD}^{(\leq i)}_X(\widehat{\cI})\cdot \widehat{\cT({\cI})}^i )\subseteq {\cD}^{(\leq i)}_X(\widehat{\cI})\cdot \widehat{\cT({\cI})}^i+ \ldots +{\cD}^{(\leq a-1)}_X(\widehat{\cI})\cdot \widehat{\cT({\cI})}^{a-1}\subseteq {\cH}\widehat{\cI}.$$
So, $\widehat{\phi}_{xy}^*({\cH}\widehat{\cI})\subseteq {\cH}\widehat{\cI}$, and since the scheme is noetherian, we obtain that $${\cH}\widehat{\cI}=\widehat{\phi}_{xy}^*({\cH}\widehat{\cI})=(\widehat{\phi}^{-1}_{x})^*(\widehat{\phi}_{y}^*({\cH}\widehat{\cI})).$$
Therefore, $\widehat{\phi}_{x}^*({\cH}\widehat{\cI})=\widehat{\phi}_{y}^*({\cH}\widehat{\cI})$, and since completion of noetherian rings is faithfully flat we also obtain the equality of the uncompleted stalks ${\phi}_{x}^*(\cH{\cI})_{\overline{q}}={\phi}_{y}^*(\cH{\cI})_{\overline{q}}$. As we noted earlier, this implies the equality of sheaves ${\phi}_{x}^*(\cH({\cI}))={\phi}_{y}^*(\cH({\cI}))$. 

(2) Follows from the construction.

(3) This follows from (1) and the earlier proven fact that $\phi_x$ and $\phi_y$ coincide on $\phi_x^{-1}(\supp(\cI,a))$.}

(4) The rest of the proof is similar to that of \cite[Lemma 3.5.5(4)]{Wlodarczyk} or \cite[3.97]{Kollar} and is suppressed.

\end{proof}

\begin{example}
Without homogenization such an \'etale (and even formal) isomorphism may fail to exist. Consider the ideal $(xy)$ on $\Spec k[x,y]$. Any nonzero element $ax + by$ is a maximal contact, however the elements $x$ and $x+y$ can not be switched by an automorphism of $k\llbracket x,y\rrbracket$ that preserves $(xy)$. However, the homogenization is $\cH(xy) = (x,y)^2$ and the automorphism of $k[x,y]$ switching $x$ and $x+y$ preserves $\cH(xy)$.
\end{example}

\section{The coefficient ideal of a marked ideal}
\addtocontents{toc}{\noindent We adapt the notion of coefficient ideal, and study its behavior under admissible modifications on maximal contact hypersurfaces.}
\label{Sec:coefficient}

\subsection{The coefficient ideal}
Let $(\cI,a)$ be a marked  ideal. For any $\cO_X$-submodule $\cF\subseteq\cD_X^1$ we define the {\em $\cF$-coefficient ideal}
$$C_\cF(\cI,a):=\sum_{i=0}^{a-1} \left(\cF^{(\leq i)}(\cI), a-i\right)=\left(\sum_{i=0}^{a-1}\left(\cF^{(\leq i)}(\cI)\right)^{a!/(a-i)},a!\right).$$ The ideal $C_X(\cI,a):=C_{\cD^1_X}(\cI,a)$ is the
{\em coefficient ideal} as in \cite{Wlodarczyk}, see also \cite{Villamayor, Bierstone-Milman-funct}. It follows from  Lemma \ref{Lem:supports2} that if $(\cI,a)$ is of maximal order then $C_\cF(\cI,a)$ is of maximal order. Moreover

\begin{lemma}\label{Lem:C-equiv}
If $(\cI,a)$ is an ideal of maximal order on a toroidal orbifold $X$ and $\cF\subseteq\cD^1_X$ is a submodule then $(\cI,a)\approx C_\cF(\cI,a)$.
\end{lemma}
\begin{proof}
By Lemma \ref{diff} we have $(\cF^{(\le i)}(\cI),a-i)\preccurlyeq(\cD^{(\le i)}_X(\cI),a-i) \preccurlyeq (\cI,a)$, hence the claim follows from  Lemma \ref{le: operations}(1).
\end{proof}

\begin{lemma}\label{Lem:coefficient-functorial} If $f:Y \to X$ is a logarithmically smooth morphism of toroidal orbifolds and $({\cI},a)$ is a marked ideal of maximal order on $X$. Then $C_Y(\cI\cO_Y,a) = C_X(\cI,a)\cO_Y$.
\end{lemma}
\begin{proof}
Recall that $(\cD_X^{(\leq i)}\cI)\cO_Y=\cD_Y^{(\leq i)}(\cI\cO_Y)$, by Corollary~\ref{Cor:logsmooth}. The claim follows by Lemma \ref{Lem:operation-pullback}.
\end{proof}

\begin{proposition}\label{Fprop}
Assume that $X$ is a toroidal orbifold with a marked ideal $(\cI,a)$ of maximal order, $\cF\subseteq\cD^1_X$ is a submodule, and $\sigma\:X'\to X$ is an $(\cI,a)$-admissible Kummer sequence. Then
$$\sigma^c(C_X(\cI,a),a!)\subseteq C_{X'}(\sigma^c(\cI),a),$$
$$C_{\sigma^c(\cF)}(\sigma^c(\cI),a)^\nor\subseteq\sigma^c(C_\cF(\cI,a),a!)^\nor,$$
and all four ideals are of order at most $a!$.
\end{proposition}
\begin{proof}
Recall that $(\sigma^c(\cI),a)$ is of maximal order by Proposition~\ref{5}, hence $(C_{\cF'}(\sigma^c(\cI),a),a!)$ is of maximal order for any $\cF' \subset \cD^1_{X'}$. By Lemma~\ref{Lem:C-equiv}, $\sigma$ is $(C_\cF(\cI,a),a!)$-admissible, hence $(\sigma^c(C_\cF(\cI,a)),a!)$ is of maximal order too.  By Lemma~\ref{norlem}(1) the integral closure of an ideal preserves  order and equivalence class, so $C_{\sigma^c(\cF)}(\sigma^c(\cI),a)^\nor$ and $\sigma^c(C_\cF(\cI,a),a!)^\nor$ are of order at most $a!$.

Corollary \ref{derivative}(2) allows us to apply  Lemma \ref{44} inductively, giving the first inclusion. Let us prove the second inclusion.

Assume, first, that $\sigma$ is of length one. Set $\cF_\sigma=\sigma^c(\cF)$ and fix $i$ with $0\le i\le a-1$. It suffices to show that the ideal $J:=\left(\cF_\sigma^{(\le i)}(\sigma^c(\cI),a)\right)^{a!/(a-i)}$ is contained in the right hand side. By Lemma~\ref{Flem} $J \subset \left(\sum_{j=0}^i \sigma^c\left(\cF^{(\le j)}(\cI),a-j\right)\right)^{a!/(a-i)}$, so it suffices to show that
 $$\left(\sum_{j=0}^i \sigma^c\left(\cF^{(\le j)}(\cI),a-j\right)\right)^{a!/(a-i)}\subseteq\sigma^c(C_\cF(\cI,a),a!)^\nor.$$
The left hand side is the sum of ideals $\cI_{\mathbf n}=\prod_{j=0}^i\left(\sigma^c(\cF^{(\le j)}(\cI),a-j)\right)^{n_j}$, where $\sum_{j=0}^i n_j=a!/(a-i)$. So, the following chain of inclusions finishes the proof
\begin{align*}
\cI_{\mathbf n}&\subseteq\left(\sum_{j=0}^i\left(\sigma^c(\cF^{(\le j)}(\cI),a-j)\right)^{a!/(a-i)}\right)^\nor\\
&\subseteq \left(\sum_{j=0}^i\left(\sigma^c(\cF^{(\le j)}(\cI),a-j)\right)^{a!/(a-j)}\right)^\nor=
\sigma^c\left(\sum_{j=0}^i\left(\cF^{(\le j)}(\cI)\right)^{a!/(a-j)},a!\right)^\nor\\&\subseteq
\sigma^c(C_\cF(\cI,a),a!)^\nor.
\end{align*}

It remains to run induction on the length of $\sigma$, so assume that $\sigma=\pi\circ \tau$ and the claim is proved for $\tau$ and $\pi$. Then we have the following chain of inclusions

\begin{align*}
C_{\cF_\sigma}(\sigma^c(\cI),a)^\nor&\subseteq\tau^c\left(C_{\cF_\pi}(\pi^c(I),a),a!\right)^\nor\subseteq \tau^c(\pi^c(C_\cF(\cI,a),a!)^\nor,a!)^\nor\\&\subseteq
\tau^c(\pi^c(C_\cF(\cI,a),a!))^\nor=\sigma^c(C_\cF(\cI,a),a!)^\nor,
\end{align*}
where the first two are obtained by applying the claim to $\tau$ and $\pi$, and the third inclusion follows from Lemma~\ref{norlem}(2). This proves the induction step and completes the proof of the theorem.
\end{proof}

\subsection{Restriction to a maximal contact}
Our next task is to study restriction of coefficient ideals onto a maximal contact $H$.

\subsubsection{Restricting basic operations}
The following result describes compatibility of basic operations with restriction onto $H$.

\begin{lemma}\label{restrictlem}
Assume that $X$ is a toroidal orbifold with a marked ideal $(\cI,a)$ and a closed toroidal suborbifold $H\into X$.  Let $\sigma\:X'\to X$ an $H$-compatible  Kummer sequence and $\tau\:H'\to H$ its restriction on $H$. Then

(1) $\cI^\nor|_H\subseteq(\cI|_H)^\nor$.

(2) If $\sigma$ is $(\cI,a)$-admissible then $\tau$ is $(\cI|_H,a)$-admissible and $\sigma^c(\cI,a)|_H=\tau^c(\cI|_H,a)$.
\end{lemma}
\begin{proof}
The first claim is clear. If $\sigma$ is a single Kummer blowing up along $\cJ$ then $\cI\subseteq(\cJ^a)^\nor$ implies that $\cI|_H\subseteq((\cJ|_H)^a)^\nor$, and hence $\tau$ is $(\cI_H,a)$-admissible. In addition, the controlled transform consists of pulling back and dividing by a power of the pullback of $\cJ$, the operations that clearly commute with restriction onto $H$. The case of a general sequence follows by a straightforward induction on the length of $\sigma$.
\end{proof}

\subsubsection{Contracting derivations}
We say that a submodule $\cF\subseteq\cD_X^1$ is {\em $H$-contracting} if $\cF(\cI_H)=\cO_X$. For example, $\cD_X^1$ is $H$-contracting.
\'Etale locally where $H = V(x)$ it means that there is a section   $\partial$ of  $\cF$ behaving like $\partial/\partial x$.

\begin{lemma}\label{contractlem}
Let $X$ be a toroidal orbifold, $H\into X$ a closed toroidal suborbifold of pure codimension one, $\sigma\:X'\to X$ an $H$-admissible Kummer sequence, and $H'\to H$ the restriction of $\sigma$. If $\cF\subseteq\cD_X^1$ is $H$-contracting then $\sigma^c(\cF)$ is $H'$-contracting.
\end{lemma}
\begin{proof}
It suffices to prove this for a single blowing up with center $\cJ$. Moreover, the same argument as in the proof of Lemma~\ref{33} shows that our assertion can be checked Kummer-locally on $X$. So, we can assume that $X$ is a variety and $\cJ$ is an ordinary ideal. Locally at $p\in H$ choose coordinates so that $H=V(x_1)$ and $\cJ=(x_1\.x_r,m_1\.m_s)$. By our assumption there exists $\partial\in\cF_p$ such that $\partial(x_1)\in\cO_{X,p}^\times$. Choose a point $p'\in H'\cap\sigma^{-1}(p)$. Then $p'$ lies in a $y$-chart, where $y\in\{x_2\.x_r,m_1\.m_s\}$. Note that $H'=V(x_1/y)$ at $p'$ and $\partial'=y\partial$ is an element of $\sigma^c(\cF)_{p'}$. Since $x_1/y\in m_{p'}$ and $\partial(x_1)$ is a unit at $p'$, the element $\partial'(x_1/y)=\partial(x_1)-x_1\partial(y)/y$ is a unit at $p'$, concluding the proof.
\end{proof}

Contracting derivations can be used to test admissibility on $X$ after restriction to $H$:

\begin{lemma}\label{reslemma}
Let $X$ be a toroidal orbifold with a marked ideal $(\cI,a)$ of maximal order. Assume that $H$ is a maximal contact and $\cF\subseteq\cD_X^1$ is an $H$-contracting submodule. Let $\cJ$ be an $H$-admissible Kummer center, and assume $\cJ|_H$ is $\cF^{(\le i)}(\cI)|_H$-admissible for all $0\leq i \leq a-1$. Then $\cJ$ is $(\cI,a)$-admissible.
\end{lemma}
\begin{proof}
We can verify this inclusion \'etale-locally, so assume that $X$ is a toroidal variety, $p\in X$ is a point and $H=V(x)$, where $x$ is an ordinary parameter at $p$. Since $\cF$ is $H$-contracting, there exists a derivation $\partial\in\cF_p$ such that $u=\partial(x)\in\cO_X^\times$. It suffices to show that if $f\in\cO_{X,p}$ satisfies $\partial^i(f)|_H\in(\cJ_{H,p}^{a-i})^\nor$ for any $i$ with $0\le i\le a-1$, then $f\in(\cJ_p^a)^\nor$. At this stage it is convenient to pass to formal completions, so fix a section of $\hatcO_{X,p}\onto\hatcO_{H,p}$ obtaining a (non-canonical) isomorphism $\hatcO_{X,p}=\hatcO_{H,p}\llbracket x\rrbracket$. Representing $f$ as $\sum_{i=0}^\infty f_ix^i$ we have that $\frac{1}{i!u^i}\partial^i(f)|_H=f_i$. Thus $f_i\in(\hatcJ_{H,p}^{a-i})^\nor\subset(\hatcJ_p^{a-i})^\nor$. Since $\cJ_p$ is $H$-admissible we have  $x\in\cJ_p$,  so that $f\in(\hatcJ_p^a)^\nor$. By flatness of the completion it then follows that $f\in(\cJ_p^a)^\nor$.
\end{proof}

Putting things together we get the following characterization of admissibility:
\begin{proposition}\label{Hlem}
Let $X$ be a toroidal orbifold with a marked ideal $(\cI,a)$ of maximal order. Assume that $H$ is a maximal contact and $\cF\subseteq\cD_X^1$ is an $H$-contracting submodule. Then a Kummer center $\cJ$ is $(\cI,a)$-admissible if and only if $\cJ$ is $H$-admissible and $\cJ|_H$ is $(C_\cF(\cI,a)|_H,a!)$-admissible.
\end{proposition}
\begin{proof}
Set $\cJ_H=\cJ|_H$ and $C=C_\cF(\cI,a)$. If $\cJ$ is $(\cI,a)$-admissible then it is $(C,a!)$-admissible by Lemma~\ref{Lem:C-equiv} and $H$-admissible by Corollary~\ref{Hcor}. By Lemma~\ref{restrictlem}(2) $\cJ_H$ is $(C|_H,a!)$-admissible.

Conversely, assume that $\cJ$ is $H$-admissible and $\cJ_H$ is $(C|_H,a!)$-admissible. Then $(\cF^{(\le i)}(\cI))^{a!/(a-i)}|_H\subseteq(\cJ_H^{a!})^\nor$ for any $i$ with $0\le i\le a-1$. Therefore $\cF^{(\le i)}(\cI)|_H\subseteq(\cJ_H^{a-i})^\nor$, and the result follows by the previous lemma.
\end{proof}

\subsubsection{Lift of admissibility}

Here is our main result on maximal contacts.

\begin{theorem}\label{Hth}
Let $X$ be a toroidal orbifold with a marked ideal $(\cI,a)$ of maximal order. Assume that $H$ is a maximal contact, $\sigma\:X'\to X$ is an $H$-admissible Kummer sequence, and $\tau\:H'\to H$ denote the restriction of $\sigma$. Then the following conditions are equivalent:

(1) $\sigma$ is $(\cI,a)$-admissible,

(2) $\tau$ is $(C_X(\cI,a),a!)$-admissible,

(3) $\tau$ is $(C_\cF(\cI,a),a!)|_H$-admissible, where $\cF\subseteq\cD_X^1$ is $H$-contracting.
\end{theorem}

We will prove equivalence of (1) and (3), as this subsumes (2) by taking $\cF=\cD_X^1$. Set $C=C_\cF(\cI,a)$ and $C_H=C|_H$. We wish to extend Proposition \ref{Hlem} to admissible Kummer sequences. The hard part is to prove an extension of lemma \ref{reslemma}, namely  that admissibility lifts from $H$ to $X$, to sequences:

\begin{lemma}\label{reslemmaseq}
Let  $\cJ$ be a Kummer center on $X'$ such that $\cI_{H'}\subset\cJ$ and $\cJ_H:=\cJ|_{H'}$ is $(\tau^c(C_H),a!)$-admissible. Then $\cJ$ is $(\sigma^c(\cI),a)$-admissible.
\end{lemma}
\begin{proof}[Proof of Lemma] Set $C=C_\cF(\cI,a)$ and $C_H=C|_H$.
We claim that the following chain of inclusions holds:
$$C_{\sigma^c(\cF)}(\sigma^c(\cI),a)|_{H'}\subseteq\sigma^c(C,a!)^\nor|_{H'}\subseteq\tau^c(C_H,a!)^\nor\subseteq(\cJ_H^{a!})^\nor.$$
Indeed, the first inclusion is due to Proposition~\ref{Fprop}. The second inclusion is due to Lemma~\ref{restrictlem}. The third inclusion holds because $\tau^c(C_H,a!)\subseteq(\cJ_H^{a!})^\nor$ by the admissibility assumption.

Recall that $\sigma^c(\cF)$ is $H'$-contracting by Lemma~\ref{contractlem}. We have just proved that $\cJ_H$ is $(C_{\sigma^c(\cF)}(\sigma^c(\cI),a)|_{H'},a!)$-admissible. This together with Proposition~\ref{Hlem} implis that $\cJ$ is $(\sigma^c(\cI),a)$-admissible, concluding the proof.
\end{proof}

\begin{proof}[Proof of Theorem \ref{Hth}]
Assuming (1), Lemma~\ref{Lem:C-equiv} implies $\sigma$ is $(C,a!)$-admissible. By Lemma~\ref{restrictlem}(2) $\tau$ is $(C_H,a!)$-admissible, giving (3).

To prove that (3) implies (1) we use induction on the length $n$ of $\sigma$. The case $n=0$ is vacuously true, and the inductive step is provided by Lemma \ref{reslemmaseq}, as needed.
\end{proof}

Finally we reduce order reduction of marked ideals of maximal order to order reduction of marked ideals in a smaller dimension:

\begin{proof}[Proof of Theorem \ref{Th:reductionthm}]
Recall that $\sigma=i_*(\tau)$ is an $H$-admissible Kummer sequence whose restriction on $H$ coincides with $\tau$. Recall that $\tau$ is a maximal $(C_X(\cI,a)|_H,a!)$-admissible Kummer sequence. By Theorem~\ref{Hth} and Corollary~\ref{Hcor}, $\sigma$ is a maximal $(\cI,a)$-admissible Kummer sequence of $X$, and this means that $\sigma$ is an order reduction of $(\cI,a)$.
\end{proof}

\begin{remark}
In this proof we only used the equivalence of (1) and (2) in Theorem~\ref{Hth}, but Statement (3) is  convenient in the proof of that theorem. Also, small coefficient ideals may be useful. For example, one may want to take $\cF$ generated by a single $H$-contracting derivation $\partial$, see \cite[Exercises 4.4(1)]{Bierstone-Milman-funct}.
\end{remark}
\begin{remark}
Homogenization without taking some form of coefficient ideal looses essential information after restricting on $H$ and cannot be used for lifting purposes: even when $I$ is homogenized, $I|_H$-admissibility does not lift to $I$-admissibility, as the following example shows.
\end{remark}

\begin{example}\label{notenough}
Consider a homogenized ideal $$(\cI,3):=((x^3,xu^3,u^6), 3),$$ {where $x$ is an ordinary parameter and $u$ a monomial. The module of logarithmic derivations  is generated by $\frac{\partial}{\partial x}$ and $u\frac{\partial}{\partial u}$.  We have  $\cD^{(\leq 1)}(\cI)=((x^2,u^3),2)$ and $\cD^{(\leq 2)}(\cI)= ((x,u^3),1)$.

Consider the hypersurface of maximal contact $x=0$.  The center $(u^2)$ is admissible for $(\cI,3)|_{x=0} =((u^6),3)$ but its lift $(x,u^2)$ is not admissible for $(\cI,3)$.

To analyze the situation, we compute $C(\cI,3) = ((x^6,x^4u^3,x^2u^6, u^9),6)$. As we see below, order reduction for $(\cI,3)$  is equivalent to order reduction for the restriction $C(\cI,3)|_{x=0} =(u^9,6)$ of the coefficient ideal.

Indeed the Kummer center $(u^{3/2})$ is admissible for $(u^9,6)$ and its lift $(x,u^{3/2})$ is admissible for $(\cI,3)$. Computing $\sigma^c(\cI,3)$, it restricts to the unit ideal on the $x$-chart, and to the ideal $(y,u^{3/2})$, with logarithmic order $1$, where $y=x/u^{3/2}$, on the $u^{3/2}$-chart, thus achieving order reduction.}
\end{example}

\section{Non-embedded desingularization}
\addtocontents{toc}{\noindent We deduce functorial non-embedded desingularization of logarithmic varieties.}
\label{Sec:nonembedded}

Throughout this section by a {\em logarithmic variety} we mean a variety $X$ provided with a fine and saturated logarithmic structure $M_X$ which is defined already in the Zariski topology. The latter condition means that for any geometric point $\op\to X$ with image $p\in X$ we have that $\oM_p=\oM_\op$.

\subsection{Embeddings of logarithmic varieties}
Let $X$ be a logarithmic variety and $p\in X$ a closed point. We will use the definition of logarithmic stratification of $X$ from \cite[\S2.2.10]{AT1}. Let $u:M=\oM_p\to\cO_{X,p}$ be a sharp monoidal chart at $p$ and let $M^+=M\setminus\{0\}$ be the maximal ideal of $M$. Then, locally at $p$, the stratum $s_p$ through $p$ is given by the ideal $u(M^+)\cO_{X,p}$.

\subsubsection{Existence of local embedding} The following lemma shows how coordinates can be used to construct strict closed {embeddings of logarithmic varieties into toroidal varieties}.

\begin{lemma}\label{emblem}
Fix a monoidal chart $u\:M\to\cO_{X,p}$ and elements $t_1\.t_n\in\cO_{X,p}$ whose images generate the cotangent space to $s_p$ at $p$. Consider a neighborhood $X_0$ of $p$ on which $u(M)$ and $t_1\.t_n$ are global functions so that a morphism of logarithmic schemes $f=(u,t)\:X_0\to\Spec(k[M])\times\AA^n_k$ arises. Then
\begin{enumerate}
\item
$f$ is strict at $p$ as a morphism of logarithmic schemes.
\item
$f$ is unramified at $p$.
\item
On a small enough neighborhood $X_1$ of $p$ the morphism $f$ factors into a composition of a closed immersion $X_1\into Y$ and an \'etale morphism $Y\to \Spec(k[M])\times\AA^n_k$.
\end{enumerate}
\end{lemma}
\begin{proof}
Set $S=s_p$, $Z=\Spec(k[M])\times\AA^n_k$ and $q=f(p)$ for shortness.

(1) The sharp monoid at $q$ is, by the construction, $M$.

(2) Since $p$ is closed $k(p)$ is finite over $k(q)=k$, and since ${\rm char}(k)=0$ the extension $k(p)/k$ is separable. By definition, $\cO_{S,p}=\cO_{X,p}/u(M^+)$ and $t_1\.t_n$ generate the maximal ideal of $\cO_{S,p}$. Therefore $m_q\cO_{X,p}=m_p$, and $f$ is unramified at $p$ by \cite[Tag:02GF]{stacks}.

(3) This is a general fact about unramified morphisms, see \cite[Tag:0395]{stacks}.
\end{proof}

\subsubsection{The embedding dimension}
\begin{corollary}\label{embcor}
Let $X$ be a logarithmic variety, $p\in X$ a closed point, $r={\rm rk}(\oM_p)$ and $d$ the dimension of the cotangent space to the logarithmic stratum $s_p$ at $p$. Then $d+r$ is the minimal natural number $n$ such that there exists a neighborhood $X_0$ of $p$ and a strict closed immersion $X_0\into Y$, where $Y$ is a toroidal variety of dimension $n$.
\end{corollary}
\begin{proof}
By Lemma \ref{emblem}, there exists a neighborhood $X_0$ of $p$ admitting a strict closed immersion into an \'etale covering of $\Spec(k[\oM_p])\times\AA^d_k$. So, $n\le d+r$. Conversely, a strict closed immersion $i\:X_0\into Y$ into a toroidal variety, induces a closed immersion of the logarithmic strata $s_p\into s_q$ through $p$ and $q=i(p)$. Therefore, $\dim(s_q)\ge d$ and hence $\dim_q(Y)\ge r+d$.
\end{proof}

\subsubsection{Compatibility of embeddings with logarithmically smooth morphisms}

\begin{lemma}\label{emblem2}
Assume that $f\:X'\to X$ is a logarithmically smooth morphism and $i\:X\into Y$ is a strict closed immersion of logarithmic varieties. Then for any point $p\in X'$ there exist an \'etale neighborhood $X'_0$ with induced logarithmic structure, a strict closed immersion $X'_0\into Y'_0$, and a logarithmically smooth morphism $Y'_0\to Y$ such that $X'_0=X\times_YY'_0$.
\end{lemma}
\begin{proof}
The question is local on $X$ at the image $q=f(p)$. Hence we can assume that $X$ possesses a global chart $X\to\Spec(k[M])$ with $M=\oM_q$. By \cite[Theorem 3.5]{Kato-log} we can find $X'_0$ such that the morphism $X'_0\to X$ factors through a strict \'etale morphism $X'_0\to X\otimes_{k[M]}k[N]$, where $M\into N$ are fs monoids. Shrinking $Y$ around $i(q)$ we can assume that the chart of $X$ lifts to a global chart $Y\to\Spec(k[M])$ of $Y$. The morphism $X\otimes_{k[M]}k[N]\to X$ lifts to a logarithmically smooth morphism $Y\otimes_{k[M]}k[N]\to Y$. So, it remains to lift the strict \'etale morphism $X'_0\to X\otimes_{k[M]}k[N]$ to a strict \'etale morphism $Y'_0\to Y\otimes_{k[M]}k[N]$. But this is the problem of lifting a usual \'etale morphism from a closed subscheme, which is easily seen to be possible locally. (For example, one can use the explicit local description of \'etale morphisms from \cite[Tag:00UE]{stacks}.)
\end{proof}

\subsubsection{\'Etale equivalence of embeddings in varieties of the same dimension}
Our main result about embeddings of a logarithmic variety $X$ into toroidal varieties is that locally at $p$ such an embedding $i\:X\into Y$ is determined by the dimension of $Y$ at $i(p)$ up to an \'etale morphism of the target. In fact, this result is almost obvious
formally-locally since each such embedding corresponds to a homomorphism of completed local rings of the form $k(p)\llbracket\oM_p,t_1\.t_n\rrbracket\onto\hatcO_{X,p}$. Using Lemma~\ref{emblem} we will obtain a more refined \'etale-local version.

\begin{theorem}\label{embth}
Assume that $X$ is a logarithmic variety, $p\in X$ is a point, and $i\:X\into Y$, $i'\:X\into Y'$ are two strict closed immersions whose targets are irreducible toroidal varieties of the same dimension. Then there exists neighborhoods $X_0$, $Y_0$ and $Y'_0$ of $p$, $i(p)$ and $i'(p)$, and \'etale morphisms $f\:Z\to Y_0$ and $f'\:Z\to Y'_0$ with the same source, such that

(1) $i$ and $i'$ restrict to closed immersions $i_0\:X_0\into Y_0$ and $i'_0\:X_0\into Y'_0$,

(2) $f$ and $f'$ restrict to isomorphisms over $i(X_0)$ and $i'(X_0)$, respectively.
\end{theorem}
Loosely speaking, the theorem asserts that locally at $p$ both $i$ and $i'$ factor through a closed immersion $X\into Z$, where $Z$ is \'etale over both $Y$ and $Y'$.
\begin{proof}
Fix coordinates $t_1\.t_n\in\cO_{X,p}$ and $u\:M=\oM_p\into\cO_{X,p}$ at $p$. Set $q=i(p)$ and $q'=i'(p)$, and lift $t_1\.t_n$ to elements $x_1\.x_n\in\cO_{Y,q}$ and
$x'_1\.x'_n\in\cO_{Y',q'}$. Furthermore, complete the latter families to families $x_1\.x_m\in\cO_{Y,q}$ and $x'_1\.x'_m\in\cO_{Y',q'}$ of ordinary parameters such that the images of $x_i$ and $x'_i$ in $\cO_{X,p}$ vanish for $n<i\le m$. The two latter families are of the same size by the assumption on the dimensions.

Next, we claim that $u$ can be lifted to a monoidal chart $v\:M=\oM_q\into\cO_{Y,q}$. Indeed, it suffices to choose $m_1\.m_r\in M$ that form a basis of $M^\gp=\ZZ^r$, and lift $u(m_i)$ to elements $v_i\in M_q\subset\cO_{Y,q}$. Since $\oM_q=M$, for any element $\sum_{i=1}^rn_im_i$ with $n_i\in\ZZ$ the element $\prod_{i=1}^rv_i^{n_i}$ of $M_q^\gp$ actually lies in $M_q$, and hence there exists a unique homomorphism $v\:M\to\cO_{Y,q}$ sending $m_1\.m_r$ to $v_1\.v_r$. Clearly, $v$ is a lifting of $u$. In the same way, fix a lifting $v'\:M\to\cO_{Y',q'}$ of $u$.

Taking appropriate neighborhoods $X_0$, $Y_0$ and $Y'_0$ of $p$, $q$ and $q'$, respectively, we can assume that (1) is satisfied and all elements we have constructed are global functions. Consider morphisms $g\:Y_0\to T=\Spec(k[M])\times\AA^m_k$ and $g'\:Y'_0\to T$ induced by $(x_1\.x_n\.x_m,v)$ and $(x'_1\.x'_n\.x'_m,v')$, respectively. Since $g$ and $g'$ are \'etale at $q$ and $q'$, shrinking $X_0$, $Y_0$ and $Y'_0$ further, we can assume that $g$ and $g'$ are \'etale everywhere. By the construction, both $g$ and $g'$ restrict to the morphism $g\:X_0\to T$ induced by $(x_1\.x_n,0\.0,u)$. Set $Z=Y_0\times_TY'_0$ and note that the composition $X_0\into X_0\times_TX_0\into Z$ is a closed immersion $j\:X\into Z$ that lifts $i$ and $i_0$ to $Z$. Since the projections $f\:Z\to Y_0$ and $f'\:Z\to Y'_0$ are \'etale, $j(X_0)$ is a connected component of both $f^{-1}(i(X_0))$ and $f'^{-1}(i'(X_0))$, and we accomplish the proof by removing all other components of these preimages from $Z$.
\end{proof}

\subsection{Desingularization of logarithmic stacks}\label{Sec:nonembedded-proof}
\subsubsection{Resolution by toroidal stacks}
{Recall that a {\em modification} of logarithmic DM  stacks is a proper morphism inducing an isomorphism of dense open substacks. Here is the key example of a modification of such stacks:}

\begin{example}\label{Ex:stack-modification}
Assume that $X$ is a toroidal orbifold and $f\:X'\to X$ is a Kummer blowing up of a Kummer ideal $\cJ$. Assume that $Z\into X$ is a strict closed immersion and let $Z'$ be the {\em strict transform} of $Z$, i.e. the schematic closure of $Z\setminus V(\cJ)$ in $X'$ with the logarithmic structure induced from $X'$. If $V(\cJ)\cap Z$ is nowhere dense in $Z$ then $Z'\to Z$ is a modification.
\end{example}

\begin{proof}[Proof of Theorem \ref{Th:nonembedded}]\hfill

{\bf Choice of embedding of constant codimension.}
It suffices to construct such a functor $\cF$ \'etale-locally. Indeed, assume $X_0\to X$ is an \'etale covering and $X_1=X_0\times_XX_0$. If we construct a \emph{functorial} desingularizations $(X_i)_\res\to X_i$ for $i=0,1$, then they are automatically compatible with respect to both projections $X_1\to X_0$ and hence descend to a desingularization $X_\res\to X$. Functoriality can be checked \'etale-locally and hence is preserved under this descent.

We can now assume that $Z$ is a logarithmic variety, and then by Corollary~\ref{embcor} we can further assume that $Z$ possesses a strict closed immersion into a toroidal variety $X$. Moreover, since $Z$ is locally equidimensional such a closed embedding $i\:Z\into X$ can be constructed so that $Z$ is of a constant codimension $d$ in $X$. Indeed, if $Z_i$ and $X_j$ are the connected components of $Z$ and $X$, then it suffices to take an embedding of the form $Z=\coprod_i Z_i\into \coprod_i X_{j(i)}\times\AA^{n_i}$.

{\bf  Principalization of $I_Z$: blowing up $Z$ is synchronous.} Let $\cI$ be the ideal of $Z$ in $X$, let $X_{n+1}\to\dots\to X_0=X$ be the principalization sequence from Theorem~\ref{Th:principalization}, and let $Z_i\into X_i$ be the sequence of strict transforms of $Z$. We claim that all generic points of $Z$ are blown up for the first time at the same stage $i$. Indeed, this is the stage when we restrict the ideal onto an iterative $d$-th maximal contact $H^{(d)}\subset X_i$. Alternatively, it is characterized by the invariant $(1...1\infty)$, with $k_0 = d+1$ and having $d$ ones. Moreover, since each generic point of $Z_i$ is of codimension $d$, the generic points of $H^{(d)}$ are precisely the generic points of $Z_i$. Therefore ${\cI_i}|_{H^{(d)}}=0$; if $V$ is a neighborhood of a connected component of $H^{(d)}$ we necessarily have $\cI_i|_V\subseteq \cI_{H^{(d)}}|_V$. Since $H^{(d)}$ is an iterative maximal contact to $(\cI_i,1)$, we also have that $\cI_{H^{(d)}}|_V\subseteq\cI_i|_V$.
Thus $Z_i=H^{(d)}$ is a toroidal orbifold. In addition, the centers of $Z_{j+1}\to Z_j$ with $j<i$ do not contain the generic points of $Z_j$, and hence the morphisms $Z_{j+1}\to Z_j$ with $i<n$ are modifications.

In particular, the blowing up of $H^{(d)}$ principalizes $\cI_i$. This proves that $i=n$ and $Z_n=H^{(d)}$ is a toroidal orbifold.

{\bf  Independence of choices with the same embedding dimension.} Let us show that the construction of $\cF(Z)\:Z_n\to Z$ does not depend on choices. First, assume that $i'\:Z\into X'$ is another embedding dominating $i$ in the sense that there exists an \'etale morphism $g\:X'\to X$ which is an isomorphism over $i(Z)$ then $Z\into X'$ is given by $\cI'=g^*(\cI)$ and the principalization of $\cI'$ is the pullback of the principalization of $\cI$ by the functoriality assertion of Theorem~\ref{Th:principalization}. Therefore, the induced resolutions of $Z$ coincide. In general, by Theorem~\ref{embth} {the closed embedding $i$ and any other closed embedding} $i'\:Z\into X'$ of constant codimension $d$ are locally dominated by a third closed embedding. Therefore, the resolutions agree in this case too.

{\bf  Independence of choices with different embedding dimension.} It remains to compare embeddings $i$ and $i'$ of constant codimensions $d$ and $d'$. Since for a fixed codimension the resolution is independent of the embedding, it suffices to compare $i$ and an embedding of the form $i'\:Z\into X'=X\times\AA^n$, for which the Re-embedding Principle, Proposition \ref{Prop:product}, applies.

{\bf Functoriality.} Finally, we claim that the resolution is compatible with logarithmically smooth morphisms $h\:Z'\to Z$. Again, it suffices to check this locally on $Z$ and then by Lemma~\ref{emblem2} we can find extend $h$ to a logarithmically smooth morphism $X'\to X$ such that $Z'=Z\times_XX'$. Then $Z'$ is given by the ideal $\cI'=f^*(\cI)$, principalizations of $\cI$ and $\cI'$ are compatible by Theorem~\ref{Th:principalization} and hence the induced desingularizations of $Z$ and $Z'$ are compatible.
\end{proof}

\begin{example}
Let $Z=\Spec(\NN^2\to k[w])$, where both generators $e_1,e_2\in\NN^2$ are mapped to $w$. It is not logarithmically smooth at $w=0$, and the singularity is a sort of ``logarithmic embedded component'' at $w=0$. One can take the embedding $Z\into X=\Spec(\NN^2\to k[u,v])$, where $e_1\mapsto u$ and $e_2\mapsto v$ and $u,v$ are mapped to $w$. Thus $X=\AA^2_k$ with the standard boundary $V(uv)$ and $Z=V(u-v)$ is not transversal to the boundary - this is precisely Example \ref{Ex:monomial}.

The embedded resolution of $Z$ runs by a single blowing up $X'\to X$ at the origin, since the strict transform $Z'$ is transversal to the boundary of $X'$ and hence is logarithmically smooth. Note that $Z'=\Spec(\NN\to k[w])$ with the generator $e$ going to $w$, and the morphism $\NN^2\to\NN$ maps both $e_1$ and $e_2$ to $e$. Indeed, the pullbacks of both $u$ and $v$ vanish to first order on the intersection of $Z'$ with the boundary of $X'$.
\end{example}

\subsubsection{The logarithmically smooth locus is preserved}
First, \'etale functoriality means that it suffices to consider $Z$ logarithmically smooth. This means $Z \to \Spec k$ is logarithmically smooth. The resolution of $\Spec k$ is trivial, therefore by functoriality so is the resolution of $Z$.

\subsubsection{The role of equidimensionality}
Next, let us explain where the equidimensionality assumption came from. The argument in the proof of Theorem~\ref{Th:nonembedded} shows that if $Z\into X\into X'$ are strict closed immersions, $X$ and $X'$ are logarithmically smooth and $X$ is of constant codimension in $X'$ then both $Z\into X$ and $Z\into X'$ induce the same desingularization of $Z$. However, if the codimension of $X$ in $X'$ varies then the induced desingularizations may differ. For example, two components of $Z$ that were resolved at the same time in $X$ and have different codimension in $X'$ will be resolved at different times in $X'$ because it will take different number of passages to the maximal contact to get to them.

The equidimensionality assumption is essential for the method. Without this assumption there is no natural method of choosing codimension of embeddings of $Z$ so that it will be compatible with localizations. On the other hand, this assumption is not a real restriction for applications, since one can easily separate components by a non-embedded method, e.g. by passing to the normalization, or by passage to the disjoint union of the irreducible components.

\subsubsection{Embedded components} The algorithm of Theorem~\ref{Th:nonembedded} applies when $Z$ is generically logarithmically smooth but contains embedded components. In particular, this means that at some stage embedded components are killed by some Kummer blowings up of the ambient manifold.

\subsubsection{Logarithmic varieties which are generically logarithmically singular}
Our algorithm applies to varieties with generically non-smooth logarithmic structures, e.g. logarithmic varieties over the {hollow logarithmic point (le point point\'e) $\Spec(\NN\stackrel{0}\to k)$, which is not logarithmically smooth.}

The same argument as in the proof of Theorem~\ref{Th:nonembedded} shows that any logarithmic variety $Z$ can be locally embedded into a toroidal variety $X$, and if the reduction of $Z$ is equidimensional then we can choose all local embeddings in a compatible way by fixing a (large enough) constant codimension of the embedding. Let $U$ be the locus where $Z$ is logarithmically smooth and let $Z_i\into X_i$ be the strict transforms of $Z$ in the principalization sequence $X_n\to\dots\to X_0=X$ of $\cI_Z$. The principalization of $\cI_Z$ blows up all generic points of $U$ simultaneously, say by $f_{i+1}\:X_{i+1}\to X_i$, and then $Z_i\to Z$ is an isomorphism over $U$ and the closure $Z'$ of $U$ in $Z_i$ is the center of $f_{i+1}$. In particular, $U$ is logarithmically smooth and each $Z_j$ with $j>i$ is nowhere logarithmically smooth; for example, its components are not reduced or contain elements of $M$ mapped to zero. Skipping the blowing up $f_{i+1}$ we obtain a sequence of proper morphisms $$Z'\coprod Z_n\to\dots Z'\coprod Z_{i+1}=Z_i\to\dots\to Z_0=Z.$$ Certainly, $Z_n=\emptyset$ (though it may happen that $Z_j=\emptyset$ for $j<n$), hence we obtain a sequence of proper morphisms $\cF(Z)\:Z'\to\dots\to Z$ preserving $U$ and such that $Z'$ is logarithmically smooth and $U$ is dense in $Z'$. This sequence depends on $Z$ functorially with respect to logarithmically smooth morphisms.

\section{Destackification and desingularization of toroidal varieties}\label{Sec:destack}
\addtocontents{toc}{\noindent We deduce functorial desingularization of varieties.}

We conclude the paper with studying the following natural question: to what extent can one eliminate stacks from the formulations of main results. In particular, can one run the resolution procedure entirely in the category of varieties.

\subsection{Tools}
The results of Section~\ref{Sec:destack} will be obtained by combining Theorems~\ref{Th:principalization} and \ref{Th:nonembedded} with the following two tools.

\subsubsection{Destackification}\label{destacksec}
Our main tool in eliminating the stacky structure will be \cite[Theorem~1.1]{ATW-destackification}, which associates to any toroidal orbifold $X$ possesses a blowing up $\cD(X)\:X'\to X$ such that $X_\cs$ is a toroidal variety and the morphism $X'\to X'_\cs$ is toroidal. In addition, the blowings up $\cD(X)$ are compatible with logarithmically smooth inert morphisms $Y\to X$.

\subsubsection{Resolution of smooth varieties}\label{ressmoothsec}
One can also \emph{functorially} resolve the singularities of a toroidal variety or orbifold. One way to do this is using non-embedded resolution in characteristic 0, see  \cite[Theorem 5.10]{Niziol}, \cite{B-M-Toric}, \cite[Theorem~3.3.16]{Illusie-Temkin}, \cite[Theorem 9.4.5]{Gillam-Molcho}.  This in particular preserves the normal crossings locus. Another is a simple combinatorial method provided in \cite[Theorem 4.4.2]{ACMW} by combining barycentric subdivisions of fans with  the lattice reduction algorithm of \cite[Theorem 11*]{KKMS}, but has the disadvantage that it modifies the normal crossings locus. A simple functorial combinatorial algorithm, again working with fans, which does not modify the normal crossings locus is provided in \cite{Wlodarczuk-functorial-toroidal}.

\subsection{Preservation of inertness}
The destackification theorem applies only to inert morphisms. So, we should first study the question when inertness of a base change morphism $Y\to X$ survives the principalization or desingularization processes $Y_n\to Y$ and $X_n\to X$.

\subsubsection{A counter-example}
We start with an example showing what can go wrong. Consider the example $X=\Spec(k[x,u])$, $u^\NN\to k[x,u]$ and $\cI=(x^2,u)$ from \S\ref{examkumsec}. It is shown there that this ideal is principalized by the Kummer blowing up $X'\to X$ along $(x,u^{1/2})$. Next, consider the Kummer covering $Y=\Spec(k[x,v])$, $v=u^{1/2}$ of $X$. By the functoriality (or by a simple computation) the ideal $\cI_Y=\cI\cO_{Y}=(x^2,v^2)$ is principalized by the blowing up $Y'\to Y$ along $(x,v)$. Note that $Y\times_XX'$ is a non-saturated stack, but its saturation is the toroidal variety $Y'$. In particular, the functoriality relation $Y'=Y\times^{fs}_XX'$ holds in the saturated category.

The base change morphism $Y\to X$ we started with is inert, but the morphism $Y'\to X'$ is not inert because $Y'$ has the trivial inertia and $X'$ has a non-trivial one. This happened because the saturation operation, which reduces to normalization on the level of the underlying stack, can modify the inertia. No destackification algorithm, can apply to $X'$ and $Y'$ in a compatible way because the only natural destackification of $Y'$ is the identity, while $X'$ requires a non-trivial destackification. To conclude, on the one hand the saturation is critical to reconcile functoriality of principalization, but on the other hand it prevents one from keeping the strongest functoriality throughout the destackification step.

\subsubsection{Quasi-saturated morphisms}
The above discussion motivates the following definition: a morphism $Y\to X$ of logarithmic schemes or stacks is called {\em quasi-saturated} if for any base change $X'\to X$ the logarithmic scheme $(Y\times_XX')^\int$ is saturated. In other words, the fiber product in the integral category is already the fiber product in the saturated one. We do not need the following facts but include them for readers convenience.

\begin{remark}
A morphism $f$ is called {\em integral} if $Y\times_XX'$ is always integral. Kummer morphisms and simplest logarithmic blowings up are neither integral nor quasi-saturated, and we do not know if there are useful quasi-saturated but non-integral morphisms. A morphism is called {\em saturated} if it is integral and quasi-saturated. The properties of being integral, quasi-saturated or saturated are combinatorial in the sense that they only depend on the induced homomorphisms of monoids $\oM_x\to\oM_y$. Saturated morphisms were introduced by Kato and Tsuji and they form an important class. In particular, $\Spec(\ZZ[Q])\to\Spec(\ZZ[P])$ is saturated if and only if it is flat and has reduced fibers.
\end{remark}

\begin{lemma}\label{quasisatlem}
Assume that $X'\to X$ and $Y\to X$ are morphism of fine and saturated logarithmic stacks, and $Y'=Y\times^{fs}_XX'$ in the fs category. If $Y\to X$ is inert and quasi-saturated, then $Y'\to X'$ is inert and quasi-saturated too.
\end{lemma}
\begin{proof}
The morphism $Y'\to Y\times_XX'$ is a closed immersion, hence inert. Since inert morphisms are preserved by base changes, we obtain that $Y'\to X'$ is inert. Quasi-saturatedness is preserved by fine base changes as well.
\end{proof}

As an immediate corollary we obtain the following result:

\begin{corollary}\label{quasisatcor}
If the logarithmically smooth morphism $Y\to Z$ in Theorem~\ref{Th:nonembedded}(2) is inert and quasi-saturated then the morphism $Y_\res\to Z_\res$ is inert and quasi-saturated too. The same strengthening applies to Theorem~\ref{Th:principalization}.
\end{corollary}

\begin{remark}
In fact, Lemma \ref{quasisatlem} implies that any process compatible with logarithmically smooth morphisms in the fs category is also compatible with logarithmically smooth quasi-saturated inert morphisms.
\end{remark}

\subsection{Desingularization of toroidal varieties}

\subsubsection{From toroidal stacks to toroidal varieties} \label{Sec:to-varieties}
Combining the algorithm for desingularization of logarithmic stacks with the destackification algorithm of \cite[Theorem 1]{ATW-destackification}, one obtains a desingularization algorithm that outputs a toroidal orbifold with toroidal coarse space. Its functoriality is as follows:

\begin{proposition}\label{stacktovar}
Let $Z$ be as in Theorem~\ref{Th:nonembedded} and let $\cF_Z\circ\cD_{\cF(Z)}\:Z'\to Z_\res\to Z$ be the composition of the desingularization and destackification. Then the modification $Z'\to Z$ is compatible with logarithmically smooth, quasi-saturated and inert morphisms $h\:Y\to Z$.
\end{proposition}
\begin{proof}
By Corollary \ref{quasisatcor}, the morphism $h_\res\:Y_\res=(Y\times_ZZ_\res)^\int\to Z_\res$ is inert. Therefore $h'\:Y'\to Z'$ is the pullback of $h_\res$.
\end{proof}

\subsubsection{Eliminating stacks from the picture}
One can also wonder if the destackification technique allows to resolve logarithmic varieties by use of logarithmic varieties only. Here is our main result in this direction.

\begin{theorem}\label{vardesingth}
Let $Z$ be a locally equidimensional and generically logarithmically smooth logarithmic \emph{variety} over $k$. Then there is a sequence of modifications of logarithmic varieties $\cF'(Z)\:Z'_{res} \to  Z$  which is an  isomorphism over the logarithmically smooth locus of $Z$ and $Z'_{res}$ is logarithmically smooth. In addition, one can achieve that the process assigning $\cF'(Z)$ to $Z$ is functorial for logarithmically smooth quasi-saturated morphisms $Y\to Z$.
\end{theorem}
\begin{proof}
This is an immediate corollary of Proposition~\ref{stacktovar}: take $\cZ_n\das\cZ_0=Z$ to be the sequence $\cF_Z\circ\cD_{\cF(Z)}$ from the proposition, and set $Z_i=(\cZ_i)_\cs$. 

\end{proof}

\begin{remark}
(1) One can also prove that the modifications in the sequence $\cF'(Z)$ are blowings up, but this will be worked out elsewhere.

(2) The intermediate logarithmic varieties $Z_i$ in Theorem~\ref{vardesingth} are not toroidal. However, $Z_i$ are {\em locally toric} in the sense that locally one can make them toroidal by increasing the logarithmic structure. This follows form the fact that each $Z_i$ is the coarse space of a toroidal orbifold.

(3) The functoriality in the theorem is stronger than the classical one but weaker than in Theorem~\ref{Th:nonembedded}. So far, the only way we know to maintain functoriality with respect to all logarithmically smooth morphisms is to work within the category of toroidal orbifolds.
\end{remark}

\subsubsection{From toroidal varieties to smooth varieties}\label{Sec:to-smooth}
One can also resolve logarithmic varieties $Z$ by classical methods as follows. First, apply to $Z$ the monoidal desingularization algorithm of Gabber, see \cite[Theorem~3.3.16]{Illusie-Temkin}. It makes the monoids $\oM_p$ free and satisfies very strong functorial properties. After this step, the logarithmic variety $Z$ locally embeds into a toroidal varieties $(X,U)$ which is monoidally smooth in the sense that the monoids $\oM_p$ are free. For a toroidal variety $(X,U)$ monoidal smoothness is equivalent to either of the following two conditions: (1) $X$ is smooth, (2) $X$ is smooth and $X\setminus U $ is an snc divisor. Therefore, further resolution of $Z$ can be deduced from the usual embedded principalization of $(X,U,\cI_Z)$. Here is the summary:

\begin{proposition}
Let $Z$ be a locally equidimensional and generically logarithmically smooth logarithmic \emph{variety} over $k$. Then there is a sequence of blowings up $\cF''(Z)\:Z''_{res} \to  Z$, which is an isomorphism over the locus where $Z$ is smooth and logarithmically smooth, and such that $Z''_{res}$ is smooth and logarithmically smooth. The process assigning to $Z$ the modification $\cF''(Z)$ is functorial for smooth strict morphisms.
\end{proposition}

On the one hand, unlike Theorem~\ref{vardesingth} the resolution of $Z$ is smooth. On the other hand, the algorithm only satisfies the classical functoriality and one cannot hope for more than that because it has to modify the locus where $Z$ is logarithmically smooth but not smooth. Note also that an algorithm satisfying the same properties can be also obtained by combining $\cF'(Z)$ from Theorem~\ref{vardesingth} with any method mentioned in \S\ref{ressmoothsec}. In this case, it is only the last step that reduces functoriality to the classical level.

\bibliographystyle{amsalpha}
\bibliography{principalization}

\end{document}